\documentclass[a4paper,10pt]{amsart}

\usepackage{tablefootnote}

\usepackage{amsthm,amsmath,amssymb,amsfonts}
\usepackage[T1]{fontenc}
\usepackage[utf8]{inputenc}
\usepackage{verbatim}
\usepackage{marvosym}
\usepackage[shortlabels]{enumitem}
\usepackage{stackrel}
\usepackage{tabularx}
\usepackage{cancel}

\usepackage{graphicx}
\usepackage{eucal}
\usepackage{bbold}
\usepackage[svgnames]{xcolor}
\definecolor{blue(munsell)}{rgb}{0.0, 0.5, 0.69}
\usepackage{lipsum}
\usepackage{svg}
\usepackage{bbm}
\usepackage{caption}
\usepackage{enumitem}
\usepackage{mathtools}
\usepackage{epigraph}
\usepackage{color}
\usepackage{microtype}
\usepackage{relsize}
\usepackage{amsfonts}
\usepackage{adjustbox}
\usepackage{subfig}
\usepackage{framed}
\usepackage{hyperref}
\hypersetup{
    colorlinks = true,
    linkbordercolor = {red},
   	linkcolor ={blue},
	anchorcolor = {pink},
	citecolor =  {red},
	filecolor = {blue},
	menucolor = {blue},
	runcolor =  {blue},
	urlcolor = {blue},
}
\usepackage{cleveref}
\usepackage{trimclip}

\DeclareFontFamily{U}{min}{}
\DeclareFontShape{U}{min}{m}{n}{<-> udmj30}{}
\newcommand{\yo}{\!\text{\usefont{U}{min}{m}{n}\symbol{'210}}\!}

\usepackage{stmaryrd}
\usepackage{tikz}
\usepackage{tikz-cd}
\usepackage{quiver}

\theoremstyle{definition}
\newtheorem{thm}{Theorem}[subsection]
\newtheorem*{thm*}{Theorem}
\newtheorem{prop}[thm]{Proposition}
\newtheorem*{prop*}{Proposition}

\newtheorem{cor}[thm]{Corollary}
\newtheorem{defn}[thm]{Definition}

\newtheorem*{war*}{Warning}
\newtheorem{rem}[thm]{Remark}
\newtheorem{constr}[thm]{Construction}
\newtheorem{exa}[thm]{Example}
\newtheorem{conj}[thm]{Conjecture}
\newtheorem{notat}[thm]{Notation}

\newcommand{\cat}{\mathsf{cat}}
\newcommand{\Cat}{\mathsf{Cat}}

\newcommand{\cprod}{\mathsf{prod}}
\newcommand{\lex}{\mathsf{lex}}
\newcommand{\Lex}{\mathsf{Lex}}
\newcommand{\coh}{\mathsf{coh}}
\newcommand{\reg}{\mathsf{reg}}
\newcommand{\disj}{\mathsf{disj}}

\newcommand{\sites}{\mathsf{sites}}

\newcommand{\skt}{\mathsf{skt}}
\newcommand{\Skt}{\mathsf{Skt}}
\newcommand{\logoi}{\mathsf{log}}
\newcommand{\Log}{\mathsf{Log}}
\newcommand{\limskt}{\mathsf{skt}_l}
\newcommand{\limSkt}{\mathsf{Skt}_l}
\newcommand{\colimskt}{\mathsf{skt}_c}
\newcommand{\colimSkt}{\mathsf{Skt}_c}
\newcommand{\leftskt}{\mathsf{Lskt}}
\newcommand{\leftSkt}{\mathsf{LSkt}}

\newcommand{\leftnskt}{\mathsf{_{ln}skt}}
\newcommand{\leftnSkt}{\mathsf{_{ln}Skt}}

\newcommand{\roundSkt}{r\Skt}

\newcommand{\PhiEx}{\Phi\text{-}\mathsf{ex}}

\newcommand{\moritaM}{\mathsf{M}}

\newcommand{\linj}{\mathbf{LInj}} %cat of (strongly) left kan inj

\newcommand{\Top}{\mathsf{Top}}
\newcommand{\Set}{\mathsf{Set}}
\newcommand{\op}{\textit{op}}
\newcommand{\Psh}{\mathsf{Psh}}

\newcommand{\sketch}{\mathcal}
\newcommand{\sS}{\sketch{S}}

\newcommand{\sT}{\sketch{T}}
\newcommand{\sM}{\sketch{M}}

\newcommand{\sC}{\sketch{C}}
\newcommand{\sD}{\sketch{D}}
\newcommand{\sX}{\sketch{X}}
\newcommand{\sY}{\sketch{Y}}

\newcommand{\sL}{\sketch{L}}

\newcommand{\sI}{\sketch{I}}

\newcommand{\categ}{\mathbb}
\newcommand{\cS}{\categ{S}}
\newcommand{\cT}{\categ{T}}
\newcommand{\cC}{\categ{C}}
\newcommand{\cD}{\categ{D}}
\newcommand{\cX}{\categ{X}}
\newcommand{\cY}{\categ{Y}}
\newcommand{\cM}{\categ{M}}

\newcommand{\cI}{\categ{I}}
\newcommand{\cL}{\categ{L}}

\newcommand{\skL}{\mathsf{L}}
\newcommand{\skC}{\mathsf{C}}

\def\black{\color{black}}

\title{Sketches and classifying Logoi}

\author{Ivan Di Liberti}
\author{Gabriele Lobbia}

\address{
Ivan \textsc{Di Liberti} \newline
Department of Philosophy, Linguistics and Theory of Science\newline
University of Gothenburg\newline
Gothenburg, Sweden\newline
\href{mailto:diliberti.math@gmail.com}{\sf diliberti.math@gmail.com}
}

\address{
Gabriele \textsc{Lobbia}: \newline
Department of Mathematics and Statistics\newline
Masaryk University, Faculty of Sciences\newline
Kotl\'{a}\v{r}sk\'{a} 2, 611 37 Brno, Czech Republic\newline
\href{mailto:gabriele.lobbia@unibo.it}{\sf gabriele.lobbia@unibo.it}
}

\thanks{The first named author received funding from Knut and Alice Wallenberg Foundation, grant no.~2020.0199.}

\begin{document}

\maketitle

\begin{abstract}
Inspired by the theory of classifying topoi for geometric theories, we define rounded sketches and logoi and provide the notion of classifying logos for a rounded sketch. Rounded sketches can be used to axiomatise all the known fragments of infinitary first order logic in $\mathbf{L}_{\infty,\infty}$, in a spectrum ranging from weaker than finitary algebraic to stronger than $\lambda$-geometric for $\lambda$ a regular cardinal. We show that every rounded sketch has an associated classifying logos, having similar properties to the classifying topos of a geometric theory. This amounts to a Diaconescu-type result for rounded sketches and (Morita small) logoi, which generalises the one for classifying topoi.
 
 
\end{abstract}

   {
   \hypersetup{linkcolor=black}
   \tableofcontents
   }

% {\red TO DO LIST:
% \begin{enumerate}[1.]
%     % \item Sec 3 limare f morita implica equivalenza degli esponenziali e dei tensori. [Fatto per 2/3..Possibile che sia falso, se non ci viene in mente una genialata BALZELLA]
% %    \item Finire proof di Prop 4.0.5 (i cavolo di morfismi). 
%     \item {\yellow Finire Intro. [GAB]}
%     \item Rileggere Sec 4, looking for Morita equiv. [TUTTI]
%   \item {  \color{magenta}  4.2.5 $\hat{F_!}$ cocontinuo (ref a fractions, trovare qualcosa di meglio di Gabriel-Zisman). [IVAN] }
%      \item {\color{magenta} 4.3.2 parte in rosa. [IVAN]} 
%     \item {\yellow 4.5.4 cose 2-d. [GAB]}
%     \item {\color{magenta} 4.6.1 rifrasare. [IVAN]} 
%     \item Looking for TYPOS. [TUTTI]
%     \item ArXiv. 
%     \item Abstract CT. [GAB]
% \end{enumerate}
% }

\section*{Introduction}

\subsection*{Motivation}

% In one sentence, the aim of this paper is to introduce the notions of \textit{rounded sketch} and (classifying) \textit{logos}, under the motivational hallucination that in the context of infinitary logic,  logoi are to rounded sketches what topoi are to sites in geometric logic.

  In one sentence, the aim of this paper is to generalise the notion of classifying topos existing for geometric logic to infinitary logic by introducing the notions of \emph{rounded sketch} and \emph{logos}.  More precisely, we want to replicate the pattern that sites are presentations of \emph{geometric} theories and that the classifying topos gives a \emph{syntax independent} avatar of the theory. In a similar way our notion of \emph{rounded sketch} gives the presentation of any infinitary theory (including geometric ones) and the classifying logos its syntax independent presentation. These ideas are summarised in the table below and explained further throughout the paper, for instance in \Cref{sec:notions-skt}. 
 
 %[What was done? In geometric logic, the idea is that  classifying topos for a site is the construction which associated] 

\begin{table}[h!]\renewcommand{\arraystretch}{1.25}
%\caption{Logical Interpretation} 
\label{table:logic-intepr} 
\begin{tabular}{|c|c|c|}
\hline
\textbf{Logic Fragment} & \textbf{Presentation}   & \textbf{Morita Classifying Object} \\ \hline
Geometric      & Site           & Topos                     \\ \hline
Infinitary     & Rounded Sketch & Logos                     \\ \hline
\end{tabular}
\end{table}

The idea of using sketches to analyse infinitary logics is as old as the notion of sketch \cite{wells1994sketches,Makkaipare,adamek_rosicky_1994}, and the literature on the topic is extremely vast. The upshot of this approach  (and more specifically of \cite[D2]{Sketches} or \cite[2.F and 5]{adamek_rosicky_1994}) is that the objects in the $2$-category of sketches, $\mathsf{Skt}$ can be understood as theories in infinitary logics, where limit and colimit specifications are used to endow models with prescribed structures. The aim of this paper is to lay the foundations for applying 2-dimensional techniques in the study of sketches having in mind applications in infinitary logics\footnote{This may sound similar, but is only mildly related to the work in \cite{makkai1997generalized,makkai1997generalizedII,makkai1997generalizedIII}, which is much more in the spirit of \cite{di2022bi} and \cite{coraglia2021context}.}.

Of course, to some extent, this general programme was already developed in (and \textit{by}) topos theory under the restriction that the fragment of infinitary logic to study was \emph{geometric logic}. Let us list some prominent examples of this scientific flow, in order to frame precisely the kind of results we would like to be able to simulate and study in a more general environment.

\begin{itemize}
    \item In \cite{pitts1983amalgamation,pitts1983application}, pullback-stability of open surjections is used to derive the classical statement of Craig interpolation theorem.
    \item In \cite{moerdijk2000proper,zawadowski1995descent} and many other papers following this research agenda, descent and lax descent are used to derive completeness theorems for first-order logic, and results that revolve around such property.
    \item In \cite{barr1974toposes}, the construction of the Barr's cover was used to prove that if a statement in geometric logic is deducible from a geometric theory using classical logic and the axiom of choice, then it is also deducible from it in constructive mathematics.
\end{itemize}

In all these contexts, and most evidently in the case of \cite{pitts1983amalgamation}, many technical bits of these theorems are not performed at the level of sites (presentation of theories), and must be carried out on their classifying topos in order to be meaningfully stated, or technically accounted. This is mostly due to the obstacles that Morita theory poses to the categorical logician. The classifying topos of a theory offers a syntactically unbiased representation of the theory that contains its logical information without committing to any symbolic representation of it. This is the reason why this paper will focus on providing an appropriate notion of \textit{classifying logos} for a sketch.
%a correct

In \cite{anel2021topo}, Anel and Joyal crystallise the theory of classifying topoi for geometric logic in a choice of name: they call $\mathsf{Logoi}$ the $2$-category of topoi, lex cocontinuous functors and natural transformations. If there is a philosophical point to this paper, besides its technical content, it is that this specific choice has a reasonable historical bias to it, and a notion of \textit{logos} -- whatever that is -- should exist also for theories that are not geometric in nature.  One of the motivations to pursue this approach was found already in topos theory. In fact, in his work \cite{espindola2020infinitary}, Esp\'indola has shown that there is the need for a general theory of $\lambda$-topoi, which remains somewhat geometric in taste and yet handles completeness-like results in $\lambda$-ary geometric logic. Our notion of logos achieves three results in this sense:
\begin{itemize}
    \item  it offers a framework to study $\lambda$-topoi for every $\lambda$;
    \item it offers a framework in which topoi and infinitary topoi \textit{coexist}, so that they can interact together\footnote{Despite being very important, we shall not comment at the moment on the non-commutativity of the diagram above. The reader will soon get a sense of its importance in \Cref{doctrines}, where a very similar situation arises, and later we shall discuss in detail the relevance of this diagram in the last section of the paper.};
% https://q.uiver.app/#q=WzAsNixbMCwwLCJcXG1hdGhzZntUb3BvaX1ee1xcdGV4dHtvcH19Il0sWzIsMCwiLi4uIl0sWzMsMCwiXFxtYXRoc2Z7VG9wb2l9X3tcXGxhbWJkYX1ee1xcdGV4dHtvcH19Il0sWzIsMSwiXFxtYXRoc2Z7TG9nb2l9Il0sWzQsMCwiLi4uIl0sWzEsMCwiXFxtYXRoc2Z7VG9wb2l9X3tcXGFsZXBoXzF9XntcXHRleHR7b3B9fSJdLFsyLDFdLFs0LDJdLFs1LDBdLFsxLDVdLFswLDMsIiIsMCx7ImN1cnZlIjoxfV0sWzUsM10sWzEsM10sWzIsM10sWzQsMywiIiwwLHsiY3VydmUiOi0xfV0sWzEwLDExLCIiLDAseyJzaG9ydGVuIjp7InNvdXJjZSI6NDAsInRhcmdldCI6NDB9fV0sWzExLDEyLCIiLDAseyJzaG9ydGVuIjp7InNvdXJjZSI6NDAsInRhcmdldCI6NDB9fV0sWzEyLDEzLCIiLDAseyJzaG9ydGVuIjp7InNvdXJjZSI6NDAsInRhcmdldCI6NDB9fV0sWzEzLDE0LCIiLDAseyJzaG9ydGVuIjp7InNvdXJjZSI6NDAsInRhcmdldCI6NDB9fV1d
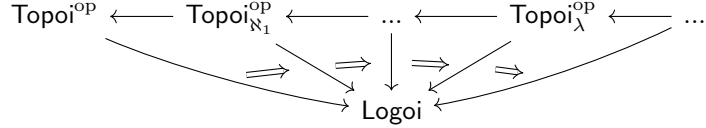
\[\begin{tikzcd}[ampersand replacement=\&]
	{\mathsf{Topoi}^{\text{op}}} \& {\mathsf{Topoi}_{\aleph_1}^{\text{op}}} \& {...} \& {\mathsf{Topoi}_{\lambda}^{\text{op}}} \& {...} \\
	\&\& {\mathsf{Logoi}}
	\arrow[from=1-4, to=1-3]
	\arrow[from=1-5, to=1-4]
	\arrow[from=1-2, to=1-1]
	\arrow[from=1-3, to=1-2]
	\arrow[""{name=0, anchor=center, inner sep=0}, curve={height=6pt}, from=1-1, to=2-3]
	\arrow[""{name=1, anchor=center, inner sep=0}, from=1-2, to=2-3]
	\arrow[""{name=2, anchor=center, inner sep=0}, from=1-3, to=2-3]
	\arrow[""{name=3, anchor=center, inner sep=0}, from=1-4, to=2-3]
	\arrow[""{name=4, anchor=center, inner sep=0}, curve={height=-6pt}, from=1-5, to=2-3]
	\arrow[shorten <=9pt, shorten >=9pt, Rightarrow, from=0, to=1]
	\arrow[shorten <=7pt, shorten >=7pt, Rightarrow, from=1, to=2]
	\arrow[shorten <=7pt, shorten >=7pt, Rightarrow, from=2, to=3]
	\arrow[shorten <=9pt, shorten >=9pt, Rightarrow, from=3, to=4]
\end{tikzcd}\]
    \item it is even more general than this specific theory, offering an environment that can encompass even non geometric-like fragments of infinitary logic. 
\end{itemize}

The main result of this paper will be a Diaconescu-type theorem for (rounded) sketches and logoi, offering the most classical starting point for the theory of classifying logoi. 
That means, for any rounded sketch $\sS$, we can construct a \textit{universal logos} $\mathsf{C}l[\sS]$ associated with it, so that $\mathsf{C}l[\sS]$ classifies the models of $\sS$ in any logos $\sT$, 
\begin{center}
$\roundSkt(\sS,\sT) \simeq \Log^\moritaM(\mathsf{C}l[\sS],\sT)$. 
    \end{center}
In more categorical terms, we can rephrase it as follows. 

\begin{thm*}[Diaconescu for Logoi, \Cref{diaconescu}] 

  The $2$-category  $\mathsf{Log}^\mathsf{M}$ of Morita small logoi is (bi)reflective in the $2$-category $r\mathsf{Skt}^\mathsf{M}$, of Morita small rounded sketches.
% https://q.uiver.app/#q=WzAsMixbMSwwLCJcXExvZ15cXG1vcml0YU0iXSxbMCwwLCJcXHJvdW5kU2t0XlxcbW9yaXRhTSJdLFswLDEsIlUiLDIseyJjdXJ2ZSI6Mn1dLFsxLDAsIlxcbWF0aGNhbHtDfWxbLV0iLDIseyJjdXJ2ZSI6Mn1dLFszLDIsIlxcdG9wIiwxLHsic2hvcnRlbiI6eyJzb3VyY2UiOjIwLCJ0YXJnZXQiOjIwfSwic3R5bGUiOnsiYm9keSI6eyJuYW1lIjoibm9uZSJ9LCJoZWFkIjp7Im5hbWUiOiJub25lIn19fV1d
\[\begin{tikzcd}[ampersand replacement=\&]
	{\roundSkt^\moritaM} \& {\Log^\moritaM}
	\arrow[""{name=0, anchor=center, inner sep=0}, "U"', curve={height=12pt}, from=1-2, to=1-1]
	\arrow[""{name=1, anchor=center, inner sep=0}, "{\mathcal{C}l[-]}"', curve={height=12pt}, from=1-1, to=1-2]
	\arrow["\top"{description}, draw=none, from=1, to=0]
\end{tikzcd}\]
    % \[\mathcal{C}l[-]: \mathsf{rSkt}^\mathsf{M} \leftrightarrows \mathsf{Log}^\mathsf{M} : j.\]
\end{thm*}

Of course, this is far from being our only result, and we shall break down the content of the paper and our contribution in the next subsection.

\subsection*{Our contribution and structure of the paper}
%Broadly speaking, this paper is divided in two parts, in the first part we provide some general results about the $2$-category of sketches. In the second part, 

The paper starts with \Cref{sec:notions-skt} which introduces the general theory of sketches, fixing the notation and recalling some examples and results of interest. 

Then, in Sections \ref{sec:2-dim-asp-skt} and \ref{sec:richard}, we  contribute to the general theory of sketches, especially the 2-dimensional part. In particular we construct the maximal/minimal sketch structure making some collection of functors sketch morphisms (\Cref{prop:min-max}) and use it to get some interesting results such as: 
\begin{itemize}
    \item Theorems \ref{thm:skt-has-p-lim} and \ref{thm:skt-has-p-colim}, where we describe explicitly pseudo-co/limits in the 2-categories $\skt/\Skt$;
    \item \Cref{thm:skt-closed-mon} which proves that Benson's tensor product (see \cite{benson1997multilinearity}) of sketches is closed.
\end{itemize}

The final three sections (\ref{sec:left-sketches}, \ref{sec:rounded-skt} and \ref{sec:class-logoi}) contain our main results, which are the Diaconescu-type theorems: \ref{diacweak} (for \emph{left} sketches) and \ref{diaconescu} (for \emph{rounded} sketches and \emph{logoi}). 

In particular, in \Cref{sec:left-sketches} we introduce the notion of left sketch (\Cref{def:left-sketch}) and construct a left sketch classifier for any (Morita small) sketch (see \Cref{costr:left-skt-class} for small sketches and \Cref{correctionlarge} for Morita small ones). 
Using left sketches as example, we also underline how useful  a more \emph{parametric} approach to Morita theory can be (see \Cref{sec:par-morita-theory}). 

Then, in \Cref{sec:rounded-skt} we define rounded sketches (\Cref{def:rounded-skt}), proving some useful results about them. We draw a close connection between the theory of rounded sketches and that of lex colimits \cite{garner2012lex} by Garner and Lack. We discuss this connection mostly in \Cref{prop:phi-ex-round}.

We end with \Cref{sec:class-logoi} defining our notion of logos (\Cref{def:logos}), proving the main result \Cref{diaconescu} and showing how this result can recover known ones in the literature (for topoi and $\Phi$-exact categories, see \Cref{rmk:one-diaconescu-all}).

% \begin{itemize}
% \item  
%     \item Normalization, left sketches and classifying left sketch, rounded sketch and logos
%     \item Notions of exactness for sketches and lex colimits
%     \item Formal category theory in the $2$-category of sketches
%     \item an analysis of what metters in the definition of topos, really. For example in (a) the lemme the comparison it's important that L preserve pullbacks and in (b) for the (localization, conservative) factorization system we need that L preserve reflexive equalizers.
% \end{itemize}

\section{Notions of Sketch}\label{sec:notions-skt}

The aim of this section is to set the scene for the sections to come, clarifying the notions of sketch we will deal with and their associated $2$-categories. Some of the content of this section can be found in the literature, but some notation needs to be set in order to properly handle size and coherence issues. We also take the opportunity to tailor the presentation towards our interests, putting into context the examples that will justify the notion of (classifying) logos later in the paper.

The literature on the general topic of sketches is extremely vast and not easy to organise coherently and compactly. In this paper, we must make some design choices concerning the definition of sketch. Mostly, we will follow \cite[2.F]{adamek_rosicky_1994}, and we shall briefly comment on the other approaches when needed. We recommend the reader to check out \cite{Makkaipare} for a more \textit{graph}-oriented approach to sketches. Finally, we refer to \cite{wells1994sketches} for a(n almost) comprehensive list of references.

\subsection{Small and large sketches}
\begin{notat}[\textit{$2$-categories of categories: $\cat, \Cat$}]
Throughout the paper, $\cat$ is the $2$-category of small categories and functors between them and $\Cat$ is the $2$-category of locally small categories. %, and $\CAT$ is the $2$-category of locally large categories. an $\CAT$
% \red [Io sta frase la toglierei proprio] Of course, $\Cat$ is an illegitimate $2$-categories as their collection of objects is a conglomerate. We shall not stress on this size issue when it is not necessary for technical purposes.\black 
\end{notat}

\begin{defn}[\textit{Notions of sketch}] \label{skt}
\begin{enumerate}[(i)]
    \item[]
    \item A \textit{sketch} $\sS= (\cS, \skL_\sS, \skC_\sS)$ is a triple where:
\begin{itemize}
    \item $\cS$ is a locally small category,
    \item $\skL_\sS$ is a class of (essentially small) diagrams $d_i\colon D_i \to \cS$ in $\cS$, each of them equipped with the choice of a cone $\pi_i\colon  \Delta(s_i) \Rightarrow d_i$. ($\Delta$ is the usual diagonal functor $\Delta\colon  \cS \to \cS^{D_i}$).
    \item $\skC_\sS$ is a class of (essentially small) diagrams $d_k\colon D_k \to \cS$ in $\cS$, each of them equipped with the choice of a cocone $j_k\colon d_k \Rightarrow \Delta(s_k)$.
\end{itemize}

    \item A sketch is \textit{normal} if the cones and the cocones are of co/limit form. 
    
    \item A sketch is \textit{limit} if $\skC_\sS$ is empty, \textit{colimit} if $\skL_\sS$ is empty. 
    
    \item A sketch is \textit{small} if $\cS$ is essentially small. 
    
    %\item A sketch is \textit{moderate} if the classes $\skL$ and $\skC$ are essentially small. 
    
    \item We may informally say \textit{large} sketch to stress on the fact that the sketch may not be small. Similarly, following the french tradition, we may informally say \textit{mixed} sketch to stress on the fact that the sketch may not be limit nor colimit.
    
    \item For a property $\square$ that quantifies on both the classes of cones and cocones, we may say \textit{right} $\square$ or \textit{left} $\square$ if only one of these requirements is verified. Example: a sketch is left normal if all the specified cocones are of colimit form. We may drop the decorations $(\skL_\sS, \skC_\sS)$ when evident from the context.

\end{enumerate}

\end{defn}

\begin{defn}[\textit{Morphism of Sketches}]
A morphism of sketches $F\colon \sS\to \sT$ is a functor between underlying categories with the property that:
\begin{itemize}
    \item for every cone $\pi$ in $\skL_\sS$, the image $F\pi$ is naturally isomorphic to a cone in~$\skL_\sT$;
    \item  for every cocone $j$ in $\skC_\sS$, $Fj$ is naturally isomorphic to a cocone in $\skC_\sT$.
\end{itemize}
\end{defn}

\begin{rem}[\textit{On the strictness of the notion of morphism}]
Notice that this notion of morphism of sketch is not the usual one (see for instance \cite[D.2.1.1~(b)]{Sketches}). The usual notion of morphism of sketch strictly preserves the (co)cones. Our choice will be better justified by the examples following the next definition.
\end{rem}

\begin{defn}[\textit{$2$-category of Sketches}] \label{sktcats}
The $2$-category of sketches $\Skt$ has objects (possibly large but locally small) sketches, $1$-cells morphisms of sketches and $2$-cells natural transformations between them. The $2$-category $\skt$ is the full sub $2$-category of $\Skt$ containing small sketches,
\[\skt \hookrightarrow \Skt.\]
We denote with $\limskt$/$\limSkt$ and $\colimskt$/$\colimSkt$ the sub-2-categories of $\skt$/$\Skt$ with objects limit and colimit sketches, respectively. 
\end{defn}

\begin{rem}\label{rem:csk-lsk-corefl-in-skt}
Clearly, there are forgetful 2-functors $U_c\colon \Skt\to\colimSkt$ and $U_l\colon \Skt\to\limSkt$ with $U_c(\cS,\skL,\skC)=(\cS,\emptyset,\skC)$ and $U_l(\cS,\skL,\skC)=(\cS,\skL,\emptyset)$. Moreover, these sub-2-categories are coreflective.
% https://q.uiver.app/#q=WzAsMyxbMCwwLCJcXENTa3QiXSxbMSwwLCJcXFNrdCJdLFsyLDAsIlxcTFNrdCJdLFswLDEsIkleYyIsMix7ImN1cnZlIjoyLCJzdHlsZSI6eyJ0YWlsIjp7Im5hbWUiOiJob29rIiwic2lkZSI6InRvcCJ9fX1dLFsyLDEsIklebCIsMCx7ImN1cnZlIjotMiwic3R5bGUiOnsidGFpbCI6eyJuYW1lIjoiaG9vayIsInNpZGUiOiJib3R0b20ifX19XSxbMSwyLCJVXmwiLDAseyJjdXJ2ZSI6LTJ9XSxbMSwwLCJVXmMiLDIseyJjdXJ2ZSI6Mn1dLFs2LDMsIlxcdG9wIiwxLHsic2hvcnRlbiI6eyJzb3VyY2UiOjIwLCJ0YXJnZXQiOjIwfSwic3R5bGUiOnsiYm9keSI6eyJuYW1lIjoibm9uZSJ9LCJoZWFkIjp7Im5hbWUiOiJub25lIn19fV0sWzUsNCwiXFx0b3AiLDEseyJzaG9ydGVuIjp7InNvdXJjZSI6MjAsInRhcmdldCI6MjB9LCJzdHlsZSI6eyJib2R5Ijp7Im5hbWUiOiJub25lIn0sImhlYWQiOnsibmFtZSI6Im5vbmUifX19XV0=
\[\begin{tikzcd}[ampersand replacement=\&]
	\colimSkt \& \Skt \& \limSkt
	\arrow[""{name=0, anchor=center, inner sep=0}, "{I^c}"', curve={height=12pt}, hook, from=1-1, to=1-2]
	\arrow[""{name=1, anchor=center, inner sep=0}, "{I^l}", curve={height=-12pt}, hook', from=1-3, to=1-2]
	\arrow[""{name=2, anchor=center, inner sep=0}, "{U_l}", curve={height=-12pt}, from=1-2, to=1-3]
	\arrow[""{name=3, anchor=center, inner sep=0}, "{U_c}"', curve={height=12pt}, from=1-2, to=1-1]
	\arrow["\top"{description}, draw=none, from=3, to=0]
	\arrow["\top"{description}, draw=none, from=2, to=1]
\end{tikzcd}\]

 It is also interesting to notice that a sketch structure on a category is literally the data of a limit sketch and colimit sketch structure, providing the 2-pullback below left. Hence, for given any two sketches $\sS$ and $\sT$, a functor $F\colon \cS\to\cT$ is a morphism of sketches if and only if it is 1-cell both in $\Skt_c$ and $\Skt_l$, i.e. the square below right is also a pullback. 

\begin{center}
    % https://q.uiver.app/#q=WzAsNCxbMCwwLCJcXFNrdCJdLFsxLDAsIlxcU2t0X2MiXSxbMCwxLCJcXFNrdF9sIl0sWzEsMSwiXFxDYXQiXSxbMCwxXSxbMCwyXSxbMiwzXSxbMSwzXSxbMCw2LCIiLDEseyJsZXZlbCI6MSwic3R5bGUiOnsibmFtZSI6ImNvcm5lciJ9fV1d
\begin{tikzcd}[ampersand replacement=\&]
	\Skt \& {\Skt_c} \\
	{\Skt_l} \& \Cat
	\arrow[from=1-1, to=1-2]
	\arrow[from=1-1, to=2-1]
	\arrow[""{name=0, anchor=center, inner sep=0}, from=2-1, to=2-2]
	\arrow[from=1-2, to=2-2]
	\arrow["\lrcorner"{anchor=center, pos=0.125}, draw=none, from=1-1, to=0]
\end{tikzcd} $\quad$
% https://q.uiver.app/#q=WzAsNCxbMCwwLCJcXFNrdChcXHNTLFxcc1QpIl0sWzEsMCwiXFxTa3RfYyhVXmNcXHNTLFVeY1xcc1QpIl0sWzAsMSwiXFxTa3RfbChVXmxcXHNTLFVebFxcc1QpIl0sWzEsMSwiXFxDYXQoXFxjUyxcXGNUKSJdLFswLDFdLFswLDJdLFsyLDNdLFsxLDNdLFswLDYsIiIsMSx7ImxldmVsIjoxLCJzdHlsZSI6eyJuYW1lIjoiY29ybmVyIn19XV0=
\begin{tikzcd}[ampersand replacement=\&]
	{\Skt(\sS,\sT)} \& {\Skt_c(U^c\sS,U^c\sT)} \\
	{\Skt_l(U^l\sS,U^l\sT)} \& {\Cat(\cS,\cT)}
	\arrow[from=1-1, to=1-2]
	\arrow[from=1-1, to=2-1]
	\arrow[""{name=0, anchor=center, inner sep=0}, from=2-1, to=2-2]
	\arrow[from=1-2, to=2-2]
	\arrow["\lrcorner"{anchor=center, pos=0.125}, draw=none, from=1-1, to=0]
\end{tikzcd}
\end{center}

\end{rem}

\begin{notat}[Decorating forgetful functors] \label{decoratedforg}
    As it just happened in the Remark above, this paper will contain a number of forgetful functor between different $2$-categories of sketches. Some abuse of notation will be unavoidable. When it's possible we will stick to the following notation. For  $\blacksquare\skt$ and $\square\skt$ $2$-categories of decorated sketches of some form, we shall call $U^\blacksquare_\square$ a forgetful functor whose domain is $\blacksquare\skt$ and codomain is $\square\skt$, \[U^\blacksquare_\square: \blacksquare\skt \to \square\skt.\]
\end{notat}
\subsubsection{Some examples: Doctrines, sites, topoi}

\begin{exa}[\textit{First order doctrines}] \label{doctrines} The blueprint of functorial semantics à la Lawvere is captured by the notion of sketch.  Every fragment of first order logic (cartesian, regular, full first order, geometric, coherent) is usually encoded in the general framework of \textit{exactness properties} or \textit{lex colimits}, 
 forming so called $2$-categories of \textit{theories}.  All of these 2-categories  admit a canonical $2$-functor in the $2$-category of sketches, as we shall discuss below. A good reference for the discussion below is \cite[D2.1]{Sketches}.
\begin{enumerate}[(a)]
    \item Let $\cprod$ be the $2$-category of small categories with finite products, functors preserving them and natural transformation. As discussed by many, these categories provide a syntax that can be used to present multisorted varieties in the sense of universal algebra \cite{adamek2003duality}. 
    We have a $2$-functor, \[i\colon \cprod \to \skt\] equipping a category with finite products with the (normal limit) sketch structure whose limit diagrams are the finite product cones. Of course this embedding is locally fully faithful, as the notion of product preserving functor coincides with that of morphism of sketches in this case.
    \item Let $\lex$ be the $2$-category of small categories with finite limits, functors preserving them and natural transformations. It was shown by Freyd \cite{FreydCartesianLogic,Aspects} that a category with finite limits can be thought as a place-holder for an essentially algebraic theory, in a sense that was later clarified by many (see \cite[3.D]{adamek_rosicky_1994} or \cite{di2021functorial}). Similarly to the previous case, we have a locally fully faithful $2$-functor 
    \[i\colon  \lex \to \skt.\]
    \item Let $\disj$ be the $2$-category of lex-tensive categories \cite[Sec 4.4]{carboni1993introduction}, coproduct preserving lex functors, and natural transformation. It was shown by Johnstone (\cite{johnstone2006syntactic} and \cite[Sections D1 and D2]{Sketches}) that these categories offer an adequate functorial semantics to discuss disjunctive logic. Clearly, we can equip every lextensive category with a mixed sketch structure, which provides us with a locally fully faithful $2$-functor 
    \[i\colon \disj \to \skt.\]
    \item Let $\reg$ be the $2$-category of regular categories, regular functors and natural transformation. These categories are the syntactic categories of regular theories, as discussed by Butz \cite{butz1998regular}. Recall that a regular functor can be understood as a lex functor that preserve regular epimorphisms. Thus, similarly to the previous example, we can equip a regular category with a (normal mixed) sketch structure whose limit cones are finite limit cones and colimit cones are the pushouts witnessing the property of being an epimorphism.  Without surprise, these are called regular sketches in the literature. We obtain in this way a locally fully faithful $2$-functor 
    \[i\colon  \reg \to \skt.\]
    \item Other $2$-categories of first order \textit{theories} like $\coh$ (the $2$-category of coherent categories) and more generally everything that falls in the general landscape of lex colimits \cite{garner2012lex}, admit a similar behavior and can be swallowed by sketches.
% https://q.uiver.app/#q=WzAsNixbMiwyLCJcXHNrdCJdLFszLDEsIlxcbGV4Il0sWzEsMSwiXFxyZWciXSxbMCwxLCJcXGNvaCJdLFs0LDEsIlxcY3Byb2QiXSxbMiwwLCJcXGRpc2oiXSxbMywwLCIiLDIseyJjdXJ2ZSI6MSwic3R5bGUiOnsidGFpbCI6eyJuYW1lIjoibW9ubyJ9fX1dLFsyLDAsIiIsMix7InN0eWxlIjp7InRhaWwiOnsibmFtZSI6Im1vbm8ifX19XSxbMSwwLCIiLDIseyJzdHlsZSI6eyJ0YWlsIjp7Im5hbWUiOiJtb25vIn19fV0sWzQsMCwiIiwwLHsiY3VydmUiOi0xLCJzdHlsZSI6eyJ0YWlsIjp7Im5hbWUiOiJtb25vIn19fV0sWzMsMl0sWzUsMV0sWzEsNF0sWzUsMF0sWzIsMV0sWzksOCwiIiwwLHsic2hvcnRlbiI6eyJzb3VyY2UiOjIwLCJ0YXJnZXQiOjIwfX1dLFs3LDYsIiIsMix7InNob3J0ZW4iOnsic291cmNlIjoyMCwidGFyZ2V0IjoyMH19XSxbOCwxMywiIiwwLHsic2hvcnRlbiI6eyJzb3VyY2UiOjIwLCJ0YXJnZXQiOjIwfX1dLFs4LDcsIiIsMix7InNob3J0ZW4iOnsic291cmNlIjoyMCwidGFyZ2V0IjoyMH19XV0=
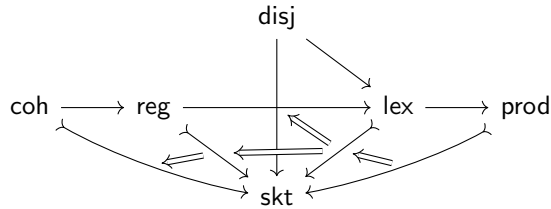
\[\begin{tikzcd}[ampersand replacement=\&]
	\&\& \disj \\
	\coh \& \reg \&\& \lex \& \cprod \\
	\&\& \skt
	\arrow[""{name=0, anchor=center, inner sep=0}, curve={height=6pt}, tail, from=2-1, to=3-3]
	\arrow[""{name=1, anchor=center, inner sep=0}, tail, from=2-2, to=3-3]
	\arrow[""{name=2, anchor=center, inner sep=0}, tail, from=2-4, to=3-3]
	\arrow[""{name=3, anchor=center, inner sep=0}, curve={height=-6pt}, tail, from=2-5, to=3-3]
	\arrow[from=2-1, to=2-2]
	\arrow[from=1-3, to=2-4]
	\arrow[from=2-4, to=2-5]
	\arrow[""{name=4, anchor=center, inner sep=0}, from=1-3, to=3-3]
	\arrow[from=2-2, to=2-4]
	\arrow[shorten <=4pt, shorten >=4pt, Rightarrow, from=3, to=2]
	\arrow[shorten <=4pt, shorten >=4pt, Rightarrow, from=1, to=0]
	\arrow[shorten <=4pt, shorten >=4pt, Rightarrow, from=2, to=4]
	\arrow[shorten <=6pt, shorten >=6pt, Rightarrow, from=2, to=1]
\end{tikzcd}\]
    It is quite important to notice that the diagram above is not strictly commutative, and indeed this is acknowledging the fact that the horizontal $2$-functors are forgetful functors, which destroy some of the information that instead the vertical ones retain. 
\item It goes without saying that the large correspondents of the above-mentioned $2$-categories  locally fully faithfully in $\Skt$, to fix the notation, and clarify the statement, we shall write one example: $\Lex \to \Skt$.
\end{enumerate}

Notice that our choice of morphisms for $\skt$ is the only one making these inclusions possible, otherwise we would need to choose a stricter notion of preservation of limits in $\cprod, \lex, \reg, \coh$, which would be un-natural. %extremely
\end{exa}

\begin{rem}[Sketches are presentations!]\label{rmk:skt-corrsp-to-diff-presentations}
    One might have noticed that in part (b) of \Cref{doctrines} we could have chosen a different sketch structure for a category with finite limits. 
    Indeed, given a category with finite limits $\cC$, one could also consider the sketch structure with cone specified only the equalisers and finite products, giving a sketch $j\cC$. 
    %Let us denote with $i^\ast\cC$ the sketch with underlying category $\cC$ and the sketch structure just described. 
    This still produce a fully faithful 2-functor 
    $$j\colon\lex\to\skt,$$
    since finite limits are constructed by equilisers and finite products. 
    %It is natural to wonder whether the identity functor $1_\cC$ is an equivalence between the two sketch structures $i\cC$ and $i^\ast\cC$.
    It is interesting to notice that the identity is a sketch morphism \emph{only in one direction}, being $j\cC\to i\cC$. 
    This fits well with the idea that sketches are  \emph{presentations} of theories and in this case we have indeed two different presentations of the same theory.

    Later, we will introduce another notion of equivalence for sketches (see \Cref{def:morita-equiv}) which will describe exactly when two sketches represent the same theory. 
    %In the literature this has been usually done with models for sketches. 
    %One could try to change the definition of sketch in order to make this two sketches equivalent
\end{rem}

\begin{exa}[\textit{(Lex) sites}] \label{sites}
Sites are used to present topoi, and they can encode among other constructions the syntactic category of a geometric theory, offering a perfect framework to discuss geometric logic. For the sake of simplicity, in this paper a site $(C,J)$ is by definition lex, meaning that $C$ has finite limits. Following \cite[D.1.4(g)]{Sketches}, we can turn a lex site into a mixed sketch, and a morphism of sites is precisely a morphism between the associated sketches. This gives us a locally fully faithful embedding,
\[\sites \to \skt.\]
\end{exa}

Our final example of interest regards the canonical sketch structure on a topos, in order to discuss it, let us recall Diaconescu theorem below, which will be helpful for the discussion.

\begin{notat}[\textit{Morita small site}] \label{moritafirstappearence}
    In the statement below, we say that a (possibly large) site is Morita small if it admits a small dense subsite (see \cite[C2.2 page 548 and C2.2.1]{Sketches}).
\end{notat}

\begin{thm}[\textit{Diaconescu}] \label{diac1}
There exists a biadjunction $\mathsf{Sh} \dashv J_{(-)}$ between the $2$-category of topoi and the $2$-category of Morita small sites.
 % \[\mathsf{Sh} :\mathsf{MSite} \leftrightarrows  \mathsf{Topoi}^{\text{op}}: J_{(-)}.\]
 % \yellow [Io forse preferisco come sotto, ma lascia quel che preferisci, no prob]
% https://q.uiver.app/#q=WzAsMixbMCwwLCJcXG1hdGhzZntNU2l0ZX0iXSxbMSwwLCJcXG1hdGhzZntUb3BvaX1ee1xcdGV4dHtvcH19Il0sWzEsMCwiSl97KC0pfSIsMix7Im9mZnNldCI6MSwiY3VydmUiOjF9XSxbMCwxLCJcXG1hdGhzZntTaH0iLDIseyJvZmZzZXQiOjEsImN1cnZlIjoxfV0sWzIsMywiXFx0b3AiLDEseyJzaG9ydGVuIjp7InNvdXJjZSI6MjAsInRhcmdldCI6MjB9LCJzdHlsZSI6eyJib2R5Ijp7Im5hbWUiOiJub25lIn0sImhlYWQiOnsibmFtZSI6Im5vbmUifX19XV0=
\[\begin{tikzcd}[ampersand replacement=\&]
	{\mathsf{MSite}} \& {\mathsf{Topoi}^{\text{op}}}
	\arrow[""{name=0, anchor=center, inner sep=0}, "{J_{(-)}}"', shift right, curve={height=6pt}, from=1-2, to=1-1]
	\arrow[""{name=1, anchor=center, inner sep=0}, "{\mathsf{Sh}}"', shift right, curve={height=6pt}, from=1-1, to=1-2]
	\arrow["\top"{description}, draw=none, from=0, to=1]
\end{tikzcd}\]
% \black 
$\mathsf{Sh}$ is taking sheaves over the site, while $ J_{(-)}$ equips a topos with the so-called canonical topology. Moreover, the counit $\epsilon\colon  \mathsf{Sh}(\mathcal{E}, J_\mathcal{E}) \to \mathcal{E}$ of this adjunction is an equivalence of categories.
\end{thm}

 We underline that \Cref{diac1} is often stated in the more down-to earth (and a bit weaker) form that there exists an equivalence of categories as below.
\[\mathsf{Topoi}(\,\mathcal{E}, \mathsf{Sh}(C,J)\,) \simeq \mathsf{MSite}(\,(C,J),(\mathcal{E},J_\mathcal{E})\,)\]

\begin{proof} \emph{(of \Cref{diac1})}
    See \cite[C2.3.9]{Sketches} for the biadjunction and \cite[C2.2.7]{Sketches} for the counit.
\end{proof}

\begin{exa}[\textit{Topoi}] \label{firsttimetopos}

We can equip any topos $\mathcal{E}$ with a sketch structure by defining $\skC$ as all colimit cocones and $\skL$ as all finite limit cones. Clearly, this gives us the dashed $2$-functor in the diagram below.

% https://q.uiver.app/#q=WzAsMyxbMCwwLCJcXG1hdGhzZntNU2l0ZX0iXSxbMiwwLCJcXG1hdGhzZntUb3BvaX1eXFx0ZXh0e29wfSJdLFsxLDEsIlxcbWF0aHNme1NrdH0iXSxbMCwyLCJcXHJlZntzaXRlc30iLDFdLFsxLDAsIlxccmVme2RpYWN9IiwyXSxbMSwyLCIiLDIseyJzdHlsZSI6eyJib2R5Ijp7Im5hbWUiOiJkYXNoZWQifX19XV0=
\[\begin{tikzcd}
	{\mathsf{MSite}} && {\mathsf{Topoi}^\text{op}} \\
	& {\mathsf{Skt}}
	\arrow["{\ref{sites}}"{description}, from=1-1, to=2-2]
	\arrow["{\ref{diac1}}"', from=1-3, to=1-1]
	\arrow[dashed, from=1-3, to=2-2]
\end{tikzcd}\]
 The only delicate point of this construction is of course its variance, as we have to choose the inverse image part of a geometric morphism. It is easy to see that this dashed $2$-functor coincides with the composition of \Cref{sites} and \Cref{diac1}, so that the diagram above commutes.
\end{exa}

\subsection{A first encounter with Morita theory} \label{moritaprimavolta}

In the previous subsection we have briefly discussed the relevance of Morita small sites (\Cref{moritafirstappearence}) in the statement of Diaconescu's theorem. Indeed, a topos is not a small site, but its information can be presented via a set of generators, which provides a number of different dense subsites. In this brief subsection we introduce an appropriate notion of Morita small sketch, which will adapt later in order to treat our generalization of Diaconescu theorem in the last section.

\begin{defn}[\textit{Test sketch}]
A \textit{test} sketch $\sM$ is a normal sketch whose underlying subcategory $\cM$ is complete and cocomplete and whose classes $\skL$ and $\skC$ consist of all small limit/colimit diagrams.
\end{defn}

\begin{rem}
    In the definition of \emph{test sketch} we are acknowledging a general tendency of the literature of studying only models of a sketch into complete and cocomplete categories. Indeed, this is nothing but a sketch morphism $\sS\to\sM$ into a test sketch. Often, authors restrict even to $\sM=\Set$.
\end{rem}

\begin{defn}[\textit{Test Morita equivalence}]\label{def:morita-equiv}
A morphism of sketches $F\colon  \sS \to \sT$ is a \textit{test Morita equivalence} if, for all test sketches $\sM$, the induced map between hom-categories is an equivalence of categories,\[-\circ F=F^*\colon  \Skt(\sT, \sM) \to \Skt(\sS, \sM).\]
\end{defn}

\begin{defn}[\textit{Test Morita small sketch}] \label{Moritasmall}
A sketch $\sT$ is test Morita small if there exist a test Morita equivalence $F\colon  \sS \to \sT$ whose domain is a small sketch. 
\end{defn}

At this stage we have introduced test Morita equivalences to prepare the ground for a parametric version of this notion that will be  used  later in \Cref{sec:left-sketches}. In particular, this paper will focus on the stronger notion of \textit{left Morita equivalence} (\Cref{def:strong-morita-equiv}). This said though, test Morita equivalences are a very natural notion, especially from the perspective of the existing literature, and would frame an interesting theory. For example, the cloud of results surrounding Gabriel-Ulmer duality suggests the conjecture below.

\begin{conj}
A morphism of limit sketches $F\colon\sS\to\sT$ is a test Morita equivalence if and only if it induces an equivalence between categories below, where $\Set$ is the category of sets with the natural test structure,
\[F^*\colon  \Skt(\sT, \Set) \to \Skt(\sS, \Set).\]
\end{conj}

Despite some attempts, we have not managed to provide a proof of the conjecture above. The difficulty seems to lie in some delicate size issues. It seems plausible that, by requiring test sketches to be LAFT categories in the sense of Brandenburg \cite[Remark~3.9]{brandenburg2021large}, these issues may be circumvented, and a proof of (a version of) the conjecture above may be delivered. Yet, the technology needed to show this result would have brought us too far from the general purpose of this paper.

\section{$2$-dimensional aspects of sketches}\label{sec:2-dim-asp-skt}

In the following section, we will show some important 2-dimensional properties of the 2-category of sketches $\skt/\Skt$, which will be useful in later sections. 

\subsection{A topological behavior}\label{sec:top-beh}
We start by studying certain \emph{topological} properties of the forgetful functors
\begin{center}
    $U^s\colon  \skt\to\cat$ and $U^S\colon \Skt\to\Cat$,
    
\end{center}
in the sense of \emph{topological functors} \cite{HERRLICH1974125}. More precisely, we will focus on the aspects of this theory regarding co/limits. In fact, given a topological functor $T\colon \cX\to\cY$ between categories, one can compute co/limits in $\cX$ using co/limits in $\cY$. The prototypical example is the forgetful functor $U\colon \Top\to\Set$. For instance, given two topological spaces $X,Y\in\Top$, to calculate the product in $\Top$ we consider the product of $UX\times UY$ in $\Set$ and equipped it with the minimal topology making the two projections continous. In this section we will prove two preliminary results, \Cref{prop:almost-w-lim} and \ref{prop:almost-w-colim}, which will be useful in \Cref{sec:w-p-co/lim} to give formulas for weigthed pseudo-co/limits of sketches.

\begin{rem}
    The results in this sections suggest that the forgetful 2-functors $U_s$ and $U_S$ should be \emph{2-topological}, the appropriate 2-dimensional counterpart of topological functors which is not yet present in the literature. Since our main motivation (the formulas for weighted pseudo-co/limits) does not need it, we shall not treat this notion here and leave it to more dedicated venues.
    A good starting point would be to adapt the approach in \cite{Garner_top-funct} using the theory of 2-fibrations \cite{10.1007/BFb0063102,Buck:2-fib}. %\yellow Define both initial lifting for 1-cells (involving a 2-dimensional universal property too) and then lift of 2-cells which should be \emph{local} (maybe some stability under composition is needed). \black 
\end{rem}

\begin{prop}[\textit{Minimal and maximal sketch structures}]
\label{prop:min-max}
Let $\lbrace \sS_i\mid i\in I\rbrace$ a family of sketches and $\cC$ a category. 
\begin{enumerate}
    \item Given a family of functors $\mathfrak{F}:=\lbrace F_i\colon  \cS_i\to\cC\rbrace$, there exists a \emph{minimal} sketch structure $\sC$ on $\cC$ making each $F_i$ sketch morphisms. \\
     Moreover, with this sketch structure, given any sketch $\sY$, a functor $G\colon \cC\to\cY$ is a sketch morphism if and only if each $GF_i\colon \cS_i\to\cY$ is a sketch morphism.
    
    \item Given a family of functors $\mathfrak{G}:=\lbrace G_i\colon  \cC\to\cS_i\rbrace$, there exists a \emph{maximal} sketch structure on $\cC$ making each $G_i$ sketch morphisms. \\
    Moreover, with this sketch structure, given any sketch $\sX$, a functor $F\colon \cX\to\cC$ is a sketch morphism if and only if each $G_iF\colon \cX\to\cS_i$ is a sketch morphism. 
\end{enumerate}
\end{prop}

\begin{proof}
\begin{enumerate}
\item[]
    \item We define $\skL^m_\mathfrak{F}$ as the union of all cones $F_i\pi$ for any $i\in I$ and $\pi\in\skL_i$, and similarly $\skC^m_\mathfrak{F}$ the union of all cocones $F_ij$ for any $i\in I$ and $j\in\skL_i$. Clearly $F_i$ become maps of sketches and the classes $\skL^m_\mathfrak{F}$ and $\skC^m_\mathfrak{F}$ are the smallest classes with this property. \\
    Let us consider now a sketch $\sY$ and a functor $G\colon \cC\to\cY$ such that, for any $i$ the functor $GF_i\colon \sS_i\to\sY$ is a sketch morphism. This means that for any $\delta\in\skL_i/\skC_i$, $GF_i\delta$ is isomorphic to an element of $\skL_\sY/\skC_\sY$, which clearly implies that $G$ is a sketch morphism.  
    
    \item We consider $\skL^M_\mathfrak{G}$ as all cones $\pi$ in $\cC$ such that for any $i\in I$ there exists a cone $\pi'\in\skL_i$ such that $G_i\pi\cong\pi'$. Similarly $\skC^M_\mathfrak{G}$ consists of all cocones $j$ in $\cC$ such that for all $i\in I$ there exists a cocone $j'\in\skL_i$ such that $G_ij\cong j'$. Clearly this definition makes all $G_i$ sketches morphisms and the classes $\skL^M_\mathfrak{G}$ and $\skC^M_\mathfrak{G}$ are the biggest classes with this property. \\
    Let us consider now a sketch $\sX$ and a functor $F\colon\cX\to\cC$ such that, for any $i$ $G_iF\colon \sX\to\sS_i$ is a sketch morphism. Then, for any $\delta\in\skL_\sX/\skC_\sX$ and $i\in I$ there exists a $\delta'\in\skL_i/\skC_i$ such that $G_iF\delta\cong\delta'$, hence $F\delta\in\skL^M_\mathfrak{G}/\skC^M_\mathfrak{G}$. 

\end{enumerate}
\end{proof}

%\red WARNING: Maybe the notation above (see $\skL^M_\mathfrak{G}$...) is shite eh. Very much pronto a cambiarla, sto solo provandola.  \black 

% \begin{rem}
%    We should be able reformulate \Cref{prop:almost-w-lim} using bi-colimit (1) and bi-limit (2). Maybe using \cite{GAGNA2022107137} we can reformulate it saying that the structure on $\cC$ makes it as particular bi-initial/bi-final objects. 
% \red
% ...i.e. for any sketch $\sT$ and for any family of sketch morphisms $H_i:\sS_i\to\sT$ such that the functors factor through, there exists a unique sketch morphism $...$ \red + 2-dim property. \black \yellow Forse meglio: such that $\sC$ is the bi-colimit [...]...\black 
% \end{rem}

% \red
% La prop è falsa ma c'è qualcosa di vero nel casino. 

% \begin{prop}
%     Let $F\colon \cC\to\cS$ be a functor and $\sS$ a sketch structure on $\cS$. 
%     \begin{itemize}
%         \item If $\skL_\sS$ consists of all limits in $\cS$, then $\skL^M_F$ on $\cC$ is right normal and contains precisely the limit cones that are created by F. 
%         \item If $\sS$ is left normal, $F$ creates all colimits if and only if $C^M_F$ is equal to all colimits in $\cC$. 
%     \end{itemize}
% \end{prop}
% \black 

For sites, the situation with a single morphism was proven in \cite[Lemma C2.3.12 and C2.3.13]{elephant2}. Then, \cite[Lemma C2.3.14]{elephant2} proves that this gives us a formula to calculate weigthed limits/colimits in Sites from weigthed limits/colimits in $\Cat$. 

Below we show how something very similar is possible also for sketches. Moreover, these technical propositions (\ref{prop:almost-w-colim} and \ref{prop:almost-w-lim}) will be helpful in the concrete construction of pseudo-co/limits of sketches (see Theorems~\ref{thm:skt-has-p-lim} and \ref{thm:skt-has-p-colim}).

\begin{notat}[\textit{Weighted pseudo-co/limits}]
\label{notat:weight-co/proj}
  Let $W\colon \cI\to\cat$ be a weight and $D\colon \cI\to\skt$ a diagram of sketches. A $W$-weighted $U_sD$-cone consists of a category $\cC\in\cat$ together with a natural transformation 
  $$\pi_\cC\colon W\to\cat(\cC,U_sD-).$$
 Let $\cD_i:=U_sD(i)$, then we denote with $p_w^i\colon \cC\to\cD_i$ the $(i,w)$-projections of the cone, i.e. the image of $w\in W(i)$ under the $i$-component of $\pi$.
 \begin{center}
     % https://q.uiver.app/?q=WzAsMixbMCwwLCJXKGkpIl0sWzEsMCwiXFxjYXQoXFxjQyxcXGNEX2kpIl0sWzAsMSwiXFxwaV9pIl1d
\begin{tikzcd}[ampersand replacement=\&]
	{W(i)} \& {\cat(\cC,\cD_i)}
	\arrow["{\pi_i}", from=1-1, to=1-2]
\end{tikzcd} \\
% https://q.uiver.app/?q=WzAsMixbMCwwLCJ3Il0sWzEsMCwiXFxjQ1xceHJpZ2h0YXJyb3d7Y193Xml9XFxjRF9pIl0sWzAsMSwiIiwwLHsic3R5bGUiOnsidGFpbCI6eyJuYW1lIjoibWFwcyB0byJ9fX1dXQ==
\begin{tikzcd}[ampersand replacement=\&]
	w \& {\cC\xrightarrow{p_w^i}\cD_i}
	\arrow[maps to, from=1-1, to=1-2]
\end{tikzcd}
 \end{center}
 Moreover, we define 
 $$\mathfrak{L}:=\lbrace p_w^i\mid i\in \cI,w\in W(i)\rbrace.$$
 For any category $\cX$, we denote with 
 $$\Phi\colon \cat(\cX,\cC)\to[\cI,\cat](\,W,\cat(\cX,U_sD-)\,)$$ 
 the functor sending $F\colon\cS\to\cC$ to the transformation $W\to\cat(\cS,U_sD-)$ with $i$-component sending an object $w\in W(i)$ to ${p_w^i}\circ{F}\colon \cS\rightarrow\cC\rightarrow\cD_i$.

 Similarly, for a $W$-weighted $U_sD$-cocone with covertex $\cC$, we write $q_w^i\colon \cD_i\to\cC$  for the $(i,w)$-coprojections, we define 
 $$\mathfrak{C}:=\lbrace q_w^i\mid i\in \cI,w\in W(i)\rbrace,$$
 and, for any category $\cY$, we denote with $\Psi$ the associated functor, 
 $$\Psi\colon \cat(\cC,\cY)\to[\cI,\cat](\,W,\cat(U_sD-,\cY)\,).$$

\end{notat}

\begin{prop}
\label{prop:almost-w-lim}
    Let $W\colon \cI\to\cat$ be a weight, $D\colon \cI\to\skt$ a diagram of sketches and $\cC\in\cat$ a vertex of a $W$-weighted $U_sD$-cone. For any $\sS\in\skt$, the sketch $(\cC,\skL^M_{\mathfrak{L}},\skC^M_{\mathfrak{L}})$ makes the diagram below a pullback in $\cat$. \footnote{The functor $\Phi_s$ is well-defined because the sketch structure of $\sC$ makes all of the weighted-projections morphisms of sketches.}
    \begin{equation}
        \label{diag:pull-max-sk}
% https://q.uiver.app/?q=WzAsNCxbMCwwLCJcXHNrdChcXHNTLFxcc0MpIl0sWzEsMCwiXFxjYXQoXFxjUyxcXGNDKSJdLFsxLDEsIltJLFxcY2F0XShcXCxXLFxcY2F0KFxcY1MsVV9zRC0pXFwsKSJdLFswLDEsIltJLFxcY2F0XShcXCxXLFxcc2t0KFxcc1MsRC0pXFwsKSJdLFswLDEsIlVfcyJdLFszLDIsIlVcXGNpcmMtIiwyXSxbMSwyLCJcXFBoaSJdLFswLDMsIlxcUGhpX3MiLDJdXQ==
\begin{tikzcd}[ampersand replacement=\&]
	{\skt(\sS,\sC)} \& {\cat(\cS,\cC)} \\
	{[\cI,\cat](\,W,\skt(\sS,D-)\,)} \& {[\cI,\cat](\,W,\cat(\cS,U_sD-)\,)}
	\arrow["{U_s}", from=1-1, to=1-2]
	\arrow["{U\circ-}"', from=2-1, to=2-2]
	\arrow["\Phi", from=1-2, to=2-2]
	\arrow["{\Phi_s}"', from=1-1, to=2-1]
\end{tikzcd}
\end{equation}
\end{prop}

\begin{proof}
    Let us consider a category $\cX$ together with two functors $Q$ and $P$ making the diagram below commutative. 
    \begin{equation}
    \label{diag:proof-pullback}
    % https://q.uiver.app/?q=WzAsNCxbMCwwLCJcXGNYIl0sWzIsMSwiXFxjYXQoXFxjUyxcXGNDKSJdLFsyLDIsIltJLFxcY2F0XShcXCxXLFxcY2F0KFxcY1MsVV9zRC0pXFwsKSJdLFsxLDIsIltJLFxcY2F0XShcXCxXLFxcc2t0KFxcc1MsRC0pXFwsKSJdLFswLDEsIlEiLDAseyJjdXJ2ZSI6LTJ9XSxbMywyLCJVXFxjaXJjLSIsMl0sWzEsMiwiXFxQaGkiXSxbMCwzLCJQIiwyLHsiY3VydmUiOjJ9XV0=
\begin{tikzcd}[ampersand replacement=\&]
	\cX \\
	\&\& {\cat(\cS,\cC)} \\
	\& {[\cI,\cat](\,W,\skt(\sS,D-)\,)} \& {[\cI,\cat](\,W,\cat(\cS,U_sD-)\,)}
	\arrow["Q", curve={height=-12pt}, from=1-1, to=2-3]
	\arrow["{U_s\circ-}"', from=3-2, to=3-3]
	\arrow["\Phi", from=2-3, to=3-3]
	\arrow["P"', curve={height=12pt}, from=1-1, to=3-2]
\end{tikzcd}
    \end{equation}
    The commutativity of this square exactly means that, for any $x\in\cX$ the bottom left diagram commutes and so for any $i\in I$ and any $w\in W(i)$ the bottom right diagram commutes. 
    \begin{center}
    % https://q.uiver.app/?q=WzAsMyxbMCwwLCJXIl0sWzEsMSwiXFxjYXQoXFxjUyxVX3NELSkiXSxbMCwxLCJcXHNrdChcXHNTLEQtKSJdLFswLDEsIlxcUGhpIChReCkiXSxbMCwyLCJQeCIsMl0sWzIsMSwiVV9zIiwyXV0=
\begin{tikzcd}[ampersand replacement=\&]
	W \\
	{\skt(\sS,D-)} \& {\cat(\cS,U_sD-)}
	\arrow["{\Phi (Qx)}", from=1-1, to=2-2]
	\arrow["Px"', from=1-1, to=2-1]
	\arrow["{U_s}"', from=2-1, to=2-2]
\end{tikzcd}
\hspace{0.5cm}\hspace{0.5cm}
% https://q.uiver.app/?q=WzAsMyxbMCwwLCJcXGNTIl0sWzEsMCwiXFxjQyJdLFsxLDEsIlxcY0RfaSJdLFswLDIsIlVfcyhcXCwoUHgpX2kodylcXCwpIiwyXSxbMCwxLCJReCJdLFsxLDIsImNfd15pIl1d
\begin{tikzcd}[ampersand replacement=\&]
	\cS \& \cC \\
	\& {\cD_i}
	\arrow["{U_s(\,(Px)_i(w)\,)}"', from=1-1, to=2-2]
	\arrow["Qx", from=1-1, to=1-2]
	\arrow["{p_w^i}", from=1-2, to=2-2]
\end{tikzcd}
    \end{center}
    Now, we want to show that actually all of $Qx\colon \cS\to\cC$ are maps of sketches, i.e. for any $\sigma\in\skL_\sS/\skC_\sS$ there exists a $\tau\in\skL^M_{\mathfrak{L}}/\skC^M_{\mathfrak{L}}$ such that $Qx(\sigma)\cong\tau$. Actually, we can even show that $Qx(\sigma)$ itself is in $\skL^M_{\mathfrak{L}}/\skC^M_{\mathfrak{L}}$\footnote{This makes sense because we took the maximal sketch structures.}. In fact, for any $i\in \cI$ and $w\in W(i)$, since $(Px)_i(w)\in\skt(\sS,\sD_i)$, then there exists a $\delta\in\skL_i/\skC_i$ such that $[(Px)_i(w)](\sigma)\cong\delta$. Thus, 
    %\footnote{Here it would have been enough $p_w^i(\,Qx(\sigma)\,)\cong[(Px)_i(w)](\sigma)$ which would be the case if we consider pseudo-pullback, but there is no need since \eqref{diag:pull-max-sk} commutes.}
    $$p_w^i(\,Qx(\sigma)\,)=[(Px)_i(w)](\sigma)\cong\delta$$
    and so $Qx(\sigma)\in\skL^M_{\mathfrak{L}}/\skC^M_{\mathfrak{L}}$ by definition. This shows that $Q$ factorise through $\skt(\sS,\cC)$ as shown below (since transformations of sketches are the same as natural transformations). 
    % https://q.uiver.app/?q=WzAsNSxbMCwwLCJcXGNYIl0sWzIsMSwiXFxjYXQoXFxjUyxcXGNDKSJdLFsyLDIsIltJLFxcY2F0XShcXCxXLFxcY2F0KFxcY1MsVV9zRC0pXFwsKSJdLFsxLDIsIltJLFxcY2F0XShcXCxXLFxcc2t0KFxcc1MsRC0pXFwsKSJdLFsxLDEsIlxcc2t0KFxcc1MsXFxzQykiXSxbMCwxLCJRIiwwLHsiY3VydmUiOi0yfV0sWzMsMiwiVVxcY2lyYy0iLDJdLFsxLDIsIlxcUGhpIl0sWzAsMywiUCIsMix7ImN1cnZlIjoyfV0sWzAsNCwiXFxleGlzdHMhUSciLDEseyJzdHlsZSI6eyJib2R5Ijp7Im5hbWUiOiJkYXNoZWQifX19XSxbNCwxLCJVX3MiXSxbNCwzLCJcXFBoaV9zIiwyXV0=
\[\begin{tikzcd}[ampersand replacement=\&]
	\cX \\
	\& {\skt(\sS,\sC)} \& {\cat(\cS,\cC)} \\
	\& {[I,\cat](\,W,\skt(\sS,D-)\,)} \& {[I,\cat](\,W,\cat(\cS,U_sD-)\,)}
	\arrow["Q", curve={height=-12pt}, from=1-1, to=2-3]
	\arrow["{U_s\circ-}"', from=3-2, to=3-3]
	\arrow["\Phi", from=2-3, to=3-3]
	\arrow["P"', curve={height=12pt}, from=1-1, to=3-2]
	\arrow["{\exists!Q'}"{description}, dashed, from=1-1, to=2-2]
	\arrow["{U_s}", from=2-2, to=2-3]
	\arrow["{\Phi_s}"', from=2-2, to=3-2]
\end{tikzcd}\]
One can check that $\Phi_s Q'=P$ using the definitions and the commutativity of \eqref{diag:proof-pullback}. The uniqueness of $Q'$ follows from the fact that $U_s$ is an inclusion. 
\end{proof}

We remind that in the following proposition we follow \Cref{notat:weight-co/proj}.

\begin{prop}
\label{prop:almost-w-colim}
    Let $W\colon \cI\to\cat$ be a weight, $D\colon \cI\to\skt$ a diagram of sketches and $\cC\in\cat$ a vertex of a $W$-weighted $U_sD$-cocone. For any $\sS\in\skt$, the sketch $(\cC,\skL^m_\mathfrak{C},\skC^m_\mathfrak{C})$ makes the diagram below a pullback in $\cat$. \footnote{The functor $\Psi_s$ is well-defined because the sketch structure of $\sC$ makes all of the weighted-coprojections morphisms of sketches.}
    \begin{equation}
        \label{diag:pull-min-sk}
% https://q.uiver.app/?q=WzAsNCxbMCwwLCJcXHNrdChcXHNDLFxcc1MpIl0sWzEsMCwiXFxjYXQoXFxjQyxcXGNTKSJdLFsxLDEsIltJLFxcY2F0XShcXCxXLFxcY2F0KFVfc0QtLFxcY1MpXFwsKSJdLFswLDEsIltJLFxcY2F0XShcXCxXLFxcc2t0KEQtLFxcc1MpXFwsKSJdLFswLDEsIlVfcyJdLFszLDIsIlVfc1xcY2lyYy0iLDJdLFsxLDIsIlxcUHNpIl0sWzAsMywiXFxQc2lfcyIsMl1d
\begin{tikzcd}[ampersand replacement=\&]
	{\skt(\sC,\sS)} \& {\cat(\cC,\cS)} \\
	{[I,\cat](\,W,\skt(D-,\sS)\,)} \& {[I,\cat](\,W,\cat(U_sD-,\cS)\,)}
	\arrow["{U_s}", from=1-1, to=1-2]
	\arrow["{U_s\circ-}"', from=2-1, to=2-2]
	\arrow["\Psi", from=1-2, to=2-2]
	\arrow["{\Psi_s}"', from=1-1, to=2-1]
\end{tikzcd}
\end{equation}
\end{prop}

\begin{proof}
Let us consider a category $\cX$ together with functors $Q$ and $P$ as below, making the outer diagram below commutative. 
% https://q.uiver.app/?q=WzAsNSxbMSwxLCJcXHNrdChcXHNDLFxcc1MpIl0sWzIsMSwiXFxjYXQoXFxjQyxcXGNTKSJdLFsyLDIsIltJLFxcY2F0XShcXCxXLFxcY2F0KFVfc0QtLFxcY1MpXFwsKSJdLFsxLDIsIltJLFxcY2F0XShcXCxXLFxcc2t0KEQtLFxcc1MpXFwsKSJdLFswLDAsIlxcY1giXSxbMCwxLCJVX3MiXSxbMywyLCJVX3NcXGNpcmMtIiwyXSxbMSwyLCJcXFBzaSJdLFswLDMsIlxcUHNpX3MiLDJdLFs0LDMsIlAiLDIseyJjdXJ2ZSI6Mn1dLFs0LDEsIlEiLDAseyJjdXJ2ZSI6LTJ9XSxbNCwwLCJcXGV4aXN0cyEgUSciLDEseyJzdHlsZSI6eyJib2R5Ijp7Im5hbWUiOiJkYXNoZWQifX19XV0=
\[\begin{tikzcd}[ampersand replacement=\&]
	\cX \\
	\& {\skt(\sC,\sS)} \& {\cat(\cC,\cS)} \\
	\& {[\cI,\cat](\,W,\skt(D-,\sS)\,)} \& {[\cI,\cat](\,W,\cat(U_sD-,\cS)\,)}
	\arrow["{U_s}", from=2-2, to=2-3]
	\arrow["{U_s\circ-}"', from=3-2, to=3-3]
	\arrow["\Psi", from=2-3, to=3-3]
	\arrow["{\Psi_s}"', from=2-2, to=3-2]
	\arrow["P"', curve={height=12pt}, from=1-1, to=3-2]
	\arrow["Q", curve={height=-12pt}, from=1-1, to=2-3]
	\arrow["{\exists! Q'}"{description}, dashed, from=1-1, to=2-2]
\end{tikzcd}\]
Similarly to \Cref{prop:almost-w-lim} it is enought to show that, for any $x\in\cX$, the functors $Qx\colon \cC\to\cS$ are actually morphisms of sketches from $\sC$ to $\sS$. Thus, let us consider a $\sigma\in\skL_\mathfrak{C}^m/\skC_\mathfrak{C}^m$, i.e. by definition $\sigma=q_w^i(\tau)$ for some $i\in \cI,w\in W(i)$ and $\tau\in\skL_i/\skC_i$\footnote{Where $Di=(\cD_i,\skL_i,\skC_i)$. }. Therefore, 
$$Qx(\sigma)=(Qx)q_w^i(\tau)=\Psi Qx(\tau)=[U_s(Px)_i(w)](\tau),$$
where the second equality holds by definition of $\Psi$ and the third by commutative of the outer diagram above. Since $(Px)_i(w)\in\skt(Di,\sS)$, there exists a $\delta\in\skL_\sS/\skC_\sS$ such that $(Px)_i(w)(\tau)\cong\delta$. Finally we can conclude $Qx(\sigma)\cong\delta$ as required. 
\end{proof}

\begin{rem}
    It is worth underlining that, since there are no size issues, all of these results hold also for $\Skt$ and $\Cat$.  
\end{rem}

% \red
% \begin{center}
%     % https://q.uiver.app/?q=WzAsNCxbMCwwLCJNb2RfXFxjWChcXHNDKSJdLFsxLDAsIk1vZF9cXGNYKFxcc1NfaSkiXSxbMSwxLCJbXFxjU19pLFxcY1hdIl0sWzAsMSwiW1xcY0MsXFxjWF0iXSxbMCwxLCItXFxjaXJjIEZfaSJdLFszLDIsIi1cXGNpcmMgRl9pIiwyXSxbMSwyXSxbMCwzXV0=
% \begin{tikzcd}[ampersand replacement=\&]
% 	{Mod_\cX(\sC)} \& {Mod_\cX(\sS_i)} \\
% 	{[\cC,\cX]} \& {[\cS_i,\cX]}
% 	\arrow["{-\circ F_i}", from=1-1, to=1-2]
% 	\arrow["{-\circ F_i}"', from=2-1, to=2-2]
% 	\arrow[from=1-2, to=2-2]
% 	\arrow[from=1-1, to=2-1]
% \end{tikzcd}
% \hspace{1cm}
% % https://q.uiver.app/?q=WzAsNCxbMSwxLCJNb2RfXFxjWChcXHNDKSJdLFswLDAsIk1vZF9cXGNYKFxcc1NfaSkiXSxbMSwwLCJbXFxjU19pLFxcY1hdIl0sWzIsMCwiW1xcY0MsXFxjWF0iXSxbMiwzLCItXFxjaXJjIEdfaSJdLFswLDMsIiIsMix7InN0eWxlIjp7InRhaWwiOnsibmFtZSI6Imhvb2siLCJzaWRlIjoidG9wIn19fV0sWzEsMCwiIiwyLHsic3R5bGUiOnsiaGVhZCI6eyJuYW1lIjoiZXBpIn19fV0sWzEsMl1d
% \begin{tikzcd}[ampersand replacement=\&]
% 	{Mod_\cX(\sS_i)} \& {[\cS_i,\cX]} \& {[\cC,\cX]} \\
% 	\& {Mod_\cX(\sC)}
% 	\arrow["{-\circ G_i}", from=1-2, to=1-3]
% 	\arrow[hook, from=2-2, to=1-3]
% 	\arrow[two heads, from=1-1, to=2-2]
% 	\arrow[from=1-1, to=1-2]
% \end{tikzcd}
% \end{center}
% \black

\subsection{Weighted bi-co/limits of sketches}
\label{sec:w-p-co/lim}
The literature contains already some results regarding limits/colimits in $\skt$.  
In \cite{lair1975etude} it is shown that the 1-category $\skt_s$ of sketches and strict morphisms (not up-to-iso) is locally finitely presentable, hence has all (1-d) limits and colimits. 
 Moreover, it is also already known that the 2-category $\skt$ has all weighted bilimits \cite[Chapter~5.1, Proposition 5.1.4]{Makkaipare}. In this section we show that $\skt/\Skt$  have weighted pseudo-co/limits (hence also all weighted bi-co/limits). We will also show that $\skt/\Skt$ have flexible co/limits (see \Cref{skt-has-flex-co/lim}). 

\begin{thm}[\textit{Existence and construction of weighted pseudo-limits of sketches}]
\label{thm:skt-has-p-lim}
\hspace{0cm}\\
 Let $W\colon \cI\to\cat$ be a weight and $D\colon \cI\to\skt$ a diagram. The $W$-weighted $D$-pseudo-limit $lim_W^pD$ in $\skt$ exists and, moreover, it is given by the sketch 
 $$\text{lim}_W^pD=(\text{lim}_W^pU_sD,\skL_p,\skC_p)$$
 where $\text{lim}_W^pU_sD$ is the $W$-weighted $U_sD$-pseudo-limit in $\cat$ and $(\skL_p,\skC_p)$ are the maximal classes making all the projections of the pseudo-limit $\text{lim}_W^pU_sD$ morphisms of sketches. 
\end{thm}

\begin{proof}
This follows immediately setting $\cC:=\text{lim}_W^pU_sD$ in \Cref{prop:almost-w-lim}. In fact, being a $W$-weighted $U_sD$-pseudo-limit means exactly that the functor $\Phi$ in diagram \eqref{diag:pull-max-sk} is an isomorphisms. Thus, since isomorphisms are stable under pullback, also $\Phi_s$ has to be an isomorphisms, which concludes that  $\sC$ is a $W$-weighted $D$-pseudo-limit.
\end{proof}

\begin{rem}
Since a 2-category with all pseudo-limits also has all bi-limits (see for instance \cite[Section~6.12]{Lack2010}), \Cref{thm:skt-has-p-lim} implies that $\skt$ has all weighted bi-limits as well. 
\end{rem}

The following proposition gives a formula to calculate weighted colimits in $\skt$ using weighted colimits in $\cat$, which is the analogous to \Cref{thm:skt-has-p-lim} for colimits. 

\begin{thm}[\textit{Existence and construction of weighted pseudo-colimits of sketches}] \label{thm:skt-has-p-colim}
\hspace{0cm}\\
Let $W\colon \cI\to\cat$ be a weight and $D\colon \cI\to\skt$ a diagram.  The $W$-weighted $D$-pseudo-colimit in $\skt$ is the sketch $(\text{colim}_W^pU_sD,\skL^p,\skC^p)$ where $\text{colim}_W^pU_sD$ is the $W$-weighted $U_sD$-pseudo-colimit in $\cat$, $(\skL^p,\skC^p)$ is the smallest sketch structure making all the (weighted) inclusions maps of sketches. 
\end{thm}

\begin{proof}
    The proof is analogous to the one of \Cref{thm:skt-has-p-lim} using \Cref{prop:almost-w-colim}.
\end{proof}

\begin{cor}
    The formulas in Theorems~\ref{thm:skt-has-p-lim} and \ref{thm:skt-has-p-colim}, applied to locally small sketches, give the constructions of weighted pseudo-limits and colimits in the 2-category $\Skt$ as well. 
\end{cor}

\begin{proof}
    It suffices to notice that Propositions~\ref{prop:min-max},~\ref{prop:almost-w-lim} and \ref{prop:almost-w-colim} do not require the categories to be small. 
\end{proof}

\begin{rem}[Co/Limit Sketches are closed under bi-co/limits]\label{rem:co/lim-skt-closed-bi-co/lim}
    It is clear from the constructions given in Theorems~\ref{thm:skt-has-p-lim} and \ref{thm:skt-has-p-colim}, that if we start with co/cones in limit sketches, the resulting pseudo-co/limit sketch is still a limit sketch. 
    In particular, this gives us the formula for such pseudo-co/limits in $\skt_l/\Skt_l$. 
    Analogously, the same formulas give also the weighted pseudo-co/limits in $\skt_c/\Skt_c$ as well. 
    
    In other words, the forgetful functors $U_l$ and $U_c$ creates weigthed pseudo-co/limits. 
\end{rem}

\subsubsection{Flexible co/lmits}

In this section we will use the results in Sections \ref{sec:top-beh} and \ref{sec:w-p-co/lim} to show that $\skt/\Skt$ have all flexible co/limits. 

\begin{cor}\label{skt-has-flex-co/lim}
   The 2-categories $\skt$ and $\Skt$ admit all flexible co/limits.
\end{cor}

\begin{proof}
    Thanks to \cite[Proposition~4.8]{BKPS:flex-lim} to show the existence of flexible limits, it suffices to show that we have (weighted) pseudo limits and splitting of idempotent equivalences.
    If we show the dual hypothesis of \cite[Proposition~4.8]{BKPS:flex-lim}, then we would get also the existence of flexible colimits. 
    In particular, these would be the existence of (weighted) pseudo colimits and the dual of splitting of idempotent equivalences. 
    Luckily, the dual of the existence of splitting of idempotent equivalences is still itself. 
    Therefore, since we have the existence of both (weighted) pseudo limits (see \Cref{thm:skt-has-p-lim}) and colimits (see \Cref{thm:skt-has-p-colim}), we only have to show that idempotent equivalences in $\skt/\Skt$ admit a splitting. 

    Let $E\colon\sS\to\sS$ be an idempotent equivalence in $\skt/\Skt$, i.e. a morphism of sketches $E$ such that $EE=E$ and there exists an invertible 2-cell $\theta\colon E\to 1_\sS$ such that $\theta E=1_E=E\theta$. 
    Since $\cat/\Cat$ has all strict limits, in particular we can consider the splitting of $E\colon\cS\to\cS$ in $\cat/\Cat$ below, i.e. two functors $R\colon\cS\to\cX$ and $I\colon\cX\to\cS$ such that $E=IR$ and $RI=1_\cX$. 
    % https://q.uiver.app/#q=WzAsMyxbMCwwLCJcXGNTIl0sWzIsMCwiXFxjUyJdLFsxLDEsIlxcY1giXSxbMCwxLCJFIl0sWzAsMiwiUiIsMl0sWzIsMSwiSSIsMl1d
\[\begin{tikzcd}[ampersand replacement=\&]
	\cS \&\& \cS \\
	\& \cX
	\arrow["E", from=1-1, to=1-3]
	\arrow["R"', from=1-1, to=2-2]
	\arrow["I"', from=2-2, to=1-3]
\end{tikzcd}\]
Now we need to endow $\cX$ with a sketch structure making both $R$ and $I$ sketch morphisms. The required structure is the minimal sketch structure making $R$ a sketch morphism as defined in \Cref{prop:min-max}.
$$\sX:=(\cX,\skL^m_R,\skC^m_R)$$
By the second part of point (1) of \Cref{prop:min-max}, $I\colon\cX\to\cS$ is a sketch morphism if and only if $IR$ is such. But $IR=E$, thus $I$ is also a sketch morphism and $(\sX,R,I)$ is a splitting of the idempotent equivalence $E$ in $\skt/\Skt$.
\end{proof}

\subsubsection{The case of pseudo-powers}\label{sec:powers}
Let us give a concrete presentation of weighted pseudo-limits in some special cases, this descriptions will come handy later in the paper when we describe the sketch structure of exponential sketches. Let $\cI$ be a (small) category, and let $\sS$ be a sketch. Then, recall that the theory developed in the previous subsection tells us that $\sS^\cI$ is equipped with the maximal sketch structure making all the projections, morphisms of sketches \[\sS^\cI \to \sS.\] 

% \red
% \begin{prop} \label{prop:powertest}
% Let $\sM$ be a test sketch and $\cI$ be a small category. Then the bipower $\sM^\cI$ is test.
% \end{prop}
% \begin{proof}[Sketch of proof]
% The test structure clearly makes the projections morphisms of sketches, and it is clearly maximal among those with this property.
% \end{proof}
% \black 

% \subsection{Lifting of Factorisation Systems}

% % \red
% % \begin{prop}\label{prop:lift-of-fact-syst}
% %     The 2-category $\Skt$ (also $\skt$) admits a \emph{lifting} of the bo-ff factorisation system from $\Cat$ ($\cat$). 

% %     Ivan: Any factorisation in $\Cat$/$\cat$ lifts, true because any factorisation lifts through a topological functor. 
% % \end{prop}
% % \black 

\subsection{Formal aspects of sketches}
\subsubsection{Duality involution in the $2$-category of sketches} \label{dualityinvolution}

Now, we introduce a duality notion for sketches. This will be useful in \Cref{sec:left-skt-class}, where we will use it to describe the underlying category of the left sketch classifier (see \Cref{turnaroundsketch}). 

\begin{defn}[Opposite sketch]
    Given a sketch $\sS$, we define its opposite sketches $\sS^\circ$ as follows. The underlying category of $\sS^\circ$ is simply $\cS^\op$, and the sketch structure inverts the cone specifications with the colimit ones.
\end{defn}

\begin{rem}
It is easy to see that this construction provides a duality involution on the $2$-category of sketches in the sense of Shulman, see \cite{shulman2016contravariance},

\[(-)^\circ \colon \Skt^\mathsf{co} \to \Skt.\]
\end{rem}

\begin{rem}
It follows directly from the definition that the duality involution interacts with the forgetful functors towards limit and colimit sketches as displayed in the equations below.

\[U^l(\sS^\circ) = (U^c \sS)^\circ \quad \quad U^c(\sS^\circ) = (U^l \sS)^\circ\]
\end{rem}

\subsubsection{On the behavior of Kan extensions in $\Skt$} \label{nicekankan}

The following result about Kan extensions will be useful in \Cref{sec:left-sketches}, in particular in \Cref{prop:left-skt-then-kan-inj}, where we compare our notion of left sketch (see \Cref{def:left-sketch}) with a more classical cocompleteness condition. 

\begin{prop}
    Let $F\colon\sC\to\sS$ and $G\colon\sC\to\sS$ be morphisms of sketches. Then, if the dashed Kan extension below right exists and is a morphism of sketches, then it can be chosen as a Kan extension in $\Skt$.
% https://q.uiver.app/#q=WzAsNixbMCwwLCJcXHNDIl0sWzAsMSwiXFxzRCJdLFsxLDAsIlxcc1MiXSxbMiwwLCJVXFxzQyJdLFsyLDEsIlVcXHNEIl0sWzMsMCwiVVxcc1MiXSxbMCwyLCJGIl0sWzAsMSwiRyIsMl0sWzMsNSwiVUYiXSxbMyw0LCJVRyIsMl0sWzQsNSwiXFx0ZXh0e2xhbn1fe1VHfVVGIiwyLHsic3R5bGUiOnsiYm9keSI6eyJuYW1lIjoiZGFzaGVkIn19fV0sWzEsMiwiIiwyLHsic3R5bGUiOnsiYm9keSI6eyJuYW1lIjoiZG90dGVkIn19fV1d
\[\begin{tikzcd}[ampersand replacement=\&]
	\sC \& \sS \& U\sC \& U\sS \\
	\sD \&\& U\sD
	\arrow["F", from=1-1, to=1-2]
	\arrow["G"', from=1-1, to=2-1]
	\arrow["UF", from=1-3, to=1-4]
	\arrow["UG"', from=1-3, to=2-3]
	\arrow["{\text{lan}_{UG}UF}"', dashed, from=2-3, to=1-4]
	\arrow[dotted, from=2-1, to=1-2]
\end{tikzcd}\]
\end{prop}
\begin{proof}
    This follows directly from the fact that the forgetful functor $U\colon\Skt\to\Cat$ is locally fully faithful.
\end{proof}

\subsubsection{An open question}
In \cite{di2022bi} it is shown that every $2$-category in \Cref{doctrines} is locally finitely bipresentable. It is an interesting question whether the $2$-category $\skt$ is locally finitely bipresentable, since this information would have some quite important consequences on its behavior. There is some evidence (\cite{lair1975etude})  that this result could be true, indeed the $1$-category of sketches is locally presentable. Notice that $\Skt$ is not expected to be locally finitely bipresentable because of some size issues.

\section{Tensor product and exponential sketches} 
\label{sec:richard}
% \red [Do we want to add Lair as a ref?] \black 
In this section we recall the monoidal structure on the $2$-category of sketches studied by Benson \cite{benson1997multilinearity} and show that is actually closed. Early done by Lair in \cite{lair1975etude} on the ($1$-)category of sketches is nowadays hard to access, yet \cite{wells1994sketches} reports he had proven the ($1$-)category to be monoidal closed.
%Our treatment is somewhat original, but the general results are folklore or discussed in some dedicate papers, such as \cite{benson1997multilinearity}. 
%The construction of the tensor product we present is due to Benson \cite{benson1997multilinearity}, while our contribution is the closedness \yellow and the \emph{topological} approach we give, thanks to \Cref{sec:top-beh} \red [Technically Benson writes the sentence "take the minimal structure making those morphisms" but he did not \emph{formalise} the notion]. \black 
Our contribution is also to reformulate some parts of this treatment using the \emph{topological} approach of \Cref{sec:top-beh}. 
Some of these ideas were already present in Benson's treatment\footnote{After \cite[Definition~2.12]{benson1997multilinearity} it is mentioned that the structure of the tensor product is indeed the \emph{coarsest} making some maps sketch morphisms.}, but not formalised.    

\begin{constr}[\textit{Exponential sketch}]\label{costr:exp-skt}
    Let $\sD= (\cD, \skL_\sD, \skC_\sD)$ be a small sketch and $\sT= (\cT, \skL_\sT, \skC_\sT)$ be a sketch. We shall now construct the sketch $\sT^\sD$ as follows. We consider the functor \[j\colon  \mathsf{Skt}(\sD, \sT) \hookrightarrow \sT^{\cD},\] where the codomain is just the pseudo-power sketch of $\sT$ along the underlying category of $\sD$ as discussed in \Cref{sec:powers}. Now we apply \Cref{prop:min-max} to equip $\mathsf{Skt}(\sD, \sT)$ with the maximal sketch structure making $j$ a sketch morphism.  
\end{constr}

 The idea is that the specifications in the exponential $\sT^\sD$ consists of family of sketch morphisms $F_i\colon\sD\to\sT$ which point-wise form a cone specification in $\sT$, i.e. for any $d\in\sD$, the $F_i(d)$ form a cone in $\skL_\sT$. For the cocone specification is completely analogous. We underline that the sketch structure of the domain $\sD$ comes into play in the condition that the cone of morphism must be \emph{sketch morphisms} and not just functors.

\begin{constr}[\textit{Tensor product of sketches}]\label{def:benson-tensor}
 Let $\sD$ and  $\sS$ be small sketches, now we will recall the definition of $\sD \boxtimes \sS$ \cite[Definition~2.12]{benson1997multilinearity}. The underlying category of $\sD \boxtimes \sS$ is the product $\cD\times\cS$, which then we equip with the smallest sketch structure which makes the canonical maps $\mathcal{I}$ morphisms of sketches. 
 \begin{center}
     $\mathcal{I}:=I_\cD\cup I_\cS$ with \\
     $I_\cD:=\lbrace i_d\colon\cS\to\cD\times\cS\mid d\in\cD\rbrace$ and
     $I_\cS:=\lbrace i_s\colon\cD\to\cD\times\cS\mid s\in\cS\rbrace$.
 \end{center}
 Hence, with the notation introduced in $\Cref{prop:min-max}$, we have 
 $$\sD\boxtimes\sS=(\cD\times\cS,\skL^m_\mathcal{I},\skC^m_{\mathcal{I}}).$$
\end{constr}

Benson also provides some (commuting) conditions on the sketches $\sD$ and $\sS$ \cite[Definition~2.10]{benson1997multilinearity} in order to achieve a sort of closedness property (for some categories with \emph{enough co/limits}) \cite[Definition2.14 and Theorem~3.1]{benson1997multilinearity}. In the following Theorem, we show that the sketch structure defined in \Cref{costr:exp-skt} provides a refinement of this result, proving that Benson's tensor product $\boxtimes$ is closed. 
% Using the terminology of this paper, we can reformulate this property as follows: for any sketches $\sS$ and $\sT$, any test sketch $\sM$ \yellow \emph{nice enough} (i.e. such that $\Skt(\sT,\sM)$ is co/complete),
% %\red nice enough [non vuol dire un cazzo ma non so come esprimere la proprietà che loro usano, i.e. Mod(T,K) ha tutti i co/limiti che vuoi] 
% \black then there is an isomorphism as below, where we consider $\Skt(\sT,\sM)$ as a test sketch. 
% $$\Skt(\sS\boxtimes\sT,\sM)\cong\Skt(\,\sS,\Skt(\sT,\sM)\,),$$
\begin{thm}\label{thm:skt-closed-mon}
There is an isomorphism of categories as below,
\[\Skt(\sD \boxtimes \sS, \sT) \cong \Skt(\sS, \sT^\sD ). \]
\end{thm}
\begin{proof}
 Using the canonical sketch morphism $\sT^\sD \to \sT^\cD$, it is clear that we have an inclusion $\Skt(\sS, \sT^\sD ) \hookrightarrow  \Skt(\sS, \sT^\cD )$. Moreover, using the universal property of the pseudo-power $\sS^\cD$, the inclusion $\Skt(\sS,\sT)\hookrightarrow\Cat(\cS,\cT)$ and the cartesian closedness of $\Cat$ we obtain the following composite morphism. 
% $$\Skt(\sS, \sT^\cD )\cong \Cat(\cD, \Skt(\sS,\sT)  )$$
 \begin{equation*}\label{eq:incl-int-hom}
     \Skt(\sS, \sT^\sD ) \hookrightarrow\Skt(\sS, \sT^\cD )\cong \Cat(\cD, \mathsf{Skt}(\sS,\sT)  )\hookrightarrow\Cat(\cD, \Cat(\cS,\cT)  )\cong\Cat(\cD \times \cS, \cT  )
 \end{equation*}

 Clearly, by definition of $\boxtimes$, we also have $\Skt(\sD\boxtimes\sS,\sT)\hookrightarrow\Cat(\cD\times\cS,\cT)$. Hence, we can prove that the wanted isomorphism holds by checking that the image of these two inclusions are the same.
% https://q.uiver.app/#q=WzAsNCxbMCwxLCJcXENhdChcXGNELCBcXENhdChcXGNTLFxcY1QpICApIl0sWzEsMSwiXFxDYXQoXFxjRCBcXHRpbWVzIFxcY1MsIFxcY1QgICkiXSxbMCwwLCJcXFNrdChcXHNTLCBcXHNUXlxcc0QgKSJdLFsxLDAsIlxcU2t0KFxcc0RcXGJveHRpbWVzXFxzUyxcXHNUKSJdLFswLDEsIlxcY29uZyIsMl0sWzIsMCwiSSIsMix7InN0eWxlIjp7InRhaWwiOnsibmFtZSI6Imhvb2siLCJzaWRlIjoiYm90dG9tIn19fV0sWzMsMSwiSiIsMCx7InN0eWxlIjp7InRhaWwiOnsibmFtZSI6Imhvb2siLCJzaWRlIjoidG9wIn19fV1d
\[\begin{tikzcd}[ampersand replacement=\&]
	{\Skt(\sS, \sT^\sD )} \& {\Skt(\sD\boxtimes\sS,\sT)} \\
	{\Cat(\cD, \Cat(\cS,\cT)  )} \& {\Cat(\cD \times \cS, \cT  )}
	\arrow["\cong"', from=2-1, to=2-2]
	\arrow["I"', hook', from=1-1, to=2-1]
	\arrow["J", hook, from=1-2, to=2-2]
\end{tikzcd}\]

First, let us notice that a functor $F\colon\cD\to\Cat(\cS,\cT)$ is in the image of $I$ if and only if: \black \footnote{Condition (1) says that $F$ is in the image of the second inclusion of \eqref{eq:incl-int-hom}, while (2) requires it to be in the image of the first one.}
\begin{enumerate}
    \item it factors through $\Skt(\sS,\sT)$, i.e. there exists $F'\colon\cD\to\Skt(\sS,\sT)$ such that $U_sF'\cong F$; 
    \item for any $s\in\cS$, post-composition of $F$ with the evaluation functor $ev_s\colon\Cat(\cS,\cT)\to\cT$ is a sketch morphism $ev_sF$ in $\Skt(\sD,\sT)$, which is the underlying category of the sketch $\sT^\sD$.
\end{enumerate}
% https://q.uiver.app/#q=WzAsNCxbMCwwLCJcXGNEIl0sWzIsMCwiXFxDYXQoXFxjUyxcXGNUKSJdLFsxLDEsIlxcU2t0KFxcc1MsXFxzVCkiXSxbMywwLCJcXGNUIl0sWzAsMSwiRiJdLFsxLDMsImV2X3MiXSxbMiwxLCJVX1MiLDIseyJzdHlsZSI6eyJ0YWlsIjp7Im5hbWUiOiJob29rIiwic2lkZSI6InRvcCJ9fX1dLFswLDIsIkYnIiwyXV0=
\[\begin{tikzcd}[ampersand replacement=\&]
	\cD \&\& {\Cat(\cS,\cT)} \& \cT \\
	\& {\Skt(\sS,\sT)}
	\arrow["F", from=1-1, to=1-3]
	\arrow["{ev_s}", from=1-3, to=1-4]
	\arrow["{U_S}"', hook, from=2-2, to=1-3]
	\arrow["{F'}"', from=1-1, to=2-2]
\end{tikzcd}\]

On the other hand, a morphism $G\colon\cD\times\cS\to\cT$ is a sketch morphism if and only if, for any $j\in\mathcal{I}$ of \Cref{def:benson-tensor}, $Gj$ is such (since $\cD\boxtimes\sS$ has the minimal sketch structure defined in \Cref{prop:min-max}). More precisely we need:
\begin{enumerate}[(i)]
    \item for any $d\in\cD$, $Gi_d=G(d,-)\colon\cS\to\cT$ must be a sketch morphism;
    \item for any $s\in\cS$, $Gi_s=G(-,s)\colon\cD\to\cT$ must be a sketch morphism.
\end{enumerate}

In order to conclude, we notice that condition (1) corresponds to (i) and (2) to (ii) under the isomorphism $\Cat(\cD,\Cat(\cS,\cT))\cong\Cat(\cD\times\cS,\cT)$.

\end{proof}

% \red 
% \begin{rem}
%     Notice that Benson non ha capito un cazzo. Cita il suo risultato su commutative sketches, bla bla. 
% \end{rem}
% \black 

% \begin{defn}[Commutative Sketches]
%     \red Let $\sS$ be a sketch and let us denote with $\sS^l:=U^l\sS$ and $\sS^c:=U^c\sS$ its limit and colimit part. We say that $\sS$ is commutative if, for any test sketch $\sM$, we have the following equivalence.
%     $$Mod(\sS^c,Mod(\sS^l,\sM))\simeq Mod(\sS^l,Mod(\sS^c,\sM))$$
%     [MEMO: We have not define models...right? at least not introduced the notation]
% \end{defn}

% \red This is the same as saying that $\sS^l$ is $\cM$-compatible with $\sS^c$ for any co/complete category $\cM$ (in the sense of Benson). \black 

% \begin{prop}
%     [\textit{Forse falsa}] $U^l\sS$ is $\sM$-compatible with $U^c\sS$ (in the sense of Benson) if and only if $\sM^\sS$ is normal. \red [I'm saying this just because models in a limit sketch are usually better behaved than mixed sketches]\black 
% \end{prop}
% \begin{defn}
%     A sketch $\sD$ is polished\footnote{polished is a bad choice of name, as the semantic field of tidy-ness has a precise meaning in topos theory. We just need something that gives a stronger vibe than rounded.} if, for every test sketch $\sS$, $\sS^\sD$ is normal \red (and thus test) [Actually is this true? Not sure if it's co/complete]. \black 
% \end{defn}

% \begin{prop}
% Right normal and left normal sketches are polished.
% \end{prop}

% \begin{prop}
% If a sketch is polished, its opposite is also.
% \end{prop}

\section{Left sketches}\label{sec:left-sketches}
In this section, we introduce and study the notion of \textit{left sketch}, which will play a crucial role in our definition of \textit{logos}. A motivating example for the definition of left sketch is the very notion of topos. Throughout the section we will often prove a general result in the context of left sketches and contextualize its meaning in the special case of topoi.

\begin{defn}[\textit{Left Sketch}]\label{def:left-sketch}
A left sketch is a sketch $\sS= (\cS, \skL_\sS, \skC_\sS)$ where: 
\begin{enumerate}
    \item $\cS$ is cocomplete;
    \item $\skC_\sS$ contains all the (essentally small) colimits diagrams.
\end{enumerate}
We denote with $\leftskt$/$\leftSkt$ the subcategories of $\skt$/$\Skt$ with objects left sketches. 
\end{defn}

% \begin{rem}
% For the rest of the discussion, there will be two natural forgetful functors to study, which we write only in the case of small sketches.

% % https://q.uiver.app/#q=WzAsMyxbMCwwLCJcXGxlZnRza3QiXSxbMSwxLCJcXHNrdCJdLFsxLDIsIlxcbGltc2t0Il0sWzAsMSwiVV5MIl0sWzEsMiwiVV9sIl0sWzAsMiwiVV5MX2wiLDIseyJjdXJ2ZSI6Mn1dXQ==
% \[\begin{tikzcd}
% 	\leftskt \\
% 	& \skt \\
% 	& \limskt
% 	\arrow["{U^L}", from=1-1, to=2-2]
% 	\arrow["{U_l}", from=2-2, to=3-2]
% 	\arrow["{U^L_l}"', curve={height=12pt}, from=1-1, to=3-2]
% \end{tikzcd}\]
% \end{rem}

\begin{rem}[Caveat: Right sketches exist too!] \label{caveatrighutskt}
It comes with no surprise that one can give the definition of \textit{right sketch}, and that every statement dualises very easily. One neat way to organize the discussion is to say that a sketch is \textit{right} if its dual sketch is left. From this point on, we will not insist on the notion of right sketches, as there are many more examples of left sketches in nature, and thus we focus on those. Our treatment though is perfectly dualisable.
\end{rem}

\begin{exa}[A pet example of left sketch] \label{topostoleft}
The inclusion of topoi into sketches discussed in \Cref{firsttimetopos} factors through left sketches. For this section, and for the rest of the paper, topoi will be our pet example of left sketch.

% https://q.uiver.app/#q=WzAsMyxbMCwwLCJcXG1hdGhzZntUb3BvaX1eXFx0ZXh0e29wfSJdLFswLDEsIlxcbWF0aHNme1NrdH0iXSxbMSwxLCJcXExTa3QiXSxbMCwxXSxbMiwxXSxbMCwyLCIiLDIseyJzdHlsZSI6eyJib2R5Ijp7Im5hbWUiOiJkYXNoZWQifX19XV0=
\[\begin{tikzcd}[ampersand replacement=\&]
	{\mathsf{Topoi}^\text{op}} \\
	{\mathsf{Skt}} \& \leftSkt
	\arrow[from=1-1, to=2-1]
	\arrow[from=2-2, to=2-1]
	\arrow[dashed, from=1-1, to=2-2]
\end{tikzcd}\]

Similarly to the case of topoi, frames admit also a canonical structure of left sketch. Thus left sketches offer a framework in which both topoi and locales coexist, without any need of the usual topos-completion of a locale in order to relate these two notions.
\end{exa}

\begin{exa}[$\lambda$-Topoi à la Espíndola as left sketches]\label{kappaleft}
The notion of $\lambda$-topos (where $\lambda$ is a regular cardinal) was introduced by Espíndola in \cite{espindola2020infinitary} to encompass infinitary variations of geometric logic. Not much is known about the $2$-category of $\lambda$-topoi, mostly because a $\lambda$-topos is not \textit{simply} a $\lambda$-lex localization of a presheaf topos  (see example 3.4 in \cite{espindola2023every}). The most updated account on the topic can be found in the very recent \cite{espindola2023every}.

For us $\mathsf{Topoi}_\lambda$ will be the $2$-category whose objects are $\lambda$-topoi in the sense of \cite[Definition~3.1]{espindola2023every}, morphisms are geometric morphisms whose left adjoint preserve $\lambda$-small limits and $2$-cells are natural transformations between left adjoints.

Of course every $\lambda$-topos has a left sketch structure, similar to that of a topos, where in the limit class we have all $\lambda$-small limit cones. Altogether, this information provides us with the diagram below.

% https://q.uiver.app/#q=WzAsNixbMCwwLCJcXG1hdGhzZntUb3BvaX1ee1xcdGV4dHtvcH19Il0sWzIsMCwiLi4uIl0sWzMsMCwiXFxtYXRoc2Z7VG9wb2l9X3tcXGxhbWJkYX1ee1xcdGV4dHtvcH19Il0sWzIsMSwiXFxMZWZ0c2t0Il0sWzQsMCwiLi4uIl0sWzEsMCwiXFxtYXRoc2Z7VG9wb2l9X3tcXGFsZXBoXzF9XntcXHRleHR7b3B9fSJdLFsyLDFdLFs0LDJdLFs1LDBdLFsxLDVdLFswLDMsIiIsMCx7ImN1cnZlIjoxfV0sWzUsM10sWzEsM10sWzIsM10sWzQsMywiIiwwLHsiY3VydmUiOi0xfV0sWzEwLDExLCIiLDAseyJzaG9ydGVuIjp7InNvdXJjZSI6NDAsInRhcmdldCI6NDB9fV0sWzExLDEyLCIiLDAseyJzaG9ydGVuIjp7InNvdXJjZSI6NDAsInRhcmdldCI6NDB9fV0sWzEyLDEzLCIiLDAseyJzaG9ydGVuIjp7InNvdXJjZSI6NDAsInRhcmdldCI6NDB9fV0sWzEzLDE0LCIiLDAseyJzaG9ydGVuIjp7InNvdXJjZSI6NDAsInRhcmdldCI6NDB9fV1d
\[\begin{tikzcd}[ampersand replacement=\&]
	{\mathsf{Topoi}^{\text{op}}} \& {\mathsf{Topoi}_{\aleph_1}^{\text{op}}} \& {...} \& {\mathsf{Topoi}_{\lambda}^{\text{op}}} \& {...} \\
	\&\& \leftSkt
	\arrow[from=1-4, to=1-3]
	\arrow[from=1-5, to=1-4]
	\arrow[from=1-2, to=1-1]
	\arrow[from=1-3, to=1-2]
	\arrow[""{name=0, anchor=center, inner sep=0}, curve={height=6pt}, from=1-1, to=2-3]
	\arrow[""{name=1, anchor=center, inner sep=0}, from=1-2, to=2-3]
	\arrow[""{name=2, anchor=center, inner sep=0}, from=1-3, to=2-3]
	\arrow[""{name=3, anchor=center, inner sep=0}, from=1-4, to=2-3]
	\arrow[""{name=4, anchor=center, inner sep=0}, curve={height=-6pt}, from=1-5, to=2-3]
	\arrow[shorten <=9pt, shorten >=9pt, Rightarrow, from=0, to=1]
	\arrow[shorten <=9pt, shorten >=9pt, Rightarrow, from=1, to=2]
	\arrow[shorten <=9pt, shorten >=9pt, Rightarrow, from=2, to=3]
	\arrow[shorten <=9pt, shorten >=9pt, Rightarrow, from=3, to=4]
\end{tikzcd}\]

The horizontal $2$-functors are simply acknowledging that when $\kappa \leq \lambda$, a $\lambda$-topos is a $\kappa$-topos. Notice though that these functors are not full on $1$-cells. The vertical arrows into $\leftSkt$ instead are locally fully faithful.
\end{exa}

The terminology \textit{left} is inspired from \cite{isbell1960adequate}, where the author uses the adjective \textit{left} to refer to density and colimit-generation. Indeed, the definition of left sketch consists of a \emph{cocompleteness} condition, so it makes sense to wonder whether it can be expressed through the existence of Kan extensions, or using the terminology introduced in \cite[Definition~1.2]{dlsolo} whether left sketches can be described as Kan injectives in $\Skt$. 

\begin{defn}[$j_\sI\colon \sI \to \sI_!$]
Let $\cI$ be any category, and consider the discrete sketch structure $\sI:= (\cI, \emptyset, \emptyset)$. Consider also $\cI_!$ the category that freely adds a terminal object to $\cI$ and the obvious inclusion \[j_\cI\colon \cI \to \cI_!.\] We equip $\cI_!$ with a sketch structure $\sI_!$, whose limit part is empty and whose colimit part contains the free cocone over $j_\cI$. Of course $j_\cI=:j_\sI\colon \sI \to \sI_!$ is a sketch map.
\end{defn}

\begin{prop}\label{prop:left-skt-then-kan-inj}
A left sketch $\sS= (\cS, \skL_\sS, \skC_\sS)$ is left Kan injective in $\Skt$ with respect to all the maps $j_\sI$ as defined above, for any small category $\cI$. Moreover, 
$$\leftSkt\subseteq\linj(\lbrace j_\sI\mid \cI\;\textrm{small category}\rbrace).$$
\end{prop}

\begin{proof}
Let us consider a small categoy $\cI$ and the diagram below, where we allow the abuse of notation of using the same letter for the maps of sketches and their underlying functors. 
% https://q.uiver.app/#q=WzAsNixbMCwwLCJcXHNJIl0sWzAsMSwiXFxzSV8hIl0sWzEsMCwiXFxzUyJdLFsyLDAsIlxcY0kiXSxbMiwxLCJcXGNJXyEiXSxbMywwLCJcXGNTIl0sWzAsMiwiRCJdLFswLDEsImpfXFxzSSIsMl0sWzMsNCwial9cXHNJIiwyXSxbNCw1LCJcXHRleHR7bGFufV97al9cXHNJfWYiLDIseyJzdHlsZSI6eyJib2R5Ijp7Im5hbWUiOiJkYXNoZWQifX19XSxbMSwyLCIiLDIseyJzdHlsZSI6eyJib2R5Ijp7Im5hbWUiOiJkb3R0ZWQifX19XSxbMyw1LCJEIl1d
\[\begin{tikzcd}[ampersand replacement=\&]
	\sI \& \sS \& \cI \& \cS \\
	{\sI_!} \&\& {\cI_!}
	\arrow["D", from=1-1, to=1-2]
	\arrow["{j_\sI}"', from=1-1, to=2-1]
	\arrow["{j_\sI}"', from=1-3, to=2-3]
	\arrow["{\text{lan}_{j_\sI}f}"', dashed, from=2-3, to=1-4]
	\arrow[dotted, from=2-1, to=1-2]
	\arrow["D", from=1-3, to=1-4]
\end{tikzcd}\]
Since $\cS$ is cocomplete (condition (1) of \Cref{def:left-sketch}) we know that $\text{lan}_{j_\sI}D$ exists in $\Cat$. Then, the requirement that $\text{lan}_{j_\sI}D$ is a morphism of sketches is just a reformulation of the condition (2) of \Cref{def:left-sketch}. It follows from \Cref{nicekankan} that  $\sS$ is left Kan injective with respect to $j_\sI$.

 We need to prove that a morphism of left sketches $F\colon\sS\to\sT$ is a morphisms of left Kan injectives. 
It is enough to notice that the Kan extensions $\text{lan}_{j_\sI}D$ calculate colimits, hence preserving the Kan extensions along $j_\sI$ (i.e. being a morphism of left Kan injectives) is the same as preserving colimits. 
On the other hand, a morphism of left sketches must preserve colimits, hence it is a morphism of left Kan injectives. 

Regarding 2-cells is enough to recall that $\Skt$ and $\Cat$ have the same 2-cells.  
\end{proof}

%Now, in the opposite direction, when $\cS$ is a left sketch, we want to show that (1) and (2) are verified. Now, notice that a sketch map $\mathcal{D}_! \to \sS$ (making the triangle commutative) is the same of the specification of a cocone for $f$, which happens to be in $\skC$. {\color{red} eh qui è un problema se non so concludere}

\begin{rem}
It is natural to wonder whether the other inclusion is also true. While we do not have a precise counterexample, by inspecting the proof of the previous proposition, a positive answer does not seem very likely. Yet, a brighter understanding of the situation would necessitate a broader development of the theory of Kan extensions in the $2$-category of Sketches.  
\end{rem}

\subsection{A parametric take on Morita theory}\label{sec:par-morita-theory}

In \Cref{moritaprimavolta} we have introduced the notion of test Morita equivalence of sketches. The general idea is that being test Morita equivalent means that the two sketches have the same model among test sketches. Now, of course, test sketches are a very interesting and somewhat preferred notion of sketch, especially in the literature, but other classes of sketches may be taken into consideration in place of test. For example, one possible interpretation of the original definition of Morita equivalence of sites means precisely that the two sites have the same model in every topos, so that topoi are taken as a notion of test.

\begin{defn}[$\mathbb{A}$-Morita equivalence]\label{def:strong-morita-equiv} \label{parametric morita}
Let  $\mathbb{A}$ be a full subcategory of the category of $\Skt$. We say that a morphism of sketches $F\colon \sS \to \sT$ is a $\mathbb{A}$-Morita equivalence if, for every $\mathcal{L}$ in $\mathbb{A}$, the induced map between hom-categories is an equivalence of categories,\[F^*\colon  \Skt(\sT, \mathcal{L}) \to \Skt(\sS, \mathcal{L}).\]
When $\mathbb{A}$ is $\leftSkt$, we may refer to this notion as left Morita equivalence, and of course when $\mathbb{A}$ is the class of test sketches, this notion collapses to \Cref{def:morita-equiv} given in \Cref{moritaprimavolta}.
\end{defn}

\begin{rem}\label{rmk:strong-Mor-impl-Mor-eq}
Because every test sketch is a left sketch, it follows on the spot that a left Morita equivalence is a test Morita equivalence.
\end{rem}

\begin{notat}[Morita small = \textit{left} Morita small] \label{safetyMorita}
For the rest of the paper, when we decorate a $2$-category with the apex $\mathsf{M}$ (as in $\mathsf{Skt}^{\mathsf{M}}$, for example), or we refer to Morita small sketches, we mean \textit{left} Morita small. We shall discuss case by case when our results generalize to the case of test Morita small sketches.
\end{notat}

\subsection{Left sketch classifiers}\label{sec:left-skt-class}

As we mentioned in the introduction and later re-discussed in \Cref{diac1}, the starting point of the theory of classifying topoi is, with no doubt, the so-called Diaconescu theorem, that establishes an adjunction between sites and topoi.
This subsection is devoted to constructing a \textit{free} left sketch associated to a (Morita) small sketch. This will provide us with a proto-form of Diaconescu-type theorem (see \Cref{diacweak}), establishing that (Morita small) left sketches are bi-reflective among (Morita small) sketches.
%\red
%\[\hat{(-)}: \mathsf{Skt}^{\mathsf{M}} \leftrightarrows \mathsf{LSkt}^{\mathsf{M}}: U^{L}.\]
%\black
% https://q.uiver.app/#q=WzAsMixbMCwwLCJcXFNrdF57XFxtYXRoc2Z7TX19Il0sWzEsMCwiXFxsZWZ0U2t0XntcXG1hdGhzZntNfX0iXSxbMSwwLCJVXntMfSIsMix7Im9mZnNldCI6MSwiY3VydmUiOjF9XSxbMCwxLCJcXGhhdHsoLSl9IiwyLHsib2Zmc2V0IjoxLCJjdXJ2ZSI6MX1dLFsyLDMsIlxcdG9wIiwxLHsic2hvcnRlbiI6eyJzb3VyY2UiOjIwLCJ0YXJnZXQiOjIwfSwic3R5bGUiOnsiYm9keSI6eyJuYW1lIjoibm9uZSJ9LCJoZWFkIjp7Im5hbWUiOiJub25lIn19fV1d
\[\begin{tikzcd}[ampersand replacement=\&]
	{\Skt^{\mathsf{M}}} \& {\leftSkt^{\mathsf{M}}}
	\arrow[""{name=0, anchor=center, inner sep=0}, "{U^{L}}"', shift right, curve={height=6pt}, from=1-2, to=1-1]
	\arrow[""{name=1, anchor=center, inner sep=0}, "{\hat{(-)}}"', shift right, curve={height=6pt}, from=1-1, to=1-2]
	\arrow["\top"{description}, draw=none, from=0, to=1]
\end{tikzcd}\] 
Let us start from providing the construction of the \textit{left sketch classifier} in the special case of a small sketch, and later we will adjust the construction for the Morita small-case.

\begin{constr}[Left sketch classifier]
\label{costr:left-skt-class}
For every small sketch  $\sS= (\cS, \skL_\sS, \skC_\sS)$ we shall now construct a left sketch $\hat{\sS}$ together sketch morphism $J_\sS\colon  \sS \to \hat{\sS}$ which we will show that induces a Morita equivalence. To start, consider the Yoneda embedding \[\yo\colon \cS \to \Psh(\cS).\] Now, consider the class of colimit diagrams $\skC_\sS$ and let $\delta\colon d \Rightarrow \Delta(s)$ be a cocone in $\skC_\sS$. By the universal property of the colimit in $\Psh(\cS)$, we can build a natural transformation $\rho_\delta \colon \text{colim} \yo d \Rightarrow \yo s$. Call $\mathcal{H}$ the set of maps of the form $\rho_\delta$ for $\delta$ in $\skC_\sS$ and define 
\[\hat{\cS} = \mathcal{H}^\perp.\]
This is a reflective subcategory of the presheaf category (by \cite[1.36]{adamek_rosicky_1994}), and thus it is both complete and cocomplete. Moreover, we get a functor $J_\cS\colon\cS\to\hat{\cS}$ as shown in the diagram below by composing $L \circ \yo$.  
% https://q.uiver.app/#q=WzAsMyxbMCwwLCJcXGNTIl0sWzEsMCwiXFxtYXRoc2Z7UHNofShcXGNTKSJdLFsxLDEsIlxcaGF0e1xcY1N9Il0sWzAsMSwiXFx5byJdLFsyLDEsIkkiLDIseyJjdXJ2ZSI6Miwic3R5bGUiOnsidGFpbCI6eyJuYW1lIjoiaG9vayIsInNpZGUiOiJib3R0b20ifX19XSxbMSwyLCJMIiwyLHsiY3VydmUiOjJ9XSxbMCwyLCJKX1xcY1MiLDIseyJzdHlsZSI6eyJib2R5Ijp7Im5hbWUiOiJkYXNoZWQifX19XSxbNSw0LCJcXGRhc2h2IiwxLHsic2hvcnRlbiI6eyJzb3VyY2UiOjIwLCJ0YXJnZXQiOjIwfSwic3R5bGUiOnsiYm9keSI6eyJuYW1lIjoibm9uZSJ9LCJoZWFkIjp7Im5hbWUiOiJub25lIn19fV1d
\[\begin{tikzcd}[ampersand replacement=\&]
	\cS \& {\Psh(\cS)} \\
	\& {\hat{\cS}}
	\arrow["\yo", from=1-1, to=1-2]
	\arrow[""{name=0, anchor=center, inner sep=0}, "I"', curve={height=12pt}, hook', from=2-2, to=1-2]
	\arrow[""{name=1, anchor=center, inner sep=0}, "L"', curve={height=12pt}, from=1-2, to=2-2]
	\arrow["{J_\cS}"', dashed, from=1-1, to=2-2]
	\arrow["\dashv"{description}, draw=none, from=1, to=0]
\end{tikzcd}\]

We complete this construction equipping $\hat{\cS}$ with a (left) sketch structure.
\begin{itemize}
    \item We define $\skL_{\hat{\sS}}$ as the set of all $J_\cS\sigma$ for any $\sigma\in\skL_{\sS}$.  Using the notation of \Cref{prop:min-max}, $\skL_{\hat{\sS}}:=\skL^m_{J_{\cS}}$.  %\footnote{\red Do we want all $\tau$ such that $\exists\delta\in\skL_\sS$ with $\tau\cong J\delta$?\black Nope because the iso are in the morphisms.}
    
    \item $\skC_{\hat{\sS}}$ is defined as all the (essentially small) colimits diagrams. 
\end{itemize}

Below we check that this definition makes $J$ a map of sketches $J_\sS:\sS\to\hat{\sS}$. 
\begin{itemize}
    \item For $\sigma\in\skL_{\sS}$, then $J\sigma=L\yo\sigma\in\skL_{\hat{\sS}}$. 
    \item For $\delta\colon d\Rightarrow\Delta(s)$ in $\skC_\sS$, we need to show that $J\delta=L\yo\delta$ is a colimit diagram. By definition, $L$ sends $\rho_\delta:\text{colim}\yo\delta\to \yo s$ to an isomorphism $L\text{colim}\yo\delta\cong L\yo s$. Moreover, since $L$ is a left adjoint $\text{colim}L\yo\delta\cong L\text{colim}\yo\delta$. Hence, $J\delta=L\yo\delta$ is a colimit diagram.
\end{itemize}
We call $\hat{\sS}$ the left sketch classifier of $\sS$. 
\end{constr}

\begin{exa}[Left sketch classifiers vs classifying topoi] \label{leftsketchclasstopos}
Let $(C,J)$ be a site, and consider their canonical sketch structure as discussed in \Cref{sites}. Then, the discussion in \cite[D2.1.4(h)]{Sketches} proves that \Cref{costr:left-skt-class}  produces the topos of sheaves over $(C,J)$.

% https://q.uiver.app/#q=WzAsNCxbMCwwLCJcXG1hdGhzZntzaXRlfSJdLFsxLDAsIlxcbWF0aHNme1RvcG9pfV5cXGNpcmMiXSxbMCwxLCJcXHNrdCJdLFsxLDEsIlxcbGVmdFNrdCJdLFswLDEsIlxcbWF0aHNme1NofSJdLFswLDIsImkiLDJdLFsyLDMsIihcXGhhdHstfSkiLDIseyJzdHlsZSI6eyJib2R5Ijp7Im5hbWUiOiJzcXVpZ2dseSJ9fX1dLFsxLDNdXQ==
\[\begin{tikzcd}
	{\mathsf{site}} & {\mathsf{Topoi}^\circ} \\
	\skt & \leftSkt
	\arrow["{\mathsf{Sh}}", from=1-1, to=1-2]
	\arrow["i"', from=1-1, to=2-1]
	\arrow["{(\hat{-})}"', squiggly, from=2-1, to=2-2]
	\arrow[from=1-2, to=2-2]
\end{tikzcd}\]

This motivates our choice of name for the \Cref{costr:left-skt-class}. At this stage, we haven't discussed the functoriality of the left sketch classifier, and we shall postpone such discussion to a later point in the section. %The proposition below recovers and generalizes a classical behavior of the topos of sheaves over a site. 
\end{exa}

The following proposition is an analogue of \cite[Example~D2.1.4~(h)]{Sketches} in the sketch-context, which recovers and generalises a classical behaviour of the topos of sheaves over a site. 

\begin{prop}\label{turnaroundsketch}
Let $\sS$ be a sketch. Then $\hat{\cS}$ coincides with $\mathsf{Skt}((U^c\sS)^\circ,\Set)$.  
\end{prop} 

\begin{proof}
This result essentially follows from \cite[1.33(7) and (8)]{adamek_rosicky_1994}, but we shall spell it out in detail for the sake of clarity. Let us start by recalling the sketch structure of $(U^c\sS)^\circ$.  Note that $(U^c\sS)$ is the same of $\cS$, where we forgot the limit structure. $(U^c\sS)^\circ$ is the sketch structure on $\cS^\circ$ that uses as a cones precisely the colimit cones of $(U^c\sS)$. So, a sketch morphism  $(U^c\sS)^\circ \to \Set$ is a presheaf over $\cS$ mapping the tip of the cocone into the limit of the diagram.
Now, let's go back to \Cref{costr:left-skt-class} and look at the orthogonality condition.

% https://q.uiver.app/#q=WzAsNCxbMSwxLCJcXHRleHR7Y29saW19IFxceW8gZCJdLFsxLDIsIlxceW8gcyJdLFsyLDEsIlAiXSxbMCwwLCJcXHlvIGQiXSxbMCwxLCJcXHJob19cXGRlbHRhIiwyXSxbMywwLCJpX2QiXSxbMywxLCIiLDIseyJjdXJ2ZSI6Miwic3R5bGUiOnsiYm9keSI6eyJuYW1lIjoiZGFzaGVkIn19fV0sWzAsMl0sWzMsMiwiIiwyLHsiY3VydmUiOi0yLCJzdHlsZSI6eyJib2R5Ijp7Im5hbWUiOiJkYXNoZWQifX19XV0=
\[\begin{tikzcd}
	{\yo d} \\
	& {\text{colim} \yo d} & P \\
	& {\yo s}
	\arrow["{\rho_\delta}"', from=2-2, to=3-2]
	\arrow["{i_d}", from=1-1, to=2-2]
	\arrow[curve={height=12pt}, dashed, from=1-1, to=3-2]
	\arrow[from=2-2, to=2-3]
	\arrow[curve={height=-12pt}, dashed, from=1-1, to=2-3]
\end{tikzcd}\]

Now, every representable appearing in the colimit comes equipped with a map $i_d$ as in the diagram above, and thus we can rewrite the orthogonality condition as in the equation below,
\[\lim P(d) \stackrel{\text{Yoneda}}{\cong} \Set^{{\cS}^{\circ}}(\text{colim} \yo d, P) \stackrel{\text{orth.}}{\cong} \Set^{{\cC}^{\circ}}(\yo s, P) \stackrel{\text{Yoneda}}{\cong} P(s).\]
This means precisely that a presheaf $P$ is orthogonal $\rho_\delta$ if and only if it maps $s$ into the limit of the diagram specified by the cocone, which is the thesis.

% https://q.uiver.app/#q=WzAsNCxbMCwwLCJOYXQoXFx0ZXh0e2NvbGltfSB5ZF9pLCBQKSJdLFsxLDAsIlxcdGV4dHtsaW19TmF0KHlkX2ksIFApIl0sWzIsMCwiXFx0ZXh0e2xpbX1QKGRfaSkiXSxbMywwLCJQKHMpIl0sWzAsMV0sWzEsMl0sWzIsM11d
\end{proof}

\begin{prop}[Universal property of the left classifier]\label{prop:rel-adj-left-class}
    Let $\sS$ be a small sketch and $J_\sS\colon\sS\to\hat{\sS}$ the sketch morphism defined in \Cref{costr:left-skt-class}. Then, $J_\sS$ is a left Morita equivalence, i.e for any left sketch $\sM$, $J_\sS$ induces induces an equivalence
    $$J_\sS^\ast\colon\Skt(\hat{\sS},\sM)\to\Skt(\sS,\sM).$$ 
\end{prop}

\begin{proof}
    We start showing that $J_\sS^\ast$ is essentially surjective. Let $F\colon\sS\to\sM$ be a morphism of sketches, by the universal property of $\Psh(\cS)$, since $\cM$ is cocomplete, there is a unique (up-to-iso) cocontinous functor $\tilde{F}$ as below. 
% https://q.uiver.app/#q=WzAsNCxbMCwwLCJcXGNTIl0sWzEsMCwiXFxQc2goXFxjUykiXSxbMSwxLCJcXGNNIl0sWzIsMCwiXFxoYXR7XFxjU30iXSxbMCwxLCJcXHlvIl0sWzEsMiwiXFx0aWxkZXtGfSIsMV0sWzEsMywiTCJdLFszLDIsIlxcaGF0e0Z9IiwwLHsic3R5bGUiOnsiYm9keSI6eyJuYW1lIjoiZGFzaGVkIn19fV0sWzAsMiwiRiIsMl0sWzgsNSwiXFxjb25nIiwxLHsic2hvcnRlbiI6eyJzb3VyY2UiOjIwLCJ0YXJnZXQiOjIwfSwic3R5bGUiOnsiYm9keSI6eyJuYW1lIjoibm9uZSJ9LCJoZWFkIjp7Im5hbWUiOiJub25lIn19fV1d
\[\begin{tikzcd}[ampersand replacement=\&]
	\cS \& {\Psh(\cS)} \& {\hat{\cS}} \\
	\& \cM
	\arrow["\yo", from=1-1, to=1-2]
	\arrow[""{name=0, anchor=center, inner sep=0}, "{\tilde{F}}"{description}, from=1-2, to=2-2]
	\arrow["L", from=1-2, to=1-3]
	\arrow["{\hat{F}}", dashed, from=1-3, to=2-2]
	\arrow[""{name=1, anchor=center, inner sep=0}, "F"', from=1-1, to=2-2]
	\arrow["\cong"{description}, draw=none, from=1, to=0]
\end{tikzcd}\]
    Now, recall that following \cite[5.4.10]{borceux_1994}, $\hat{\cS}=\mathcal{H}^\perp$ can be understood as a localisation $\Psh(\cS)[\mathcal{E}_\mathcal{H}^{-1}]$ where $\mathcal{E}_\mathcal{H}$ is a certain  closure of $\mathcal{H}$ under colimits. Because $\tilde{F}$ is cocontinuous, it follows that to check whether it extends to $\Psh(\cS)[\mathcal{E}_\mathcal{H}^{-1}]$, it is enough to check whether  $\tilde{F}$ sends all maps of $\mathcal{H}$ to isomorphisms, then by the universal property of the localisation $\Psh(\cS)[\mathcal{E}_\mathcal{H}^{-1}]$ we would get a unique functor $\hat{F}$ making the triangle above right commutative. This is true exactly if, for any cocone $\delta\in\skC_\sS$, $\tilde{F}\yo\delta$ is a colimit diagram, which is true because $\tilde{F}\yo\delta\cong F\delta$ and $F$ is a morphism of sketches. Of course it follows from the general theory of localizations, say \cite[5.3.1]{borceux_1994} that $\hat{F} \cong \tilde{F} \circ I$, with $I$ being the inclusion presented in the previous construction.
% Let us consider  $\delta:d\Rightarrow\Delta(s)\in\skC_\sS$ and corresponding $\rho_\delta: \text{colim}\yo d\Rightarrow \yo s\in\mathcal{H}$. We notice that, since $\tilde{F}$ is cocontinous, preserves colimits and in particular $\tilde{F}\text{colim}\yo d\cong\text{colim}\tilde{F}\yo d=\text{colim}F d$
%$\rho_\delta\in\mathcal{H}$ $\tilde{F}\rho_\delta$
We will conclude the proof that $J_\sS^\ast$ is essentially surjective showing that $\hat{F}$ is a morphism of sketches. 
\begin{itemize}
    \item The cones in $\skL_{\hat{S}}$ are all the $J_{\cS}\sigma$ for $\sigma\in\skL_{\sS}$. Then $\hat{F}J_{\cS}\sigma\cong F\sigma$ which is a limit diagram because $F$ is a map of sketches. Alternatevely, it follows by the second part of \Cref{prop:min-max} since $\hat{F}$ preserves cones in $\skL_{\hat{S}}$ if and only if $\hat{F}J$ does.
    \item $\hat{F}$ is a left adjoint. This follows from that fact that $\tilde{F}$ is such and using \cite[Lemma 2.7]{ADLL:aft-kz-mon} with, following their notation, $L$ as $l$, $I$ as $r$ and $\tilde{F}$ as $l'$. 
\end{itemize}

Now let us show that $J^\ast_\sS$ is fully faithful. Clearly the diagram below commutes. 
% https://q.uiver.app/#q=WzAsNSxbMCwwLCJcXFNrdChcXGhhdHtcXHNTfSxcXHNNKSJdLFsyLDAsIlxcU2t0KFxcaGF0e1xcc1N9LFxcc00pIl0sWzIsMSwiXFxDYXQoXFxoYXR7XFxjU30sXFxjTSkiXSxbMCwxLCJcXHRleHR7Q29Db250fShcXGhhdHtcXGNTfSxcXGNNKSJdLFsxLDEsIlxcdGV4dHtDb0NvbnR9KFxcUHNoKFxcY1MpLFxcY00pIl0sWzEsMiwiZi5mLiJdLFswLDMsImYuZi4iLDJdLFs0LDIsIlxceW9eXFxhc3QiLDJdLFszLDQsIkxeXFxhc3QiLDJdLFswLDEsIkpfXFxzU15cXGFzdCJdXQ==
\[\begin{tikzcd}[ampersand replacement=\&]
	{\Skt(\hat{\sS},\sM)} \&\& {\Skt(\hat{\sS},\sM)} \\
	{\text{CoCont}(\hat{\cS},\cM)} \& {\text{CoCont}(\Psh(\cS),\cM)} \& {\Cat(\hat{\cS},\cM)}
	\arrow["{f.f.}", from=1-3, to=2-3]
	\arrow["{f.f.}"', from=1-1, to=2-1]
	\arrow["{\yo^\ast}"', from=2-2, to=2-3]
	\arrow["{L^\ast}"', from=2-1, to=2-2]
	\arrow["{J_\sS^\ast}", from=1-1, to=1-3]
\end{tikzcd}\]
% \begin{itemize}
Since $\yo^\ast$ is an equivalence of categories, is enough to show that $L^\ast$ is fully faithful to prove that $J_\sS^\ast$ is such. We recall that $L$ is the left adjoint of the fully faithful functor $I\colon\hat{\cS}\hookrightarrow\Psh(\cS)$. This imply directly that $L^\ast$ is fully faithful. Indeed, \cite[Proposition~I,6.4]{Gray_formal_ct_adj} tells us that $I^\ast\dashv L^\ast$ with counit invertible (since its components are the one of the counit of $L\dashv I$) and so $L^\ast$ is fully faithful. 
%     \item \textbf{Faifhful:} Let $\beta,\beta':G\Rightarrow G'$
%     \item \textbf{Full:}
% \end{itemize}
\end{proof}

\begin{constr}[Functoriality of $\hat{(-)}$]  \label{functoriality}
Let $F\colon \sS \to \sT$ be a morphism of small sketches. In the diagram below, consider the adjunction $F_! \dashv F^*$, where $F^*$ is the precomposition functor, induced via the usual Kan extension.
 
% https://q.uiver.app/#q=WzAsNixbMSwwLCJcXGNTIl0sWzEsMiwiXFxtYXRoc2Z7UHNofShcXGNTKSJdLFswLDEsIlxcaGF0e1xcY1N9Il0sWzMsMCwiXFxjVCJdLFszLDIsIlxcbWF0aHNme1BzaH0oXFxjVCkiXSxbNCwxLCJcXGhhdHtcXGNUfSJdLFswLDEsIlxceW9fe1xcY1N9Il0sWzIsMSwiSV97XFxjU30iLDIseyJjdXJ2ZSI6Miwic3R5bGUiOnsidGFpbCI6eyJuYW1lIjoiaG9vayIsInNpZGUiOiJib3R0b20ifX19XSxbMSwyLCJMX3tcXGNTfSIsMSx7ImN1cnZlIjoyfV0sWzAsMiwiSl9cXGNTIiwyXSxbMCwzLCJGIl0sWzMsNCwiXFx5b197XFxjVH0iLDJdLFszLDUsIkpfXFxjVCJdLFs0LDUsIkxfe1xcY1R9IiwxLHsiY3VydmUiOi0yfV0sWzUsNCwiSV9cXGNUIiwwLHsiY3VydmUiOi0yLCJzdHlsZSI6eyJ0YWlsIjp7Im5hbWUiOiJob29rIiwic2lkZSI6InRvcCJ9fX1dLFs0LDEsIkZeKiIsMSx7ImN1cnZlIjotMywic3R5bGUiOnsiYm9keSI6eyJuYW1lIjoiZGFzaGVkIn19fV0sWzEsNCwiRl8hIiwxLHsic3R5bGUiOnsiYm9keSI6eyJuYW1lIjoiZGFzaGVkIn19fV0sWzgsNywiIiwxLHsibGV2ZWwiOjEsInN0eWxlIjp7Im5hbWUiOiJhZGp1bmN0aW9uIn19XSxbMTMsMTQsIiIsMSx7ImxldmVsIjoxLCJzdHlsZSI6eyJuYW1lIjoiYWRqdW5jdGlvbiJ9fV0sWzE2LDE1LCIiLDIseyJsZXZlbCI6MSwic3R5bGUiOnsibmFtZSI6ImFkanVuY3Rpb24ifX1dLFs5LDYsIjo9IiwzLHsic2hvcnRlbiI6eyJzb3VyY2UiOjIwLCJ0YXJnZXQiOjIwfSwic3R5bGUiOnsiYm9keSI6eyJuYW1lIjoibm9uZSJ9LCJoZWFkIjp7Im5hbWUiOiJub25lIn19fV0sWzEyLDExLCI6PSIsMyx7InNob3J0ZW4iOnsic291cmNlIjoyMCwidGFyZ2V0IjoyMH0sInN0eWxlIjp7ImJvZHkiOnsibmFtZSI6Im5vbmUifSwiaGVhZCI6eyJuYW1lIjoibm9uZSJ9fX1dLFs2LDExLCJcXGNvbmciLDMseyJzaG9ydGVuIjp7InNvdXJjZSI6MjAsInRhcmdldCI6MjB9LCJzdHlsZSI6eyJib2R5Ijp7Im5hbWUiOiJub25lIn0sImhlYWQiOnsibmFtZSI6Im5vbmUifX19XV0=
\[\begin{tikzcd}[ampersand replacement=\&]
	\& \cS \&\& \cT \\
	{\hat{\cS}} \&\&\&\& {\hat{\cT}} \\
	\& {\mathsf{Psh}(\cS)} \&\& {\mathsf{Psh}(\cT)}
	\arrow[""{name=0, anchor=center, inner sep=0}, "{\yo_{\cS}}", from=1-2, to=3-2]
	\arrow[""{name=1, anchor=center, inner sep=0}, "{I_{\cS}}"', curve={height=12pt}, hook', from=2-1, to=3-2]
	\arrow[""{name=2, anchor=center, inner sep=0}, "{L_{\cS}}"{description}, curve={height=12pt}, from=3-2, to=2-1]
	\arrow[""{name=3, anchor=center, inner sep=0}, "{J_\cS}"', from=1-2, to=2-1]
	\arrow["F", from=1-2, to=1-4]
	\arrow[""{name=4, anchor=center, inner sep=0}, "{\yo_{\cT}}"', from=1-4, to=3-4]
	\arrow[""{name=5, anchor=center, inner sep=0}, "{J_\cT}", from=1-4, to=2-5]
	\arrow[""{name=6, anchor=center, inner sep=0}, "{L_{\cT}}"{description}, curve={height=-12pt}, from=3-4, to=2-5]
	\arrow[""{name=7, anchor=center, inner sep=0}, "{I_\cT}", curve={height=-12pt}, hook, from=2-5, to=3-4]
	\arrow[""{name=8, anchor=center, inner sep=0}, "{F^*}"{description}, curve={height=-18pt}, dashed, from=3-4, to=3-2]
	\arrow[""{name=9, anchor=center, inner sep=0}, "{F_!}"{description}, dashed, from=3-2, to=3-4]
	\arrow["\dashv"{anchor=center, rotate=-139}, draw=none, from=2, to=1]
	\arrow["\dashv"{anchor=center, rotate=-41}, draw=none, from=6, to=7]
	\arrow["\dashv"{anchor=center, rotate=-89}, draw=none, from=9, to=8]
	\arrow["{:=}"{marking, allow upside down}, draw=none, from=3, to=0]
	\arrow["{:=}"{marking, allow upside down}, draw=none, from=5, to=4]
	\arrow["\cong"{marking, allow upside down}, draw=none, from=0, to=4]
\end{tikzcd}\]
 
Also, consider the functor $L_{\mathbb{T}}F_!$. If we show that $L_{\mathbb{T}}F_!$ inverts all the maps in $\mathcal{H}^{\sS}$, then applying the universal property of the category of fractions (see the proof of \Cref{prop:rel-adj-left-class}) we would get a left adjoint, as in the diagram below. This follows from the fact that $F_!$ preserve colimits, the middle square in the diagram above is (naturally) pseudo-commutative and $L_\cT$ inverts the maps in $\mathcal{H}^{\sT}$. 
Indeed, for any $\delta\in\skL_\sS$ we consider $\rho_\delta \colon \text{colim} \yo d \Rightarrow \yo_\cS s$. Then, we apply $F_!(\rho)$ and get 
$$F_!\text{colim} \yo_\cS d\cong  \text{colim}F_!\yo_cS d\cong \text{colim}\yo_\cT F d\Rightarrow F_!\yo_\cS s\cong \yo_\cT Fs.$$
We notice that, since $F$ is a sketch morphism, the morphism above is in $\mathcal{H}^{\sT}$, hence $L_\cT$ inverts it. 
% https://q.uiver.app/#q=WzAsNixbMCwwLCJcXGNTIl0sWzAsMiwiXFxtYXRoc2Z7UHNofShcXGNTKSJdLFsxLDEsIlxcaGF0e1xcY1N9Il0sWzMsMCwiXFxjVCJdLFszLDIsIlxcbWF0aHNme1BzaH0oXFxjVCkiXSxbMiwxLCJcXGhhdHtcXGNUfSJdLFswLDEsIlxceW9fe1xcY1N9IiwyXSxbMSwyLCJMX3tcXGNTfSIsMV0sWzAsMiwiSl9cXGNTIiwyXSxbMCwzLCJGIl0sWzMsNCwiXFx5b197XFxjU30iXSxbMyw1LCJKX1xcY1QiXSxbNCw1LCJMX3tcXGNUfSIsMV0sWzEsNCwiRl8hIiwxXSxbMiw1LCJcXGhhdHtGXyF9IiwwLHsic3R5bGUiOnsiYm9keSI6eyJuYW1lIjoiZGFzaGVkIn19fV0sWzksMTQsIlxcY29uZyIsMSx7InNob3J0ZW4iOnsic291cmNlIjoyMCwidGFyZ2V0IjoyMH0sInN0eWxlIjp7ImJvZHkiOnsibmFtZSI6Im5vbmUifSwiaGVhZCI6eyJuYW1lIjoibm9uZSJ9fX1dLFsxMSw0LCI6PSIsMyx7InNob3J0ZW4iOnsic291cmNlIjoyMH0sInN0eWxlIjp7ImJvZHkiOnsibmFtZSI6Im5vbmUifSwiaGVhZCI6eyJuYW1lIjoibm9uZSJ9fX1dLFs4LDEsIjo9IiwzLHsic2hvcnRlbiI6eyJzb3VyY2UiOjIwfSwic3R5bGUiOnsiYm9keSI6eyJuYW1lIjoibm9uZSJ9LCJoZWFkIjp7Im5hbWUiOiJub25lIn19fV1d
\[\begin{tikzcd}[ampersand replacement=\&]
	\cS \&\&\& \cT \\
	\& {\hat{\cS}} \& {\hat{\cT}} \\
	{\mathsf{Psh}(\cS)} \&\&\& {\mathsf{Psh}(\cT)}
	\arrow["{\yo_{\cS}}"', from=1-1, to=3-1]
	\arrow["{L_{\cS}}"{description}, from=3-1, to=2-2]
	\arrow[""{name=0, anchor=center, inner sep=0}, "{J_\cS}"', from=1-1, to=2-2]
	\arrow[""{name=1, anchor=center, inner sep=0}, "F", from=1-1, to=1-4]
	\arrow["{\yo_{\cS}}", from=1-4, to=3-4]
	\arrow[""{name=2, anchor=center, inner sep=0}, "{J_\cT}", from=1-4, to=2-3]
	\arrow["{L_{\cT}}"{description}, from=3-4, to=2-3]
	\arrow["{F_!}"{description}, from=3-1, to=3-4]
	\arrow[""{name=3, anchor=center, inner sep=0}, "{\hat{F_!}}", dashed, from=2-2, to=2-3]
	\arrow["\cong"{description}, draw=none, from=1, to=3]
	\arrow["{:=}"{marking, allow upside down}, draw=none, from=2, to=3-4]
	\arrow["{:=}"{marking, allow upside down}, draw=none, from=0, to=3-1]
\end{tikzcd}\]
Now let us describe some structural properties of $\hat{F_!}$.  Since $\hat{F_!}$ is defined through the universal property of the localisation $\hat{\cS}$, then $\hat{F_!}L_\cS=L_\cT F_!$. Moreover, since $\hat{\cS}$ is reflective, $L_\cS I_\cS\cong 1$, and so $\hat{F_!}\cong L_{\cT}F_!I_{\cS}$. It is easy to see that it is left adjoint to its \textit{conjugate} $L_{\cS}F^*I_{\cT}$ (similarly to \Cref{prop:rel-adj-left-class} this is again true by \cite[Lemma 2.7]{ADLL:aft-kz-mon}). Since it is a left adjoint, it is cocontinuous.

Because $\hat{F_!}$ is cocontinuous, it preserves the left part of the sketch structure between $\hat{\sS}$ and $\hat{\sT}$, and it follows by \Cref{prop:min-max} that it also preserves the limit part (because $\hat{F_!} J_\sS\cong J_\sT F$ does).
\end{constr}

\begin{prop}\label{prop:funct-of-hat}
    The assigment $\widehat{(-)}$ is a pseudofunctor $\widehat{(-)}\colon\skt\to\leftSkt^{\mathsf{M}}$. 
\end{prop}

\begin{proof}
    The action of $\widehat{(-)}$ on objects is defined in \Cref{costr:left-skt-class}, on morphisms in \Cref{functoriality} and on 2-cells is the identity. The pseudofunctoriality follows from pseudofunctoricality of $\Psh(-)=(-)_!$ and the universal property of localisation. 
\end{proof}

% \begin{cor}\label{cor:left-norm-morita-eq}
%     For any small sketch $\sS$, the sketch morphism $J_\sS\colon\sS\to\hat{\sS}$ defined in \Cref{costr:left-skt-class} is a left Morita equivalence. 
% \end{cor}
% \begin{proof}
% This can be taken as a rewording of  \Cref{prop:rel-adj-left-class}.
% \end{proof}

\subsection{Diaconescu for left sketches}
\begin{rem}
    Putting together \Cref{costr:left-skt-class} and \Cref{functoriality}, we obtain a pseudofunctor (left adjoints are only identified up to isomorphisms),     mapping a small sketch to its left classifier, and mapping a morphism of sketches to its free extension.
    \[\widehat{(-)}\colon \skt \to \mathsf{L}\Skt^{\mathsf{M}} \]
    Moreover, the content of \Cref{prop:rel-adj-left-class} is hinting to the fact  that, in the left hand side of the diagram below, the $\widehat{(-)}$-construction  should be  a relative left-biadjoint to $j$, the inclusions of Morita small left sketches into Morita-small sketches. %serves as

% https://q.uiver.app/#q=WzAsNixbMiwwLCJcXHNrdCJdLFszLDAsIlxcbGVmdFNrdF5cXG1hdGhzZntNfSJdLFsyLDEsIlxcU2t0XntcXG1hdGhzZntNfX0iXSxbMCwwLCJcXHNrdCJdLFsxLDAsIlxcbGVmdFNrdF57XFxtYXRoc2Z7TX19Il0sWzAsMSwiXFxTa3Ree1xcbWF0aHNme019fSJdLFswLDEsIihcXGhhdHstfSkiXSxbMCwyXSxbNCw1LCJqIl0sWzMsNSwiaSIsMl0sWzMsNCwiKFxcaGF0ey19KSJdLFsyLDEsIiIsMix7ImN1cnZlIjozLCJzdHlsZSI6eyJib2R5Ijp7Im5hbWUiOiJkb3R0ZWQifX19XSxbMSwyLCJqIiwxXSxbMTAsOCwiIiwyLHsibGV2ZWwiOjEsInN0eWxlIjp7Im5hbWUiOiJhZGp1bmN0aW9uIn19XV0=
\[\begin{tikzcd}
	\skt & {\leftSkt^{\mathsf{M}}} & \skt & {\leftSkt^\mathsf{M}} \\
	{\Skt^{\mathsf{M}}} && {\Skt^{\mathsf{M}}}
	\arrow["{(\hat{-})}", from=1-3, to=1-4]
	\arrow[from=1-3, to=2-3]
	\arrow[""{name=0, anchor=center, inner sep=0}, "j", from=1-2, to=2-1]
	\arrow["i"', from=1-1, to=2-1]
	\arrow[""{name=1, anchor=center, inner sep=0}, "{(\hat{-})}", from=1-1, to=1-2]
	\arrow[curve={height=18pt}, dotted, from=2-3, to=1-4]
	\arrow["j"{description}, from=1-4, to=2-3]
	\arrow["\dashv"{anchor=center, rotate=-75}, draw=none, from=1, to=0]
\end{tikzcd}\]

In the next remarks, we shall discuss how to handle size issues properly and extend \Cref{costr:left-skt-class} to any Morita small sketch. This will upgrade the relative (bi)adjunction to a proper biadjunction as in the right-hand side of the diagram above.  We chose to present this \textit{corrections} to the construction separately because they would have made the general idea behind the construction less clear.
\end{rem}

\begin{rem}[\Cref{costr:left-skt-class} for Morita-small sketches] \label{correctionlarge}
Let $\sS$ be a Morita-small sketch and let us denote with $\sD$ and $j\colon \sD \to \sS$ the small sketch and morphism of sketches providing the Morita equivalence. 
\begin{itemize}
    \item[($\yo$)] First, we replace the presheaf construction with the small presheaf construction $\mathcal{P}(\cS)$ (see \cite{adamek2020nice} for a general theory of this category and \cite[2.5 and 2.6]{di2023accessibility} for its universal property).
    \item[($\hat{\sS}$)] We can still define the orthogonality class $\mathcal{H}^\perp$. To show that $\mathcal{H}^\perp$ is reflective, we cannot invoke \cite[1.36]{adamek_rosicky_1994} on the spot, because $\mathcal{H}$ could be a priori a proper class and $\mathcal{P}(\cS)$ is not locally presentable. Yet, call $\mathcal{H}_j \subset \mathcal{H}$ the subclass of those maps in $\mathcal{H}$ that come from a colimit diagram specified in $\sD$. Clearly $\mathcal{H}_j$ is small and if we show that  $\mathcal{H}_j^\perp = \mathcal{H}^\perp$ we can apply \cite[Theorem~10.2]{kelly1980unified}, so that $\mathcal{H}^\perp$ is indeed reflective. To see that $\mathcal{H}_j^\perp = \mathcal{H}^\perp$, recall the proof of \Cref{turnaroundsketch}. $P$ is $\mathcal{H}^\perp$ if and only if it maps the colimits specified \textit{declared} by $\mathcal{H}$ into limits. Because $j$ is Morita equivalence of sketches, it is enough to test this property by precomposing with $j$, which is precisely what it means to check the orthogonality with respect to $\mathcal{H}_j^\perp$. This allows us to construct $\hat{\sS}$.
    \item[$(\mathsf{M})$] Up to this point we can be sure we have constructed a (possibly) large left classifier for every left Morita-small sketch.
\[\widehat{(-)}\colon \Skt^{\mathsf{M}} \to \leftSkt\]
It remains to show that this construction lands in left Morita small sketches. Of course, because $\sS$ is not small, we can't simply apply \Cref{prop:rel-adj-left-class} to conclude that it is Morita small. In order to finish the proof, it is enough to notice that the proof of \Cref{prop:rel-adj-left-class} carries with very minor adjustments using the recipe just described (using the universal property of the small presheaf construction) and that Morita equivalences compose, and so $J_\sS j\colon\sD\to\sS\to\hat{\sS}$ is a Morita equivalence between $\hat{\sS}$ and a small sketch~$\sD$. 
\end{itemize}
\textit{En passant}, let us notice that using \Cref{prop:left-strong-iff-hat-inv}, it will follow that $\hat{\sD}\simeq\hat{\sS}$. 
    
% We shall now show that $\hat{\sS}$ coincides with $\hat{\sD}$ via abstract nonsense, and thus it is indeed Morita small, as $\sD$ is small in first place. To do so, consider a test sketch $\cM$, and the following (pseudonatural) chain of equivalence.

% we notice that if $j\colon\sD\to\sS$ is further a left Morita equivalence, then $\hat{\sD}\simeq\hat{\sS}$. This follows from Yoneda and the chain of equivalences below, for any left sketch $\sM$.   

% \[\leftSkt(\hat{\sD},\sM) \simeq  \Skt(\sD,\sM) \stackrel{\text{Left Morita}}{\simeq} \Skt(\sS,\sM) \simeq  \leftSkt(\hat{\sS},\sM)\]
 
\end{rem}

%It follows that $\hat{\sD}$ and $\hat{\sS}$ behave in the same way against any test sketch. 

\begin{rem}[\Cref{functoriality} for Morita-small sketches]
This is a straightforward adaptation/analysis of \Cref{functoriality}, where we use the fact that the small presheaf construction is the free completion under colimits.
\end{rem}
\black

\begin{rem}[A sanity check: left sketches need no hat] \label{rmk:left-class-stable} 
A key result of topos theory, which sits as a conceptual bit in the proof of Diaconescu Theorem (\Cref{diac1}), is that every topos can be understood as a canonical site for itself (see \cite[C2.2.7]{Sketches}), 

\[\mathsf{Sh}(\mathcal{E}, J_{\mathcal{E}}) \cong \mathcal{E}.\]

As a sanity check for our wanna-be version of Diaconescu Theorem, one can see that if we start with a left sketch $\sS$, then the left sketch classifier $\hat{\sS}$ is equivalent to $\sS$ itself. This is an immediate consequence of \Cref{prop:rel-adj-left-class}, since, when $\sS$ is left, $J_\sS^\ast$ reduces to an equivalence
    $$\leftSkt(\hat{\sS},\sM)\to\leftSkt(\sS,\sM)$$
    natural in $\sM$. Therefore, by Yoneda, $J_\sS^\ast$ is an equivalence.
A more direct proof of this fact could follow these lines. 
Let us start by identifying $\mathcal{S}$. First, observe that the Yoneda embedding of $\mathcal{S}$ into its category of small presheaves has a left adjoint,
%\[L :\mathcal{P}(\mathcal{S}) \leftrightarrows \mathcal{S} : \yo_{\mathcal{S}}\]
% https://q.uiver.app/#q=WzAsMixbMCwwLCJcXG1hdGhjYWx7UH0oXFxtYXRoY2Fse1N9KSJdLFsxLDAsIlxcbWF0aGNhbHtTfSJdLFsxLDAsIiBcXHlvX3tcXG1hdGhjYWx7U319IiwyLHsib2Zmc2V0IjoyfV0sWzAsMSwiTCIsMix7Im9mZnNldCI6Mn1dLFszLDIsIlxcdG9wIiwxLHsic2hvcnRlbiI6eyJzb3VyY2UiOjIwLCJ0YXJnZXQiOjIwfSwic3R5bGUiOnsiYm9keSI6eyJuYW1lIjoibm9uZSJ9LCJoZWFkIjp7Im5hbWUiOiJub25lIn19fV1d
\[\begin{tikzcd}[ampersand replacement=\&]
	{\mathcal{P}(\mathcal{S})} \& {\mathcal{S}.}
	\arrow[""{name=0, anchor=center, inner sep=0}, "{ \yo_{\mathcal{S}}}"', shift right=2, from=1-2, to=1-1]
	\arrow[""{name=1, anchor=center, inner sep=0}, "L"', shift right=2, from=1-1, to=1-2]
	\arrow["\top"{description}, draw=none, from=1, to=0]
\end{tikzcd}\]
This is a rephrasing of the fact that $\mathcal{S}$ is cocomplete (see \cite[Proposition~2.2]{garner2012lex}). Similarly to the discussion in \Cref{turnaroundsketch}, it is easy to see that the orthogonality condition in \Cref{costr:left-skt-class}, for the case of $\hat{\mathcal{S}}$ is precisely that given by the adjunction above, so that $\hat{\mathcal{S}}$ is nothing but $\mathcal{S}$.
\end{rem}

\begin{thm}(Diaconescu for left sketches) \label{diacweak}
The $2$-category of Morita small left sketches is bireflective in the $2$-category of Morita small sketches, and the reflector is given by the $\widehat{(-)}$ construction.

% https://q.uiver.app/#q=WzAsMixbMCwwLCJcXFNrdF57XFxtYXRoc2Z7TX19Il0sWzEsMCwiXFxsZWZ0U2t0XntcXG1hdGhzZntNfX0iXSxbMSwwLCJVXntMfSIsMix7Im9mZnNldCI6MSwiY3VydmUiOjF9XSxbMCwxLCJcXGhhdHsoLSl9IiwyLHsib2Zmc2V0IjoxLCJjdXJ2ZSI6MX1dLFsyLDMsIlxcdG9wIiwxLHsic2hvcnRlbiI6eyJzb3VyY2UiOjIwLCJ0YXJnZXQiOjIwfSwic3R5bGUiOnsiYm9keSI6eyJuYW1lIjoibm9uZSJ9LCJoZWFkIjp7Im5hbWUiOiJub25lIn19fV1d
\[\begin{tikzcd}[ampersand replacement=\&]
	{\Skt^{\mathsf{M}}} \& {\leftSkt^{\mathsf{M}}}
	\arrow[""{name=0, anchor=center, inner sep=0}, "{U^{L}}"', shift right, curve={height=6pt}, from=1-2, to=1-1]
	\arrow[""{name=1, anchor=center, inner sep=0}, "{\hat{(-)}}"', shift right, curve={height=6pt}, from=1-1, to=1-2]
	\arrow["\top"{description}, draw=none, from=0, to=1]
\end{tikzcd}\]

\end{thm}
\begin{proof}
 \Cref{prop:rel-adj-left-class} provides a equivalence, for any small sketch $\sS$ and left sketch $\sM$, 
$$\leftSkt(\hat{\sS},\sM)=\Skt(\hat{\sS},\sM)\simeq\Skt(\sS,U^L\sM).$$
Then,  \Cref{correctionlarge} extends this to Morita small sketches, hence we get the equivalence below for any Morita small sketch $\sS$ and Morita small left sketch $\sM$. 
$$\leftSkt^\mathsf{M}(\hat{\sS},\sM)\simeq\Skt^\mathsf{M}(\sS,U^L\sM)$$
We notice that, the direction $\leftSkt^\mathsf{M}(\hat{\sS},\sM)\to\Skt^\mathsf{M}(\sS,U^L\sM)$ is given by precomposition with $J_\sS$ and so it is strictly natural. The other direction instead, is only pseudonatural, as apparent from the construction given in \Cref{prop:rel-adj-left-class} where $\hat{F}$ is define up to natural isomorphism. 

For these reasons, we get a biadjunction (and neither 2-natural nor pseudo). 

We underline that by \Cref{rmk:left-class-stable} the counit of the biadjunction is an equivalence, since if we start with a left sketch $\sM$, then
$\widehat{(-)}\circ U^L (\sM)=\hat{\sM}\simeq\sM.$
\end{proof}

A quite classical way to look at Diaconescu theorem is to say that the $2$-category of topoi is a localization of the $2$-category of sites up to inverting Morita equivalences. Despite having been folklore for a long time, this theorem was formalized only very recently in \cite{ramos2018grothendieck}. In the proposition below we recover at least the essence of this result.

\begin{prop}\label{prop:left-strong-iff-hat-inv}
    Let $F\colon\sS \to \sT$ be a sketch morphism, then $\hat{F}$ is invertible if and only if $F$ is a left Morita equivalence. 
\end{prop}
\begin{proof}
In the diagram below, the vertical arrows are equivalences of functors by \Cref{prop:rel-adj-left-class}. Thus, the top horizontal arrow is an equivalence if and only if the bottom one is. The rest follows by Yoneda lemma, because both $\hat{\sT}$ and $\hat{\sS}$ are left sketches.

% % https://q.uiver.app/#q=WzAsNCxbMCwwLCJcXFNrdChcXHNULC0pIl0sWzEsMCwiXFxTa3QoXFxzUywtKSJdLFswLDEsIlxcbGVmdFNrdChcXGhhdHtcXHNUfSwtKSJdLFsxLDEsIlxcbGVmdFNrdChcXGhhdHtcXHNTfSwtKSJdLFsyLDAsIkpeKl97XFxzVH0iXSxbMywxLCJKXipfe1xcc1N9IiwyXSxbMSwwLCJGXioiLDJdLFszLDIsIlxcaGF0e0Z9XioiXV0=
\[\begin{tikzcd}[ampersand replacement=\&]
	{\Skt(\sT,-)} \& {\Skt(\sS,-)} \\
	{\leftSkt(\hat{\sT},-)} \& {\leftSkt(\hat{\sS},-)}
	\arrow["{J^*_{\sT}}", from=2-1, to=1-1]
	\arrow["{J^*_{\sS}}"', from=2-2, to=1-2]
	\arrow["{F^*}"', from=1-2, to=1-1]
	\arrow["{\hat{F}^*}", from=2-2, to=2-1]
\end{tikzcd}\]
\end{proof}

\subsection{About the \textit{lemme de comparison}} \label{lemmelemme}
A very classical and celebrated result in topos theory is the so-called \textit{lemme de comparison}, which establishes that every generating subcategory of a topos admits a site structure that turns the inclusion into a dense subsite. See \cite[Prop 5.5 and Thm 5.7]{caramello2019denseness} for a modern reference on this theorem, and a quite substantial generalization of the original result due to the french school. Some evidence that a form of the lemme de comparison could be true in our context is given by the proposition below.

\begin{prop}[Evidence pro \textit{lemme de comparison}]\label{prop:left-morita-then-dense}
Let $F\colon\sS\to\sT$ is a left Morita equivalence where $\sT$ is a left sketch, then the underlying functor $F$ is dense (in $\Cat$).
\end{prop}

\begin{proof}
 Consider the (pseudo) commutative diagram below. The dashed functors are equivalences of categories, $J_\sT$ by \Cref{rmk:left-class-stable} and $\hat{F}$ by \Cref{prop:left-strong-iff-hat-inv}. Thus, $F$ is dense in $\Cat$ if and only if $J_{\sS}$ is. 
% https://q.uiver.app/#q=WzAsNCxbMCwwLCJcXHNTIl0sWzEsMCwiXFxzVCJdLFsxLDEsIlxcaGF0e1xcc1R9Il0sWzAsMSwiXFxoYXR7XFxzU30iXSxbMCwzLCJKX3tcXHNTfSIsMl0sWzEsMiwiSl97XFxzVH0iLDAseyJzdHlsZSI6eyJib2R5Ijp7Im5hbWUiOiJkYXNoZWQifX19XSxbMCwxLCJGIl0sWzMsMiwiXFxoYXR7Rn0iLDIseyJzdHlsZSI6eyJib2R5Ijp7Im5hbWUiOiJkYXNoZWQifX19XSxbNCw1LCJcXGNvbmciLDEseyJzaG9ydGVuIjp7InNvdXJjZSI6MjAsInRhcmdldCI6MjB9LCJzdHlsZSI6eyJib2R5Ijp7Im5hbWUiOiJub25lIn0sImhlYWQiOnsibmFtZSI6Im5vbmUifX19XV0=
\[\begin{tikzcd}[ampersand replacement=\&]
	\sS \& \sT \\
	{\hat{\sS}} \& {\hat{\sT}}
	\arrow[""{name=0, anchor=center, inner sep=0}, "{J_{\sS}}"', from=1-1, to=2-1]
	\arrow[""{name=1, anchor=center, inner sep=0}, "{J_{\sT}}", dashed, from=1-2, to=2-2]
	\arrow["F", from=1-1, to=1-2]
	\arrow["{\hat{F}}"', dashed, from=2-1, to=2-2]
	\arrow["\cong"{description}, draw=none, from=0, to=1]
\end{tikzcd}\]
The latter is true by construction of $J_\sS$, indeed, following the notation set in \Cref{costr:left-skt-class}, we have the chain of isomorphisms below. 
\begin{align*}
    \text{lan}_{J_\sS}{J_\sS} & = \text{lan}_{L\yo_{\cS}}{(L\yo_{\cS})}  
    & \\
                            & \cong \text{lan}_{L} (\text{lan}_{\yo_{\cS}}(L\yo_{\cS}))
    & (\text{lan is a ladj, ladjs compose})\\
        & \cong \text{lan}_{\yo_{\cS}}(L\yo_{\cS}) \circ I 
    & (\text{lan}_L(-) \cong - \circ I)\\
               & \cong L \text{lan}_{\yo_{\cS}}(\yo_{\cS}) \circ I
           & (\text{ladjs preserve lans})\\
       & \cong L \circ I 
    & (\text{density of Yoneda})\\
              & \cong 1
    & (\text{reflectivity})
\end{align*}
 We remark that $\text{lan}_L(-) \cong - \circ I$ because $L$ is left adjoint, and thus $\text{lan}_L(1) \cong I$ and moreover such Kan extension is absolute.
 % We have $\text{lan}_L(-)\dashv -\circ L$, but Yoneda preserves adjunctions so, since $L\dashv I$, hence $-\circ I \dashv -\circ L$ and so $\text{lan}_L(-) \cong - \circ I$. 
\end{proof}

Unfortunately though, no version of the lemme de comparison seems to be true at our level of generality, in at least two senses: 
\begin{enumerate}
    \item The fact that any strong generator is automatically a dense subcategory is a very special behavior of a topos (see \cite{street1986categories}) and thus the only version of the lemma that has a chance of being true would be that if $F\colon \cS \to \sT$ is a dense functor, there exists a sketch structure on $\cS$ that turns $F$ into a left Morita equivalence.
    \item Moreover, the canonical strategy to prove such result has a flaw, which we shall discuss below for its relevance to the general theory. 
\end{enumerate}

In order to try and build a sketch structure on $\cS$, we start by embedding $\cT$ in the presheaf category over $\cS$. Because $F$ is dense, the nerve-realization paradigm establishes $\cT$ as a reflective subcategory of $\Psh(\sT)$, in particular, this establishes $\cT$ as an orthogonality class in $\Psh(\cS)$.
% https://q.uiver.app/#q=WzAsMyxbMCwwLCJcXGNTIl0sWzAsMiwiXFxtYXRoc2Z7UHNofShcXGNTKSJdLFsyLDAsIlxcY1QiXSxbMCwxLCJcXHlvIiwyXSxbMiwxLCJOKGYpIiwwLHsiY3VydmUiOi0zLCJzdHlsZSI6eyJ0YWlsIjp7Im5hbWUiOiJob29rIiwic2lkZSI6ImJvdHRvbSJ9LCJib2R5Ijp7Im5hbWUiOiJkYXNoZWQifX19XSxbMSwyLCJcXHRleHR7bGFufV9cXHlvIEYiXSxbMCwyLCJGIl0sWzUsNCwiIiwyLHsibGV2ZWwiOjEsInN0eWxlIjp7Im5hbWUiOiJhZGp1bmN0aW9uIn19XV0=
\[\begin{tikzcd}[ampersand replacement=\&]
	\cS \&\& \cT \\
	\\
	{\Psh(\cS)}
	\arrow["\yo"', from=1-1, to=3-1]
	\arrow[""{name=0, anchor=center, inner sep=0}, "{N(f)}", curve={height=-18pt}, dashed, hook', from=1-3, to=3-1]
	\arrow[""{name=1, anchor=center, inner sep=0}, "{\text{lan}_{\yo} F}", from=3-1, to=1-3]
	\arrow["F", from=1-1, to=1-3]
	\arrow["\dashv"{anchor=center, rotate=-39}, draw=none, from=1, to=0]
\end{tikzcd}\]
We shall call $\mathcal{H}$ the class of maps inverted by ${\text{lan}_{\yo} F}$, so that we have $\cT \simeq \mathcal{H}^\perp$. At this point, one would like to say that it is enough to invert maps whose codomain is a representable, so that they look as follow, \[P \to \yo s,\]
and such family of maps gives us a family of cocones on $\cS$, which would be the colimit part of our sketch structure. Such argument can be carried when $\cT$ is a topos because $\mathcal{H}$ will be closed under pullback, and using the density of representables, and the fact that a presheaf category is extensive one can actually prove that $\mathcal{H}$ is generated by maps of the form $P \to \yo c$. Yet, in the general case, there is no reason for this to be true, and the canonical proof strategy breaks.

% \red
% The following prop is false!
% \begin{prop}
%     If $\sS$ is colimit, then $\hat{\sS}$ is precisely the exponential $\Set^{\sS^\circ}$.
% \end{prop}
% \begin{proof}
% This is the previous proposition, plus a point-set analysis of the sketch structure.
% \end{proof}
% \black
% \begin{prop}
%      If $\sS$ is colimit then $\Set^{\sS^\circ}$ is a right sketch, (where on $\Set$ we consider the test structure or the \textit{pure limit} structure).
% \end{prop}
% \begin{proof}
% Clearly   $\Set^{\cS^\circ}$ is test. Also, the inclusion $\Set^{\sS^\circ}  \hookrightarrow \Set^{\cS^\circ}$ creates (all) limits. Thus, by the definition of the sketch structure in $\Set^{\sS^\circ}$, it follows that all limit diagrams are in the limit class. This is the definition of right sketch.
% \end{proof}

% \begin{rem}
% $\hat{(-)}$ lifts of the esponentiation $\Set^{\sS^\circ}$, in the sense that we are just putting on $\Set^{\sS^\circ}$ the sketch structure that makes the "Yoneda embedding" a sketch morphism.
% % https://q.uiver.app/#q=WzAsNCxbMCwxLCJcXHNrdF5jIl0sWzAsMCwiXFxza3QiXSxbMSwxLCJcXGxlZnRza3ReYyJdLFsxLDAsIlxcbGVmdHNrdCJdLFsxLDBdLFsxLDMsIlxcaGF0eygtKX0iLDAseyJzdHlsZSI6eyJib2R5Ijp7Im5hbWUiOiJkYXNoZWQifX19XSxbMywyXSxbMCwyLCJcXFNldF57KC0pXlxcY2lyY30iLDJdXQ==
% \[\begin{tikzcd}
% 	\skt & \leftSkt \\
% 	{\skt^c} & {\leftSkt^c}
% 	\arrow[from=1-1, to=2-1]
% 	\arrow["{\hat{(-)}}", dashed, from=1-1, to=1-2]
% 	\arrow[from=1-2, to=2-2]
% 	\arrow["{\Set^{(-)^\circ}}"', from=2-1, to=2-2]
% \end{tikzcd}\]  
% \end{rem}
\subsection{Keep it small: left normalization}
This subsection is a small detour on the general theme of left sketches. Recall from the definition of left sketch that a left sketch is, \textit{trivially}, left normal (see \Cref{skt}). This gives us the horizontal forgetful functor in the diagram below.

% https://q.uiver.app/#q=WzAsMyxbMCwwLCJcXGxlZnRTa3ReXFx0ZXh0e019Il0sWzIsMCwiXFxsZWZ0blNrdF5cXHRleHR7TX0iXSxbMSwyLCJcXFNrdF5cXHRleHR7TX0iXSxbMCwxLCJVXkxfe2xufSJdLFsxLDIsIlVee2xufSIsMl0sWzAsMiwiVV5MIl0sWzIsMCwiXFxoYXR7KC0pfSIsMCx7ImN1cnZlIjotNH1dLFsyLDEsIlxcdGlsZGV7KC0pfSIsMix7ImN1cnZlIjozLCJzdHlsZSI6eyJib2R5Ijp7Im5hbWUiOiJkYXNoZWQifX19XSxbNiw1LCIiLDAseyJsZXZlbCI6MSwic3R5bGUiOnsibmFtZSI6ImFkanVuY3Rpb24ifX1dLFs3LDQsIiIsMCx7ImxldmVsIjoxLCJzdHlsZSI6eyJuYW1lIjoiYWRqdW5jdGlvbiJ9fV1d
\[\begin{tikzcd}[ampersand replacement=\&]
	{\leftSkt^\text{M}} \&\& {\leftnSkt^\text{M}} \\
	\\
	\& {\Skt^\text{M}}
	\arrow["{U^L_{ln}}", from=1-1, to=1-3]
	\arrow[""{name=0, anchor=center, inner sep=0}, "{U^{ln}}"', from=1-3, to=3-2]
	\arrow[""{name=1, anchor=center, inner sep=0}, "{U^L}", from=1-1, to=3-2]
	\arrow[""{name=2, anchor=center, inner sep=0}, "{\widehat{(-)}}", curve={height=-24pt}, from=3-2, to=1-1]
	\arrow[""{name=3, anchor=center, inner sep=0}, "{\widetilde{(-)}}"', curve={height=18pt}, dashed, from=3-2, to=1-3]
	\arrow["\dashv"{anchor=center, rotate=30}, draw=none, from=2, to=1]
	\arrow["\dashv"{anchor=center, rotate=151}, draw=none, from=3, to=0]
\end{tikzcd}\]

In this section we will see that a small variant of $\widehat{(-)}$ provides the construction
of the free left normal sketch. This is particularly useful because it respects sizes of
sketches, in the sense that when S is small its left normalization will still be small.

\begin{constr}[Left normalisation]\label{costr:left-normalisation}
From $\hat{\sS}$ we construct the \emph{left normalisation} $\Tilde{\sS}$ of $\sS$ using the es-ff (essentially surjective on objects and fully faithful) weak 2-dimensional factorisation system on $\Cat$.\black 
% https://q.uiver.app/#q=WzAsMyxbMCwwLCJcXGNTIl0sWzIsMCwiXFxoYXR7XFxjU30iXSxbMSwxLCJcXHRpbGRle1xcY1N9Il0sWzAsMSwiSl9cXGNTIl0sWzAsMiwiUF9cXGNTIiwyXSxbMiwxLCJSX1xcY1MiLDJdLFs0LDUsIlxcY29uZyIsMSx7InNob3J0ZW4iOnsic291cmNlIjoyMCwidGFyZ2V0IjoyMH0sInN0eWxlIjp7ImJvZHkiOnsibmFtZSI6Im5vbmUifSwiaGVhZCI6eyJuYW1lIjoibm9uZSJ9fX1dXQ==
\[\begin{tikzcd}[ampersand replacement=\&]
	\cS \&\& {\hat{\cS}} \\
	\& {\tilde{\cS}}
	\arrow["{J_\cS}", from=1-1, to=1-3]
	\arrow[""{name=0, anchor=center, inner sep=0}, "{P_\cS}"', from=1-1, to=2-2]
	\arrow[""{name=1, anchor=center, inner sep=0}, "{R_\cS}"', from=2-2, to=1-3]
	\arrow["\cong"{description}, draw=none, from=0, to=1]
\end{tikzcd}\]
Hence, $\tilde{\cS}$ is small. We define $\tilde{\sS}$ as the category $\tilde{\cS}$ equipped with the minimal sketch structure making $P_\cS$ a morphism of sketches, using \Cref{prop:min-max}. 
\end{constr}

We underline that in \Cref{costr:left-normalisation} above one could have chosen the bo-ff factorisation system (which is a strict 2-dimensional system) and still recover some results similar to the ones in this section. 
We chose the weak version, because the construction $\widehat{(-)}$ is only \emph{pseudo}functorial, hence in order to get a (bi)adjunction in \Cref{thm:left-norm-adj} we need the weak one.

\begin{rem}[Left normalization and subcanonical topologies] 
    The left normalization $\tilde{\sS}$ of a sketch $\sS$ is designed to mimik and generalize a standard construction in topos theory. Let $(C,J)$ be a site, and recall that a topology is called \textit{subcanonical}, when the sheafification functor is fully faithful, i.e. when representables are automatically sheaves. It is well known that for $(C,J)$ a site, one can always find a Morita equivalent site $(C,J) \to (\tilde{C}, \tilde{J}) $ whose topology $\tilde{J}$ is subcanonical. The standard way to do so is indeed to consider the (es-ff) factorization of the sheafification functor, and then equip $\tilde{C}$ in the diagram below with the correct topology.
    % https://q.uiver.app/#q=WzAsMyxbMCwwLCJDIl0sWzIsMCwiXFxtYXRoc2Z7U2h9KEMsSikiXSxbMSwxLCJcXHRpbGRle0N9Il0sWzAsMSwiSiJdLFswLDIsIlAiLDIseyJzdHlsZSI6eyJib2R5Ijp7Im5hbWUiOiJkYXNoZWQifX19XSxbMiwxLCJSIiwyLHsic3R5bGUiOnsiYm9keSI6eyJuYW1lIjoiZGFzaGVkIn19fV1d
\[\begin{tikzcd}[ampersand replacement=\&]
	C \&\& {\mathsf{Sh}(C,J)} \\
	\& {\tilde{C}}
	\arrow["J", from=1-1, to=1-3]
	\arrow["P"', dashed, from=1-1, to=2-2]
	\arrow["R"', dashed, from=2-2, to=1-3]
\end{tikzcd}\]
    As we have seen in \Cref{leftsketchclasstopos}, for $(C,J)$ a site (identified with its associated sketch from \cref{sites}) the construction of the classifying left sketch coincides with the construction of $\mathsf{Sh}(C,J)$ and thus, indeed, the left normalization of $(C,J)$ is the same of its correction into a subcanonical site.
\end{rem}

\begin{rem}[Left normal sketches need no tilde] \label{rem:left-norm-no-like-tilde} 
    Similarly to \Cref{rmk:left-class-stable}, we can check that when we start with a left normal sketch $\sS$, then its left normalisation $\tilde{\sS}$ is equivalent to itself. 
    The key point is to notice that when $\sS$ is left normal, then $J_\sS$ is fully faithful. 
    In order to show this, we go back to the orthogonality condition in \Cref{costr:left-skt-class} as we did in \Cref{turnaroundsketch}. This time, because $\sS$ is left normal, we know that the codomain of any $\rho_\delta$\footnote{ We recall that, given any $\delta\colon d\Rightarrow \Delta(s)\in\skC_\sS$, we define $\rho_\delta \colon \text{colim} \yo d \Rightarrow \yo s$ as the universal map given by the colimit. } is itself of the form $\yo (\text{colim}d)$ and what we want to show is that a representable is always orthogonal to such map. Indeed, when this is true, then $L\yo (x) \cong \yo(x)$ and thus $J_\sS$ will be fully faithful. 
 
    % https://q.uiver.app/#q=WzAsMyxbMCwwLCJcXHRleHR7Y29saW19IChcXHlvIGQpIl0sWzAsMSwiXFx5byAoXFx0ZXh0e2NvbGltfWQpIl0sWzEsMCwiXFx5byh4KSJdLFswLDEsIlxccmhvX1xcZGVsdGEiLDJdLFswLDJdXQ==
\[\begin{tikzcd}[ampersand replacement=\&]
	{\text{colim} (\yo d)} \& {\yo(x)} \\
	{\yo (\text{colim}d)}
	\arrow["{\rho_\delta}"', from=1-1, to=2-1]
	\arrow[from=1-1, to=1-2]
\end{tikzcd}\]
    The fact that a representable is orthogonal to such map follows directly from the discussion in \Cref{turnaroundsketch}. Indeed, a presheaf is orthogonal to such map if and only if it maps the tip to the limit of the diagram, and this is true because representables are continuous (and -- again -- $\sS$ is left normal).
    Therefore, $L\yo=J_\sS$ coincides with the Yoneda embedding restricted to its image. 
    We can then conclude that $J_\sS$ is fully faithful since both $\yo$ and $L$ restricted to representable are as well. 
    Now, by definition $P$ is essentially surjective on objects, and since $RP=J_\sS$ which is fully faithful, it is also fully faithful. 
    Hence, if $\sS$ is left normal, then $P$ is an equivalence. 
\end{rem}

\begin{thm}
\label{thm:left-norm-adj}
    \Cref{costr:left-normalisation} gives a biadjoint as below. 
% https://q.uiver.app/#q=WzAsMixbMCwwLCJcXHNrdCJdLFsxLDAsIlxcbGVmdG5za3QiXSxbMCwxLCJcXHRpbGRleygtKX0iLDIseyJjdXJ2ZSI6Mn1dLFsxLDAsIiIsMCx7ImN1cnZlIjoyLCJzdHlsZSI6eyJ0YWlsIjp7Im5hbWUiOiJob29rIiwic2lkZSI6ImJvdHRvbSJ9fX1dLFszLDIsIlxcdG9wIiwxLHsic2hvcnRlbiI6eyJzb3VyY2UiOjIwLCJ0YXJnZXQiOjIwfSwic3R5bGUiOnsiYm9keSI6eyJuYW1lIjoibm9uZSJ9LCJoZWFkIjp7Im5hbWUiOiJub25lIn19fV1d
\[\begin{tikzcd}[ampersand replacement=\&]
	\skt \& \leftnskt
	\arrow[""{name=0, anchor=center, inner sep=0}, "{\tilde{(-)}}"', curve={height=12pt}, from=1-1, to=1-2]
	\arrow[""{name=1, anchor=center, inner sep=0}, curve={height=12pt}, hook', from=1-2, to=1-1]
	\arrow["\top"{description}, draw=none, from=1, to=0]
\end{tikzcd}\]

\end{thm}

\begin{proof}
Given $\sS\in\skt$ and $\sT\in\leftnskt$, we want to show that precomposing with $P_\sS\colon\sS\to\tilde{\sS}$ induces  a natural equivalence as below. 
$$-\circ P_\sS\colon \leftnskt(\tilde{\sS},\sT)\simeq\skt(\sS,U^{ln}\sT)$$
Since $\sT$ is left normal, $P_\sT\colon\sT\to\tilde{\sT}$ is an equivalence, hence it suffices to show that there is a natural equivalence as below.
$$\leftnskt(\tilde{\sS},\tilde{\sT})\simeq\skt(\sS,U^{ln}\sT)$$
For any morphism of sketches $F\colon \sS \to U^{ln}\sT$, the (pseudo)functoriality of the es-ff weak factorization system provides us with the essentially unique functor $\tilde{F}$. 
% https://q.uiver.app/#q=WzAsNixbMCwwLCJcXHNTIl0sWzAsMSwiXFx0aWxkZXtcXHNTfSJdLFsyLDAsIlxcc1QiXSxbMCwyLCJcXGhhdHtcXHNTfSJdLFsyLDIsIlxcaGF0e1xcc1R9Il0sWzIsMSwiXFx0aWxkZXtcXHNUfSJdLFswLDIsIkYiXSxbMCwxLCJQX3tcXHNTfSIsMl0sWzEsNSwiXFx0aWxkZXtGfSIsMix7InN0eWxlIjp7ImJvZHkiOnsibmFtZSI6ImRhc2hlZCJ9fX1dLFsyLDUsIlBfe1xcc1R9Il0sWzEsMywiUl97XFxzU30iLDJdLFswLDMsIkpfe1xcc1N9IiwyLHsiY3VydmUiOjR9XSxbNSw0LCJSX3tcXHNUfSJdLFsyLDQsIkpfe1xcc1R9IiwwLHsiY3VydmUiOi00fV0sWzMsNCwiXFxoYXR7Rn0iLDJdLFs3LDksIlxcY29uZyIsMSx7InNob3J0ZW4iOnsic291cmNlIjoyMCwidGFyZ2V0IjoyMH0sInN0eWxlIjp7ImJvZHkiOnsibmFtZSI6Im5vbmUifSwiaGVhZCI6eyJuYW1lIjoibm9uZSJ9fX1dLFsxMCwxMiwiXFxjb25nIiwxLHsic2hvcnRlbiI6eyJzb3VyY2UiOjIwLCJ0YXJnZXQiOjIwfSwic3R5bGUiOnsiYm9keSI6eyJuYW1lIjoibm9uZSJ9LCJoZWFkIjp7Im5hbWUiOiJub25lIn19fV1d
\[\begin{tikzcd}[ampersand replacement=\&]
	\sS \&\& \sT \\
	{\tilde{\sS}} \&\& {\tilde{\sT}} \\
	{\hat{\sS}} \&\& {\hat{\sT}}
	\arrow["F", from=1-1, to=1-3]
	\arrow[""{name=0, anchor=center, inner sep=0}, "{P_{\sS}}"', from=1-1, to=2-1]
	\arrow["{\tilde{F}}"', dashed, from=2-1, to=2-3]
	\arrow[""{name=1, anchor=center, inner sep=0}, "{P_{\sT}}", from=1-3, to=2-3]
	\arrow[""{name=2, anchor=center, inner sep=0}, "{R_{\sS}}"', from=2-1, to=3-1]
	\arrow["{J_{\sS}}"', curve={height=24pt}, from=1-1, to=3-1]
	\arrow[""{name=3, anchor=center, inner sep=0}, "{R_{\sT}}", from=2-3, to=3-3]
	\arrow["{J_{\sT}}", curve={height=-24pt}, from=1-3, to=3-3]
	\arrow["{\hat{F}}"', from=3-1, to=3-3]
	\arrow["\cong"{description}, draw=none, from=0, to=1]
	\arrow["\cong"{description}, draw=none, from=2, to=3]
\end{tikzcd}\]
Moreover, $\tilde{F}$ is a morphism of sketches, since $\tilde{F}P_\sS\cong P_\sT F$ is such (by definition of the sketch structure on $\tilde{\sS}$ and \Cref{prop:min-max}). On the other hand, given any $G\colon\tilde{\sS}\to\tilde{\sT}$, we can recover $F\colon\sS\to\sT$ precomposing with $P_\sS$ and postcomposing with a inverse of $P_\sT$.  
On 2-cells it follows because the es-ff is a weak 2-categorical factorization. 

 As argued at the end of \Cref{diacweak}, if we have two directions which are strictly natural, we get a pseudo adjunction, if they are instead only pseudonatural we get a biadjunction. 
 On the domain, once chosen a weak inverse $P^{-1}_\sT$ of $P_\sT$, we have that the equivalence is given by $P^{-1}_\sT\circ-\circ P_\sS$, hence stricly natural. 
 On the other hand, also in this case we have that the equivalence on the codomain is only pseudonatural  (since it is given by the property of the weak 2-dimensional factorisation system), hence we get a biadjunction.

% \red [RISCRIVERLA] \black 
% Let $F\colon \sS \to \sT$ be a morphism of sketches into a left normal sketch. We shall now show that there exists a unique extension of $F$ to a morphism of left normal sketches as below.
% % https://q.uiver.app/#q=WzAsMyxbMCwwLCJcXHNTIl0sWzAsMSwiXFx0aWxkZXtcXHNTfSJdLFsxLDAsIlxcc1QiXSxbMCwyLCJGIl0sWzAsMSwiUCIsMl0sWzEsMiwiXFx0aWxkZXtGfSIsMix7InN0eWxlIjp7ImJvZHkiOnsibmFtZSI6ImRhc2hlZCJ9fX1dXQ==
% \[\begin{tikzcd}[ampersand replacement=\&]
% 	\sS \& \sT \\
% 	{\tilde{\sS}}
% 	\arrow["F", from=1-1, to=1-2]
% 	\arrow["P"', from=1-1, to=2-1]
% 	\arrow[""', dashed, from=2-1, to=1-2]
% \end{tikzcd}\]

% \red Let us start by observing that if such a morphism exists, then it is essentially unique, indeed $P$ is bijective on objects, and thus it is right-faithful  as shown in \cite[Lemma~3.1]{frey20092}. [Problem, this does not imply epi, so we still need to prove it] \black 
% Now let us discuss the existence. 

% Consider the diagram below, and {\ivan{ observe that because $\sT$ is already left normal, $J_\sT$ will be fully faithful, }} and thus $P_{\sT}$ is an equivalence of categories. Moreover, the functoriality of the bo-ff factorization system provides us with the dashed morphism of sketches $\tilde{F}$.

\end{proof}

\begin{cor}
The sketch morphisms $P$ and $R$ defined in \Cref{costr:left-normalisation} are left Morita equivalences.
\end{cor}

\begin{proof}
Since any left $\sL$ is in particular left normal, it follows that $P$ is a left Morita equivalence. Because both $J_{\sS}$ and $P$ are both left Morita equivalences and $J_{\sS}=RP$, it follows that $R$ must be a left equivalence as well.
\end{proof}

\black 

\subsection{Morita small left sketches}

Many properties of the $2$-categories of Morita-small left sketches have already implicitly emerged in the previous subsection. In this subsection we shall state them. %collect them and discuss their relevance. 

\begin{prop}
The underlying category of a Morita small left sketch is locally presentable.
\end{prop}

\begin{proof}
 Let $\sS$ be a Morita-small sketch and let us denote with $\sD$ and $j\colon \sD \to \sS$ the small sketch and morphism of sketches providing the Morita equivalence. As we noted in \Cref{correctionlarge}, it follows from \Cref{prop:left-strong-iff-hat-inv} that $\hat{\cS}$ coincides with $\hat{\cD}$. Now, by \Cref{costr:left-skt-class} this is a small orthogonality class in a locally presentable category, and thus is it locally presentable by \cite[1.40]{adamek_rosicky_1994}.
\end{proof}

\begin{rem}
More can be said: because $\lambda$-presentable objects are closed under $\lambda$-small colimits, if $\lambda$ is a cardinal that bounds the cardinality of the objects in every diagram in the cocones in $\sS$, then $\sT$ will be locally $\lambda$-presentable. Similar results are discussed in \cite[Section~5]{di2020gabriel} for the case of topoi and sites of definition.
\end{rem}

\begin{cor}
There is an obvious forgetful functor $\mathsf{LSkt}^{\mathsf{M}} \to \mathsf{Pres}$, the $2$-category of locally presentable categories, cocontinuous functors and natural transformations.
\end{cor}

\section{Rounded sketches}\label{sec:rounded-skt}

In this section we introduce the notion of (left) rounded sketch. In terms of intuition, (left) rounded sketchs are a sketch-version of the notion of site from topos theory and are designed to encode Giraud-like axioms in the left sketch classifier. From a more technical point of view, what we do is -- in a sense -- to import the technology of lex colimits to the theory of sketches.

% \subsubsection{Left tame sketches}

% \begin{defn}[Left tame sketch]
% A left tame sketch is a left sketch $\sS= (\cS, \skL_\sS, \skC_\sS)$ such that $\cS$ is LAFT (or strongly compact) in the sense of \cite[Sec 3 and Rem 3.8]{brandenburg2021large}.
% \end{defn}

% \begin{prop}
% Let $\sS= (\cS, \skL_\sS, \skC_\sS)$ be left Morita small sketch, then it is tame.
% \end{prop}

% \subsection{Right Stuff}

% Disclaimer: tutto è uguale daje. 

% \subsection{Normal Stuff}

\begin{defn}[(Left) rounded sketch]\label{def:rounded-skt}
Let $\sS=(\cS,\skL,\skC)$ be a sketch. We say that $\sS$ is left rounded when, in the construction of $\hat{\sS}$, the reflection $L$ transforms $\skL$ into limit diagrams, i.e. $J_\sS$ sends cones in $\skL$ to limit diagrams.

% https://q.uiver.app/#q=WzAsMyxbMCwwLCJcXGNTIl0sWzEsMCwiXFxoYXR7XFxjU30iXSxbMCwxLCJcXG1hdGhzZntQc2h9KFxcY1MpIl0sWzAsMSwiSl9cXHNTIl0sWzAsMiwiXFx5byIsMl0sWzIsMSwiTCIsMl1d
\[\begin{tikzcd}[ampersand replacement=\&]
	\cS \& {\hat{\cS}} \\
	{\Psh(\cS)}
	\arrow["{J_\sS}", from=1-1, to=1-2]
	\arrow["\yo"', from=1-1, to=2-1]
	\arrow["L"', from=2-1, to=1-2]
\end{tikzcd}\]
\end{defn}

\begin{rem}(Right rounded sketches exist too!)
    As in \Cref{caveatrighutskt}, for an appropriate variant of \Cref{sec:left-skt-class}, it would be possible to define right rounded sketches. Since we have no interesting examples of these objects, and because the theory would be entirely dual to that of left rounded sketches, we choose to stick to the theory of left rounded sketches and make no further mention of right rounded sketches.
\end{rem}

We shall now proceed to better understand this notion via a number of examples and contextualization with similar ideas in the literature. Before doing that though, we start with two propositions that will hopefully demystify the technical nature of the definition. The reader that may prefer to see an example first can jump to \Cref{roundedexample}.

\begin{prop}[]\label{prop:left-class-of-rounded-is-normal}
    A sketch $\sS$ is rounded if and only if its left classifier $\hat{\sS}$ is normal. 
\end{prop}

\begin{proof}
    Since any left sketch is in particular left normal, it suffices to show that a sketch is rounded if and only if its left classifier is right normal. We recall that the limit part of $\hat{\sS}$ is defined as
$$\skL_{\hat{\sS}}:=\lbrace J_\sS\pi\,\mid\,\pi\in\skL_\sS\rbrace.$$
$\hat{\sS}$ is right normal if all the specified cones are of limit form. Therefore, $\hat{\sS}$ is right normal if and only $J_\sS$ sends cones in $\skL_\sS$ to limit diagrams, which is exactly the definition of rounded sketch. 
\end{proof}

\begin{prop}
If $\sS$ is a rounded sketch, then its left normalization $\tilde{\sS}$ is normal. Moreover, if $R_\sS$ creates limits, the converse is true.
\end{prop}

\begin{proof}
    We recall that given a sketch $\sS$, its left normalisation $\tilde{\sS}$ is defined through the es-ff factorisation of $J_\sS\cong R_\sS P_\sS\colon\cS\to\hat{\cS}$ (see \Cref{costr:left-normalisation}). In particular, the sketch structure on $\tilde{\sS}$ is the minimal structure making the essentially surjective on objects functor $P\colon\cS\to\Tilde{\cS}$ a sketch morphism. Moreover, since $\tilde{\sS}$ is left normal, it suffices to show that any diagram $\pi\in\skL_{\tilde{\sS}}$ is a limit diagram (i.e. $\tilde{\sS}$ is right normal). 

    Using the characterisation of rounded sketches in \Cref{prop:left-class-of-rounded-is-normal} and the fact that $R$ is a sketch morphism, we conclude that $R_\sS\pi$ is a limit diagram. Therefore, since $R$ is fully faithful and so it reflects limits, $\pi$ itself must be a limit diagram.  
    
    Moreover, let us consider the situation when $R_\sS$ creates limits and $\tilde{\sS}$ is left normal. We want to show that $J_\sS\colon\cS\to\hat{\cS}$ sends any cone in $\skL_\sS$ to a limit cone in $\hat{\cS}$c (i.e. $\sS$ is left rounded). We recall that $J_\sS\cong R_\sS P_\sS$, so any  $\pi\in\skL_\sS$ is sent by $J_\sS$ to $R_\sS P_\sS\pi$. Now, since $P_\sS$ is a sketch morphism and $\tilde{\sS}$ is left normal by hypothesis, then $P_\sS\pi$ must be a limit. 
    Therefore, since $R_\sS$ creates limits, $R_\sS P_\sS\pi\cong J_\sS\pi$ is a limit as well as required. 
\end{proof}

\begin{rem}
    In all the examples that we are familiar with, $R$ does indeed create limits. In the next subsection (\Cref{roundedexample}) we will display such examples.
\end{rem}

% \red

% \begin{prop}[]\label{prop:left-class-of-rounded-is-normal}
% Let $\mathcal{S}$ be a sketch. The following are equivalent.
% \begin{enumerate}
%     \item $\mathcal{S}$ is rounded;
%     \item its left classifier $\hat{\sS}$ is normal;
%     \item its left normalization $\tilde{\sS}$ is normal.
% \end{enumerate}
% \end{prop}

% GAB: Non so se è vera sta cosa, perché $R$ essendo fully faithful riflette ma non preserva i limiti. Magari c'è qualche altro trucco da usare ma mi puzza. Stesso motivo per cui sapendo che $\hat{\sS}$ è left normal allora concludiamo che $\tilde{\sS}$ lo è pure ma non viceversa. 

% MA: If we prove that $R$ creates limits then clearly the statement above is true. This could be the case because Yoneda does create, preserve, reflect and tutto-bello limits. TO BE CHECKED.

% \black

\begin{prop}\label{prop:rounded-then-left-class-rounded}
    If  $\mathcal{S}$ is a rounded sketch, then its left classifier $\hat{\mathcal{S}}$ is rounded.
\end{prop}
\begin{proof}
Let us start by identifying what being rounded means for $\hat{\mathcal{S}}$. Recall, as in \Cref{rmk:left-class-stable}, that the Yoneda embedding of $\hat{\mathcal{S}}$ into its category of small presheaves has a left adjoint,
%\[L :\mathcal{P}(\hat{\mathcal{S}}) \leftrightarrows \hat{\mathcal{S}} : \yo_{\hat{\mathcal{S}}}\]
% https://q.uiver.app/#q=WzAsMixbMCwwLCJcXG1hdGhjYWx7UH0oXFxoYXR7XFxzU30pIl0sWzEsMCwiXFxoYXR7XFxzU30iXSxbMSwwLCJcXHlvX1xcc1MiLDIseyJvZmZzZXQiOjN9XSxbMCwxLCJMIiwyLHsib2Zmc2V0IjozfV0sWzMsMiwiXFx0b3AiLDEseyJzaG9ydGVuIjp7InNvdXJjZSI6MjAsInRhcmdldCI6MjB9LCJzdHlsZSI6eyJib2R5Ijp7Im5hbWUiOiJub25lIn0sImhlYWQiOnsibmFtZSI6Im5vbmUifX19XV0=
\[\begin{tikzcd}[ampersand replacement=\&]
	{\mathcal{P}(\hat{\sS})} \& {\hat{\sS}.}
	\arrow[""{name=0, anchor=center, inner sep=0}, "{\yo_\sS}"', shift right=3, from=1-2, to=1-1]
	\arrow[""{name=1, anchor=center, inner sep=0}, "L"', shift right=3, from=1-1, to=1-2]
	\arrow["\top"{description}, draw=none, from=1, to=0]
\end{tikzcd}\]
This is a rephrasing of the fact that $\hat{\mathcal{S}}$ is cocomplete (see \cite[Proposition~2.2]{garner2012lex}). It is easy to see that the orthogonality condition in \Cref{costr:left-skt-class}, for the case of $\hat{\mathcal{S}}$ is precisely that given by the adjunction above, so that $\hat{\hat{\mathcal{S}}}$ is nothing but $\hat{\mathcal{S}}$. Thus -- because in this case $J_{\hat{\mathcal{S}}}$ is naturally isomorphic to the identity of $\hat{\mathcal{S}}$ -- to be rounded, coincides with the requirement $\hat{\mathcal{S}}$ is (right) normal, which follows on the spot from \Cref{prop:left-class-of-rounded-is-normal}.
\end{proof}

% \begin{thm}
% Every small sketch is Morita equivalent to a normal small sketch.
% \end{thm}
% \begin{proof}
% Take inspiration from \cite[4.3.3]{borceux_19943}. Notice that this theorem may be false in this form, it seems that you can (functorially!) normalize only one of the classes. The theorem is morally true though, see  \cite[5.6.7]{borceux1994handbook} (this construction is not functorial though, sad sad). 
% \end{proof}

\begin{rem}[Compatibility à la Lack-Tendas] \label{tendaslack}
By inspecting the definition of left roundedness one can see a similarity with that of \textit{compatibility} by Lack and Tendas \cite[Definition~3.5]{lack2024accessible}. It is indeed true that left roundedness encodes the statement that limit cones in $L$ are \textit{compatible} with colimits cocones in $C$, in the language of sketches. The precise sense in which this statement is true is formalized by the theory of lex colimits (which is generalised by the notion of compatibility), that we shall discuss in the next subsection.
\end{rem}

% \section{Sketching sketches}

% In this section we follow the inspiration of \cite{lair1975etude}, where the author shows that the $1$-category of (strict) sketches is the category of models of a limit sketch. We will deliver a $2$-dimensional version of Lair's construction. In the next section, we will see how this result implies the (bi)(co)completeness of the bicategoty of sketches.

\subsection{Lex colimits are left rounded}

This subsection is devoted to the most common example of left roundedness, which is given by the general phenomenology of lex colimits. In particular we study a very specific kind of sketch, these will be categories $\cS$ with finite limits and some specified family of cocones. Thus, the kind of sketch structure we are confronted with is of the form $\sS=(\cS,\skL,\skC)$, where $\skL$ contains precisely all finite limit diagrams.

The general aim of the subsection is to show that the theory developed by Garner and Lack in \cite{garner2012lex} embeds in that of rounded sketches. Before doing so, we start with a simpler example that requires less technology and whose behavior contains the blueprint of the main proposition of the subsection (\Cref{prop:phi-ex-round}).
% \begin{defn}
% Let $\sS=(\cS,\skL,\skC)$ be a sketch. We say that $\sS$ is distributive when, for every diagram $D$ in $\skC$, and every diagram $C$ in $\skL$, the colimit functor $\text{colim}: S^C \to S$ preserves limits of shape $D$.
% \end{defn}

% \red Remark: If $\sS$ is a limit normal sketch, then $\sS$ is distributive. If it is only limit..boh? Probably nope.  \black 

% \red 
% \begin{prop}
%     [MAYBE] A sketch $\sS$ is distributive iff (for any) $\sX$ sketch $\sS$ is $\sX$-compatible with itself [see MULTILINEARITY OF SKETCHES Def 2.10].
% \end{prop}

% \begin{proof}
%     They say that compatibility means that cocones and cones in $\sS$ commute. 
% \end{proof}
% \black 

\begin{exa}[Lex sites are rounded sketches] \label{roundedexample}
Recall that a lex site $(C,J)$ is the data of a lex category $C$, equipped with a Grothendieck pretopology $J$. Following \cite[D2.1.4(g)]{Sketches}, we can transform every lex site into a lex sketch, so that this construction produces a $2$-functor, $\mathsf{site}_{\text{lex}} \to \mathsf{skt}_{\text{lex}}.$ Now, recall the discussion in \Cref{leftsketchclasstopos}, where we observed that \Cref{costr:left-skt-class} is nothing but computing the topos of sheaves over $(C,J)$. This means, in particular, that in the diagram below,

% https://q.uiver.app/#q=WzAsMyxbMCwwLCJcXGNTIl0sWzIsMCwiXFxtYXRoc2Z7UHNofShcXGNTKSJdLFsyLDEsIlxcaGF0e1xcY1N9Il0sWzAsMSwiXFx5byJdLFsyLDEsIkkiLDIseyJjdXJ2ZSI6Miwic3R5bGUiOnsidGFpbCI6eyJuYW1lIjoiaG9vayIsInNpZGUiOiJib3R0b20ifX19XSxbMSwyLCJMIiwyLHsiY3VydmUiOjJ9XSxbMCwyLCJKX1xcY1MiLDJdLFs1LDQsIlxcZGFzaHYiLDEseyJzaG9ydGVuIjp7InNvdXJjZSI6MjAsInRhcmdldCI6MjB9LCJzdHlsZSI6eyJib2R5Ijp7Im5hbWUiOiJub25lIn0sImhlYWQiOnsibmFtZSI6Im5vbmUifX19XSxbNiwxLCI6PSIsMyx7InNob3J0ZW4iOnsic291cmNlIjoyMH0sInN0eWxlIjp7ImJvZHkiOnsibmFtZSI6Im5vbmUifSwiaGVhZCI6eyJuYW1lIjoibm9uZSJ9fX1dXQ==
\[\begin{tikzcd}
	\cS && {\Psh(\cS)} \\
	&& {\hat{\cS}}
	\arrow["\yo", from=1-1, to=1-3]
	\arrow[""{name=0, anchor=center, inner sep=0}, "I"', curve={height=12pt}, hook', from=2-3, to=1-3]
	\arrow[""{name=1, anchor=center, inner sep=0}, "L"', curve={height=12pt}, from=1-3, to=2-3]
	\arrow[""{name=2, anchor=center, inner sep=0}, "{J_\cS}"', from=1-1, to=2-3]
	\arrow["\dashv"{description}, draw=none, from=1, to=0]
	\arrow["{:=}"{marking, allow upside down}, draw=none, from=2, to=1-3]
\end{tikzcd}\]

the reflector $L$ is left exact, because it coincides with the sheafification functor, which is a rephrasing of the fact that the sketch associated to $(C,J)$ is rounded. The diagram below witnesses the fact that the inclusion of lex sites actually lands in rounded (lex) sketches.

% https://q.uiver.app/#q=WzAsMyxbMCwxLCJcXG1hdGhzZntzaXRlfV97XFx0ZXh0e2xleH19Il0sWzEsMSwiXFxtYXRoc2Z7c2t0fV97XFx0ZXh0e2xleH19Il0sWzEsMCwiXFxtYXRoY2Fse3J9XFxtYXRoc2Z7c2t0fV97XFx0ZXh0e2xleH19Il0sWzAsMSwiaSJdLFsyLDFdLFswLDIsIiIsMix7InN0eWxlIjp7ImJvZHkiOnsibmFtZSI6ImRhc2hlZCJ9fX1dXQ==
\[\begin{tikzcd}
	& {r\mathsf{skt}_{\text{lex}}} \\
	{\mathsf{site}_{\text{lex}}} & {\mathsf{skt}_{\text{lex}}}
	\arrow["i", from=2-1, to=2-2]
	\arrow[from=1-2, to=2-2]
	\arrow[dashed, from=2-1, to=1-2]
\end{tikzcd}\]
\end{exa}

\begin{cor} \label{topostologos}
Also the $2$-functor $\mathsf{Topoi}^\circ \to \mathsf{Skt}$ lands in rounded sketches.
\end{cor}
\begin{proof}
This is a corollary of the previous example, when we consider a topos as a canonical site for itself equipped with the canonical topology.

% https://q.uiver.app/#q=WzAsMyxbMSwwLCJcXG1hdGhzZntzaXRlfV97XFx0ZXh0e2xleH19Il0sWzIsMSwiXFxtYXRoc2Z7c2t0fV97XFx0ZXh0e2xleH19Il0sWzAsMSwiXFxtYXRoc2Z7VG9wb2leXFxjaXJjfSJdLFswLDEsImkiXSxbMiwxXSxbMiwwLCJKX3tcXHRleHR7Y2FufX0iXV0=
\[\begin{tikzcd}
	& {\mathsf{MSite}_{\text{lex}}} \\
	{\mathsf{Topoi^\circ}} && {\mathsf{Skt}_{\text{lex}}}
	\arrow["i", from=1-2, to=2-3]
	\arrow[from=2-1, to=2-3]
	\arrow["{J_{\text{can}}}", from=2-1, to=1-2]
\end{tikzcd}\]
\end{proof}

% In order to provide the last class of examples, we shall start by recalling a classical proposition coming from general category theory.

% \begin{prop}
%     The following are equivalent for $S$ a category with colimits of shape $D$.
%     \begin{itemize}
%         \item The colimit functor $S^D \to S$ is lex.
%         \item the algebra structure below,
%         % https://q.uiver.app/#q=WzAsMyxbMCwwLCJTIl0sWzEsMCwiUyJdLFswLDEsIlxcUGhpXkQoUykiXSxbMCwxXSxbMCwyLCJcXHlvbiIsMl0sWzIsMSwiTCIsMl1d
% \[\begin{tikzcd}
% 	S & S \\
% 	{\Phi^D(S)}
% 	\arrow[from=1-1, to=1-2]
% 	\arrow["\yon"', from=1-1, to=2-1]
% 	\arrow["L"', from=2-1, to=1-2]
% \end{tikzcd}\]
%     where $\Phi^D$ is the completion under colimits of shape $D$ is lex.
%     \end{itemize}
% \end{prop}

\begin{exa}[$\lambda$-Topoi à la Espíndola as left rounded sketches]\label{kapparounded}
Similarly to \Cref{topostologos} and \Cref{kappaleft}, the $2$-functor $\mathsf{Topoi}_\lambda^\op \to \Skt$ lands in rounded sketches. The proof of this fact technically does not appear in the literature, so we shall say something about it. Similarly to the standard case, a $\lambda$-topos is a $\lambda$-site for itself, this follows from the usual proof (\cite[C2.3.9]{Sketches} and \cite[C2.2.7]{Sketches}) the only $\lambda$-aspect of it is whether the canonical topology is a $\lambda$-topology. That is precisely the definition of a $\lambda$-topos. Then, the fact that the sheafification functor is $\lambda$-lex (which ensures that the sketch associated to a $\lambda$-topos is rounded) follows \cite[Prop 3.6]{espindola2023every}.
\end{exa}

We are now ready to discuss the theory of lex colimits. In the construction below we shall recall what is needed to deliver our result, but we assume some familiarity with the content of \cite{garner2012lex}.

\begin{constr}[From exact categories to lex sketches]\label{funct-phi-ex-to-skt}
Let $\Phi$ be a class of lex-weights, i.e. a collection of functors $\phi\colon\cD^{op}\to\Cat$, with each $\cD$ small and finitely complete. We will denote with $\PhiEx$ the category of $\Phi$-exact categories (see \cite[Section~3]{garner2012lex}). 
% These come equipped with a forgetful functor $\PhiEx\to\lex$. 
% Postcomposing with the functor $\lex\to\skt$ defined in \Cref{doctrines}, we get a 2-functor as below. 
% $$\PhiEx\to\skt_{\text{lex}}$$ 
Similarly to \Cref{doctrines} we can construct a 2-functor as below.
$$I_\Phi\colon\PhiEx\to\skt_{\text{lex}}$$
Precisely, given a $\Phi$-exact category $\cC$ we can send this to the sketch $I_\Phi\cC$ defined:
\begin{itemize}
    \item as underlying category we take $\cC$ itself;
    \item we define the limit part $\skL_{I_\Phi\cC}$ as all (essentially small) limit diagrams;
    \item as colimit part $\skC_{I_\Phi\cC}$  we take all the colimit diagrams specified by $\Phi$, i.e. we define $\skC_{I_\Phi\cC}$ as follows
    $$\skC_{I_\Phi\cC}:=\lbrace \textrm{El}(\varphi)\to\cD\xrightarrow{D}\cC\,\mid\,\textrm{with}\,\cD\,\textrm{lex and}\,\varphi\in\Phi\rbrace.$$ 
\end{itemize}
We recall that a morphism $F\colon\cC\to\cD$ in $\PhiEx$ (see \cite[Page~12]{garner2012lex}) is a functor which preserves finite limits and the colimits specified by $\Phi$. 
This is the same as saying that the functor $F$ is a sketch morphism from $I_\Phi\cC$ to $I_\Phi\cD$. 

%\red 
%[GAB: Not sure about the notation for $\PhiEx\to\skt$] \black 

\end{constr}

\begin{rem}We underline that, even if there is a forgetful functor $\PhiEx\to\lex$, the 2-functor $I_\Phi$ defined in \Cref{funct-phi-ex-to-skt} does not factor through the functor $\lex\to\skt$ defined in \Cref{doctrines}. The situation is -- by no coincidence -- identical to that displayed in point (e) of \Cref{doctrines}.

% https://q.uiver.app/#q=WzAsMyxbMCwwLCJcXFBoaUV4Il0sWzEsMSwiXFxza3QiXSxbMiwwLCJcXGxleCJdLFsyLDFdLFswLDEsIklfe1xcUGhpfSIsMl0sWzAsMl0sWzMsNCwiIiwwLHsic2hvcnRlbiI6eyJzb3VyY2UiOjIwLCJ0YXJnZXQiOjIwfX1dXQ==
\[\begin{tikzcd}[ampersand replacement=\&]
	\PhiEx \&\& \lex \\
	\& \skt
	\arrow[""{name=0, anchor=center, inner sep=0}, from=1-3, to=2-2]
	\arrow[""{name=1, anchor=center, inner sep=0}, "{I_{\Phi}}"', from=1-1, to=2-2]
	\arrow[from=1-1, to=1-3]
	\arrow[shorten <=15pt, shorten >=15pt, Rightarrow, from=0, to=1]
\end{tikzcd}\]
\end{rem}

\begin{prop}\label{prop:phi-ex-round}
The $2$-functor $I_\Phi\colon\PhiEx\to\skt$ defined in \Cref{funct-phi-ex-to-skt} factors through the 2-category $\roundSkt$ of rounded sketches.
\end{prop}

\begin{proof}
Let us recall that  in \cite[Section~7,~Page~31]{garner2012lex} they define, for any class of lex-weights $\Phi$ and any small lex-category $\cC$, $P_\Phi\cC$ as the full subcategory of  $\Psh(\cC)$  spanned by functors $F\colon\cC^\circ\to\Set$ sending any $\Phi^\ast$-lex-colimit in $\cC$ to a limit in $\Set$. 
Therefore, using the characterisation of the underlying category of the left sketch classifier given in \Cref{turnaroundsketch}, it follows that if we start with a small $\Phi$-exact category $\cC$, then the underlying category of $\widehat{I_\Phi\cC}$ is exactly $P_\Phi\cC$.
Then, \cite[Proposition~7.3]{garner2012lex} says that $L\colon\Psh(\cC)\to\widehat{I_\Phi\cC}=P_\Phi\cC$ preserves finite limits. Since the Yoneda embedding preserves limits, the functor $J_{I_\Phi\cC}=L\yo\colon\cC\to \widehat{I_\Phi\cC}=P_\Phi$ preserves finite limits. We conclude that $I_\Phi\cC$ is rounded, since cones in $\skL_{I_\Phi\cC}$ are precisely  finite limit diagrams, which are sent to limit diagrams by $L$. 

% \red SISTEMARE QUA SOTTO \black 

% Using \Cref{turnaroundsketch} [LAST LINE OF THE PROOF, ma in realtà è letteralmente lo statement] we can see that the construction $P_\Phi\cC$ in \cite[Section~7]{garner2012lex} is the same as our $\hat{\cC}$ for $\PhiEx$ [SCRIVERE MEGLIO]. 

% Then, \cite[Proposition~7.3]{garner2012lex} says that $L\colon\hat{\cC}\to\Psh(\cC)$ preserves finite limits and \cite[Corollary~7.4]{garner2012lex} $J_\sS$ is a sketch morphism. \\

% \red OLD WRONG PROOF \black

% The idea is that the left classifier construction coincides with $\Phi_l$ in \cite[Section~3]{garner2012lex} when restricted to $\PhiEx$. 
% In fact, once we would have proven this, then we could use \cite[Proposition~3.3]{garner2012lex} to show that $J_\cC$ must be a right adjoint, and therefore preserve the limits specified by $I_\Phi\cC$, i.e. all (essentially small) limits.  
% {Ivan warning!\color{red} Non ho aperto Sec 3 di GL12, ma riesco ad immaginare pochi esempi in cui J sia un aggiunto destro. Nella parte maggiore dei casi quello che dovrebbe succedere è che L preserva i limiti che servono "by default"}

% Now we turn to $\hat{I_\Phi\cC}$ for a $\Phi$-exact category. 
% \red IDEA: Use \cite[Proposition~3.1]{garner2012lex} where they state that $\Phi_l\cC$ is the closure of representable under finite limits and $\Phi$-lex-colimits. \black 
\end{proof}

It is interesting to notice that \cite[Corollary~7.4]{garner2012lex} corresponds in our setting to the fact that $J_{I_\Phi\cC}$ is a sketch morphism. We will show more connections between our theory and the one of lex-colimits in the last section.

\begin{rem}[First order doctrines are rounded sketches]
Finally, recall from \cite[Sec. 5]{garner2012lex} that all the  \Cref{doctrines}  are indeed $2$-categories of $\Phi$-exact categories, and their associated sketch structure clearly coincides with that described in \Cref{funct-phi-ex-to-skt}, thus we obtain that all the $2$-functors presented in \Cref{doctrines} land in rounded sketches.

% https://q.uiver.app/#q=WzAsNSxbMiwxLCJyXFxza3Rfe1xcdGV4dHtsZXh9fSJdLFszLDAsIlxcbGV4Il0sWzEsMCwiXFxyZWciXSxbMCwwLCJcXGNvaCJdLFs0LDAsIlxcY3Byb2QiXSxbMywwLCIiLDIseyJjdXJ2ZSI6MSwic3R5bGUiOnsidGFpbCI6eyJuYW1lIjoibW9ubyJ9fX1dLFsyLDAsIiIsMix7InN0eWxlIjp7InRhaWwiOnsibmFtZSI6Im1vbm8ifX19XSxbMSwwLCIiLDIseyJzdHlsZSI6eyJ0YWlsIjp7Im5hbWUiOiJtb25vIn19fV0sWzQsMCwiIiwwLHsiY3VydmUiOi0xLCJzdHlsZSI6eyJ0YWlsIjp7Im5hbWUiOiJtb25vIn19fV0sWzMsMl0sWzEsNF0sWzIsMV0sWzgsNywiIiwwLHsic2hvcnRlbiI6eyJzb3VyY2UiOjIwLCJ0YXJnZXQiOjIwfX1dLFs2LDUsIiIsMix7InNob3J0ZW4iOnsic291cmNlIjoyMCwidGFyZ2V0IjoyMH19XSxbNyw2LCIiLDIseyJzaG9ydGVuIjp7InNvdXJjZSI6MjAsInRhcmdldCI6MjB9fV1d
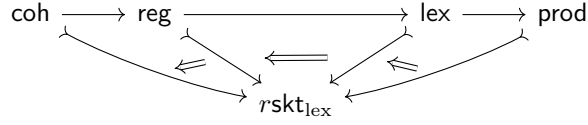
\[\begin{tikzcd}[ampersand replacement=\&]
	\coh \& \reg \&\& \lex \& \cprod \\
	\&\& {r\skt_{\text{lex}}}
	\arrow[""{name=0, anchor=center, inner sep=0}, curve={height=6pt}, tail, from=1-1, to=2-3]
	\arrow[""{name=1, anchor=center, inner sep=0}, tail, from=1-2, to=2-3]
	\arrow[""{name=2, anchor=center, inner sep=0}, tail, from=1-4, to=2-3]
	\arrow[""{name=3, anchor=center, inner sep=0}, curve={height=-6pt}, tail, from=1-5, to=2-3]
	\arrow[from=1-1, to=1-2]
	\arrow[from=1-4, to=1-5]
	\arrow[from=1-2, to=1-4]
	\arrow[shorten <=6pt, shorten >=6pt, Rightarrow, from=3, to=2]
	\arrow[shorten <=6pt, shorten >=6pt, Rightarrow, from=1, to=0]
	\arrow[shorten <=15pt, shorten >=15pt, Rightarrow, from=2, to=1]
\end{tikzcd}\]
\black 
\end{rem}

\section{Logoi}\label{sec:class-logoi}

The word \textit{logos} has an eventful history in category theory. %and our choice of terminology makes a statement in this scientific debate. [MAYBE BETTER: We will explain why we decide to use this name in our context. 
%Potremmo togliere la parte in rosso e mettere tipo "spieghiamo perché ha senso per noi estendere il significato" che forse è più DC
Let us start by framing the evolution of the use of such word.
As a general statement, each fragment of first order logic has a name: \textit{equational, regular, disjunctive, coherent, geometric}, etc. Among category theorists these are called \textit{doctrines} (of first order logic).

The first use of the word logos was suggested by Freyd and Scedrov in \cite{freyd1990categories}. Nowadays we would call \textit{Heyting} category what they call logos and the fragment of logic they capture is intuitionistic first order logic. This choice did not have much success in the literature, and nowadays the name Heyting category seems more appropriate as it does not put intuitionistic first order logic in such a central (and hard to justify) role.

It was later suggested by Joyal, and strongly popularized by his collaboration with Anel \cite{anel2021topo}, that a logos should be an object in the opposite $2$-category of topoi. This point of view emphasise on the logico-geometric duality of topos theory, and chooses some evocative name to refer to geometric logic.

Our feeling is that as much as the word \textit{topos} was chosen to refer to a quite general notion of \textit{place}, the word \textit{logos} should be used to refer to a quite broad notion of theory, which is not restricted to any fragment of first order logic. In this sense, both Anel-Joyal's choice (which stresses on the relevance of geometric logic) and Freyd-Scedrov's choice seem not general enough. 

For example, $\lambda$-topoi in the sense of Espíndola \cite{Espindola2019infinitary,espindola2023every} offer classifying objects for $\lambda$-geometric logic, and we would like to have a theory where they can be considered logoi too. And even more, we would want a framework that a priori can encompass any variation of infinitary first order logic that comes to mind, beyond those of \textit{geometric taste}. Of course the theory of left sketches offers the perfect environment where all $\lambda$-topoi interact at the same time (\Cref{kappaleft}).

% https://q.uiver.app/#q=WzAsNixbMCwwLCJcXG1hdGhzZntUb3BvaX1ee1xcdGV4dHtvcH19Il0sWzIsMCwiLi4uIl0sWzMsMCwiXFxtYXRoc2Z7VG9wb2l9X3tcXGxhbWJkYX1ee1xcdGV4dHtvcH19Il0sWzIsMSwiXFxMZWZ0c2t0Il0sWzQsMCwiLi4uIl0sWzEsMCwiXFxtYXRoc2Z7VG9wb2l9X3tcXGFsZXBoXzF9XntcXHRleHR7b3B9fSJdLFsyLDFdLFs0LDJdLFs1LDBdLFsxLDVdLFswLDMsIiIsMCx7ImN1cnZlIjoxfV0sWzUsM10sWzEsM10sWzIsM10sWzQsMywiIiwwLHsiY3VydmUiOi0xfV0sWzEwLDExLCIiLDAseyJzaG9ydGVuIjp7InNvdXJjZSI6NDAsInRhcmdldCI6NDB9fV0sWzExLDEyLCIiLDAseyJzaG9ydGVuIjp7InNvdXJjZSI6NDAsInRhcmdldCI6NDB9fV0sWzEyLDEzLCIiLDAseyJzaG9ydGVuIjp7InNvdXJjZSI6NDAsInRhcmdldCI6NDB9fV0sWzEzLDE0LCIiLDAseyJzaG9ydGVuIjp7InNvdXJjZSI6NDAsInRhcmdldCI6NDB9fV1d
\[\begin{tikzcd}[ampersand replacement=\&]
	{\mathsf{Topoi}^{\text{op}}} \& {\mathsf{Topoi}_{\aleph_1}^{\text{op}}} \& {...} \& {\mathsf{Topoi}_{\lambda}^{\text{op}}} \& {...} \\
	\&\& \leftSkt
	\arrow[from=1-4, to=1-3]
	\arrow[from=1-5, to=1-4]
	\arrow[from=1-2, to=1-1]
	\arrow[from=1-3, to=1-2]
	\arrow[""{name=0, anchor=center, inner sep=0}, curve={height=6pt}, from=1-1, to=2-3]
	\arrow[""{name=1, anchor=center, inner sep=0}, from=1-2, to=2-3]
	\arrow[""{name=2, anchor=center, inner sep=0}, from=1-3, to=2-3]
	\arrow[""{name=3, anchor=center, inner sep=0}, from=1-4, to=2-3]
	\arrow[""{name=4, anchor=center, inner sep=0}, curve={height=-6pt}, from=1-5, to=2-3]
	\arrow[shorten <=9pt, shorten >=9pt, Rightarrow, from=0, to=1]
	\arrow[shorten <=9pt, shorten >=9pt, Rightarrow, from=1, to=2]
	\arrow[shorten <=9pt, shorten >=9pt, Rightarrow, from=2, to=3]
	\arrow[shorten <=9pt, shorten >=9pt, Rightarrow, from=3, to=4]
\end{tikzcd}\]

Yet, we also need to account for the Giraud-like axioms that these left sketches have to verify, and this will be taken into account via the technology developed in the previous section.
So, for us, a \textit{logos} is the \textit{generic semantics} of some first order theory sitting in some fragment of first order logic, in a spectrum that spans from (finitary) equational to $\lambda$-geometric logic (in the sense of Espíndola), without making any commitment on the specific fragment (as long as first order).  The main result of the section is \Cref{diaconescu}, which provides a Diaconescu-type theorem for (left) rounded sketches and logoi that fully generalises Diaconescu theorem for topoi, and any possible geralization of it for $\lambda$-topoi (which for the moment is not in the literature).

\begin{defn}[Logos]\label{def:logos}
A logos $\sS$ is a left rounded left sketch. \\ %\red [Not normal? The classifyin logos is normal by \Cref{prop:left-class-of-rounded-is-normal}. BUT Gabriele, look at the next Cor!] \black 
We define the $2$-categories $\logoi$, $\Log^\mathsf{M}$ and $\Log$ as the full sub-2-categories of $\skt$, $\Skt^\mathsf{M}$ and $\Skt$ with objects small, Morita small and locally small logoi. 
\end{defn}

In most of the relevant examples of this paper, a logos will be Morita small. 
Nonetheless, this assumption is often not needed to show the most relevant properties of logoi. 

This section is, almost by definition, pretty much about the intersection of the previous two sections. From a technical point of view, there is not much that remains to be proven, the main aim of the section is to state the version of Diaconescu theorem for logoi and rounded sketches (\Cref{diaconescu}) and make some final considerations on the notion of logos.

\begin{cor} \label{normaloni}
Logoi are normal sketches. 
\end{cor}
\begin{proof}
Because they are left rounded we know that their left classifier is normal (\Cref{prop:left-class-of-rounded-is-normal}). Because they are left, they coincide with their left classifier (\Cref{rem:left-norm-no-like-tilde}), hence they are normal. 
\end{proof}

\begin{rem}[A simpler definition of logos?]\label{rmk:logoi-as-loc-pres}
It follows from the corollary above and the definition of logos, that a left Morita small logos is the same as a couple $(\sL, \mathsf{L})$ where $\sL$ is a locally presentable category and $\mathsf{L}$ is a class of diagrams that are \textit{compatible} (in the sense of \Cref{tendaslack} and \Cref{def:rounded-skt}) with all colimits.  A morphism of logoi then is nothing but a cocontinuous functor (or a left adjoint -- by the adjoint functor theorem --) preserving the specified collection of limits.
\end{rem}

% \begin{rem}
%     The underlying category of a logos is always complete.
% \end{rem}

\begin{exa}[Topoi \`a la Anel-Joyal and $\lambda$-Topoi à la Espìndola]
Of course the $2$-category of logoi \`a la Anel-Joyal (which is \textit{just} the opposite of the $2$-category of Topoi) embeds fully faithfully into Logoi. This follows directly from  \Cref{topostoleft} and \Cref{topostologos}. Similarly to the case of topoi, $\lambda$-topoi are logoi by \Cref{kappaleft} and \Cref{kapparounded}. This observation finally gives us the diagram that appeared in the introduction of the paper. More is true, all the resulting logoi are left Morita small, almost by definition of topos.

% https://q.uiver.app/#q=WzAsNixbMCwwLCJcXG1hdGhzZntUb3BvaX1ee1xcdGV4dHtvcH19Il0sWzIsMCwiLi4uIl0sWzMsMCwiXFxtYXRoc2Z7VG9wb2l9X3tcXGxhbWJkYX1ee1xcdGV4dHtvcH19Il0sWzIsMSwiXFxtYXRoc2Z7TG9nb2l9Il0sWzQsMCwiLi4uIl0sWzEsMCwiXFxtYXRoc2Z7VG9wb2l9X3tcXGFsZXBoXzF9XntcXHRleHR7b3B9fSJdLFsyLDFdLFs0LDJdLFs1LDBdLFsxLDVdLFswLDMsIiIsMCx7ImN1cnZlIjoxfV0sWzUsM10sWzEsM10sWzIsM10sWzQsMywiIiwwLHsiY3VydmUiOi0xfV0sWzEwLDExLCIiLDAseyJzaG9ydGVuIjp7InNvdXJjZSI6NDAsInRhcmdldCI6NDB9fV0sWzExLDEyLCIiLDAseyJzaG9ydGVuIjp7InNvdXJjZSI6NDAsInRhcmdldCI6NDB9fV0sWzEyLDEzLCIiLDAseyJzaG9ydGVuIjp7InNvdXJjZSI6NDAsInRhcmdldCI6NDB9fV0sWzEzLDE0LCIiLDAseyJzaG9ydGVuIjp7InNvdXJjZSI6NDAsInRhcmdldCI6NDB9fV1d
\[\begin{tikzcd}[ampersand replacement=\&]
	{\mathsf{Topoi}^{\text{op}}} \& {\mathsf{Topoi}_{\aleph_1}^{\text{op}}} \& {...} \& {\mathsf{Topoi}_{\lambda}^{\text{op}}} \& {...} \\
	\&\& {\Log^\moritaM}
	\arrow[from=1-4, to=1-3]
	\arrow[from=1-5, to=1-4]
	\arrow[from=1-2, to=1-1]
	\arrow[from=1-3, to=1-2]
	\arrow[""{name=0, anchor=center, inner sep=0}, curve={height=6pt}, from=1-1, to=2-3]
	\arrow[""{name=1, anchor=center, inner sep=0}, from=1-2, to=2-3]
	\arrow[""{name=2, anchor=center, inner sep=0}, from=1-3, to=2-3]
	\arrow[""{name=3, anchor=center, inner sep=0}, from=1-4, to=2-3]
	\arrow[""{name=4, anchor=center, inner sep=0}, curve={height=-6pt}, from=1-5, to=2-3]
	\arrow[shorten <=9pt, shorten >=9pt, Rightarrow, from=0, to=1]
	\arrow[shorten <=7pt, shorten >=7pt, Rightarrow, from=1, to=2]
	\arrow[shorten <=7pt, shorten >=7pt, Rightarrow, from=2, to=3]
	\arrow[shorten <=9pt, shorten >=9pt, Rightarrow, from=3, to=4]
\end{tikzcd}\]

\end{exa}

% \subsection{Logoi as a left Kan injectivity class}

\subsection{Classifying logoi}

Finally we can present our notion of classifying logos for a rounded sketch. Most of the work was done in the previous sections, and thus this subsection will be dedicated to mostly state the desired results and round off the work.

\begin{notat}[Logoi of fake sheaves over a rounded sketch]
Let $\sS$ be a left rounded small sketch. We may call its left classifier $\hat{\sS}$ the \textit{category of fake sheaves over} $\sS$, or -- more conceptually -- the \textit{classifying logos} of $\sS$. The reason for this choice of name is essentially explained by our pet example of sites and topoi.
\end{notat}

\begin{thm}[Diaconescu for Logoi] \label{diaconescu}

    $\mathsf{Log}^\moritaM$ is (bi)reflective in $r\mathsf{Skt}^\moritaM$.
    % https://q.uiver.app/#q=WzAsMixbMSwwLCJcXExvZ15cXG1vcml0YU0iXSxbMCwwLCJcXHJvdW5kU2t0XlxcbW9yaXRhTSJdLFswLDEsIlUiLDIseyJjdXJ2ZSI6Mn1dLFsxLDAsIlxcbWF0aGNhbHtDfWxbLV0iLDIseyJjdXJ2ZSI6Mn1dLFszLDIsIlxcdG9wIiwxLHsic2hvcnRlbiI6eyJzb3VyY2UiOjIwLCJ0YXJnZXQiOjIwfSwic3R5bGUiOnsiYm9keSI6eyJuYW1lIjoibm9uZSJ9LCJoZWFkIjp7Im5hbWUiOiJub25lIn19fV1d
\[\begin{tikzcd}[ampersand replacement=\&]
	{\roundSkt^\moritaM} \& {\Log^\moritaM}
	\arrow[""{name=0, anchor=center, inner sep=0}, "U"', curve={height=12pt}, from=1-2, to=1-1]
	\arrow[""{name=1, anchor=center, inner sep=0}, "{\mathcal{C}l[-]}"', curve={height=12pt}, from=1-1, to=1-2]
	\arrow["\top"{description}, draw=none, from=1, to=0]
\end{tikzcd}\]
\end{thm}

\begin{proof}

    The idea of this theorem is to restrict the (bi)adjunction of \Cref{diacweak} to rounded sketches. 
    First of all, let us define the pseudofunctor $\mathcal{C}l[-]\colon\roundSkt^\moritaM\to\Log^\moritaM$. 
    The only thing that we have to notice is that the left classifier of a rounded sketch is a logos since it is left by construction and it is rounded by \Cref{prop:rounded-then-left-class-rounded}. 
    Hence, since the inclusions $\roundSkt^\moritaM\hookrightarrow\Skt^\moritaM$ and $\Log^\moritaM\hookrightarrow\leftSkt^\moritaM$ are fully faithful, $\hat{(-)}$ restricts to $\mathsf{C}l[-]:=\hat{(-)}\colon\roundSkt^\moritaM\to\Log^\moritaM$. 
% https://q.uiver.app/#q=WzAsNCxbMCwyLCJcXHJvdW5kU2t0XlxcbW9yaXRhTSJdLFsyLDIsIlxcTG9nXlxcbW9yaXRhTSJdLFswLDAsIlxcU2t0XlxcbW9yaXRhTSJdLFsyLDAsIlxcbGVmdFNrdF5cXG1vcml0YU0iXSxbMCwxLCJcXG1hdGhzZntDfWxbLV0iLDJdLFswLDIsIiIsMix7InN0eWxlIjp7InRhaWwiOnsibmFtZSI6Imhvb2siLCJzaWRlIjoidG9wIn19fV0sWzEsMywiIiwwLHsic3R5bGUiOnsidGFpbCI6eyJuYW1lIjoiaG9vayIsInNpZGUiOiJ0b3AifX19XSxbMiwzLCJcXGhhdHsoLSl9IiwyXSxbMywyLCJVXkwiLDIseyJjdXJ2ZSI6Mywic3R5bGUiOnsiYm9keSI6eyJuYW1lIjoiZGFzaGVkIn19fV0sWzEsMCwiVSIsMix7ImN1cnZlIjozLCJzdHlsZSI6eyJib2R5Ijp7Im5hbWUiOiJkYXNoZWQifX19XSxbOCw3LCJcXHRvcCIsMSx7InNob3J0ZW4iOnsic291cmNlIjoyMCwidGFyZ2V0IjoyMH0sInN0eWxlIjp7ImJvZHkiOnsibmFtZSI6Im5vbmUifSwiaGVhZCI6eyJuYW1lIjoibm9uZSJ9fX1dLFs0LDksIj8iLDEseyJzaG9ydGVuIjp7InNvdXJjZSI6MjAsInRhcmdldCI6MjB9LCJzdHlsZSI6eyJib2R5Ijp7Im5hbWUiOiJub25lIn0sImhlYWQiOnsibmFtZSI6Im5vbmUifX19XV0=
\[\begin{tikzcd}[ampersand replacement=\&]
	{\Skt^\moritaM} \&\& {\leftSkt^\moritaM} \\
	\\
	{\roundSkt^\moritaM} \&\& {\Log^\moritaM}
	\arrow[""{name=0, anchor=center, inner sep=0}, "{\mathsf{C}l[-]}"', from=3-1, to=3-3]
	\arrow[hook, from=3-1, to=1-1]
	\arrow[hook, from=3-3, to=1-3]
	\arrow[""{name=1, anchor=center, inner sep=0}, "{\hat{(-)}}"', from=1-1, to=1-3]
	\arrow[""{name=2, anchor=center, inner sep=0}, "{U^L}"', curve={height=18pt}, dashed, from=1-3, to=1-1]
	\arrow[""{name=3, anchor=center, inner sep=0}, "U"', curve={height=18pt}, dashed, from=3-3, to=3-1]
	\arrow["\top"{description}, draw=none, from=2, to=1]
	\arrow["{?}"{description}, draw=none, from=0, to=3]
\end{tikzcd}\]
    Finally, it is enough to notice that the vertical inclusions are fully faithful, and therefore the (bi)adjunction above $\hat{(-)}\dashv U^L$ (from \Cref{diacweak}) restricts to one below $\mathsf{C}l[-]\dashv J$.  We show it explicitly through the chain of equivalences below, with $\sS\in\roundSkt^\moritaM$ and $\sT\in\Log^\moritaM$.  
    \begin{center}
$\roundSkt(\sS,J\sT)\cong\Skt^\moritaM(\sS,U^L\sT)\simeq\leftSkt^\moritaM(\hat{\sS},\sT)\cong\Log^\moritaM(\mathsf{C}l[\sS],\sT)$ 
\footnote{With abuse of notation we will write $U^L\sT$ for the forgetful 2-functor $\Log^\moritaM\to\Skt^\moritaM$, which can be obtained by $J$ followed by the inclusion $\roundSkt^\moritaM\hookrightarrow\Skt^\moritaM$ or by postcomposing the inclusion $\Log^\moritaM\hookrightarrow\leftSkt^\moritaM$ with $U^L$.}
    \end{center}
    
\end{proof}

\begin{rem}[One Diaconescu to rule them all]\label{rmk:one-diaconescu-all}
Let us conclude the paper with some remarks on the Diaconescu's theorem above. Of course, our result recovers the original Diaconescu Theorem, but also encodes \cite[Theorem~7.5]{garner2012lex}, as summarised in the diagram below. Notice that every $\Phi$-exact category admits a site structure making the diagram (not considering the dashed arrows) below  commutative. 
% https://q.uiver.app/#q=WzAsNixbMCwwLCJcXFBoaV5sXFx0ZXh0ey19XFxtYXRoc2Z7RXh9Il0sWzIsMCwiXFxtYXRoc2Z7TVNpdGV9Il0sWzAsMiwiXFxpbmZ0eVxcdGV4dHstfVxcbWF0aHNme0V4fV5cXG1vcml0YU0iXSxbMiwyLCJcXG1hdGhzZntUb3BvaX1ee1xcdGV4dHtvcH19Il0sWzQsMCwiclxcU2t0XlxcbW9yaXRhTSJdLFs0LDIsIlxcTG9nXlxcbW9yaXRhTSJdLFswLDIsIlBfe1xcUGhpfSIsMCx7ImN1cnZlIjotM31dLFsxLDMsIlxcbWF0aHNme1NofSIsMCx7ImN1cnZlIjotM31dLFsyLDNdLFsxLDRdLFs0LDUsIlxcbWF0aGNhbHtDfWxbLV0iLDAseyJjdXJ2ZSI6LTN9XSxbMywxLCJKIiwwLHsiY3VydmUiOi0zLCJzdHlsZSI6eyJib2R5Ijp7Im5hbWUiOiJkYXNoZWQifX19XSxbNSw0LCJKIiwwLHsiY3VydmUiOi0zLCJzdHlsZSI6eyJib2R5Ijp7Im5hbWUiOiJkYXNoZWQifX19XSxbMiwwLCJVIiwwLHsiY3VydmUiOi0zLCJzdHlsZSI6eyJib2R5Ijp7Im5hbWUiOiJkYXNoZWQifX19XSxbMyw1XSxbMCwxXSxbNiwxMywiIiwxLHsibGV2ZWwiOjEsInN0eWxlIjp7Im5hbWUiOiJhZGp1bmN0aW9uIn19XSxbNywxMSwiIiwxLHsibGV2ZWwiOjEsInN0eWxlIjp7Im5hbWUiOiJhZGp1bmN0aW9uIn19XSxbMTAsMTIsIiIsMSx7ImxldmVsIjoxLCJzdHlsZSI6eyJuYW1lIjoiYWRqdW5jdGlvbiJ9fV1d
\[\begin{tikzcd}[ampersand replacement=\&]
	{\Phi\text{-}\mathsf{ex}} \&\& {\mathsf{MSite}} \&\& {r\Skt^\moritaM} \\
	\\
	{\infty\text{-}\mathsf{Ex}^\moritaM} \&\& {\mathsf{Topoi}^{\text{op}}} \&\& {\Log^\moritaM}
	\arrow[""{name=0, anchor=center, inner sep=0}, "{P_{\Phi}}", curve={height=-18pt}, from=1-1, to=3-1]
	\arrow[""{name=1, anchor=center, inner sep=0}, "{\mathsf{Sh}}", curve={height=-18pt}, from=1-3, to=3-3]
	\arrow[from=3-1, to=3-3]
	\arrow[from=1-3, to=1-5]
	\arrow[""{name=2, anchor=center, inner sep=0}, "{\mathcal{C}l[-]}", curve={height=-18pt}, from=1-5, to=3-5]
	\arrow[""{name=3, anchor=center, inner sep=0}, "J", curve={height=-18pt}, dashed, from=3-3, to=1-3]
	\arrow[""{name=4, anchor=center, inner sep=0}, "J", curve={height=-18pt}, dashed, from=3-5, to=1-5]
	\arrow[""{name=5, anchor=center, inner sep=0}, "U", curve={height=-18pt}, dashed, from=3-1, to=1-1]
	\arrow[from=3-3, to=3-5]
	\arrow[from=1-1, to=1-3]
	\arrow["\dashv"{anchor=center, rotate=-180}, draw=none, from=0, to=5]
	\arrow["\dashv"{anchor=center, rotate=180}, draw=none, from=1, to=3]
	\arrow["\dashv"{anchor=center, rotate=-180}, draw=none, from=2, to=4]
\end{tikzcd}\]
Let us clarify that the leftmost part of the diagram above is imprecise and not present in the literature.
Indeed \cite[Theorem~7.5]{garner2012lex} only provides a relative adjunction (due to size issues), so that there is no functor $U\colon\infty\text{-}\mathsf{Ex}\to\Phi\text{-}\mathsf{ex}$ going from bottom to top. 
By adapting the results in \cite{garner2012lex} to \emph{Morita small exact categories}, one would obtain such a functor and the discussion would carry.

%The Lack of a discussion about Morita small exact categories in \cite{garner2012lex} is mostly a choice of the authors, and indeed such treatment would be possible.}}
\end{rem}

\subsection{Anatomy of a \cancel{topos} logos}

As we mentioned in the previous sections, the theory of logoi was designed to offer a more general framework than that of topoi, one that could provide a treatment of more expressive logics. As a side product of this effort, we get a better understanding of what properties of a topos allow for some very important constructions. We shall end the paper with this brief subsection, highlighting three situations in which the $2$-category of topoi behaves significantly better than the $2$-category of (Morita small) logoi.

\begin{rem}[Modular logoi]\label{rmk:modular-log} 
Recall that by \Cref{rmk:logoi-as-loc-pres} a logos $\sL$ can be understood as the specification of a cocomplete category, and a family $\skL$ of limits that are \textit{exact}. Let $\Phi$ be a collection of diagrams. Define \[\Log^{\Phi}\hookrightarrow\Log\] as the full sub $2$-category of those logoi that contain $\Phi$-limits as part of the exact family $\skL$. 
For instance, the the $2$-category of topoi is contained in $\mathsf{Log}^{\Phi}$ for $\Phi$ the family of finite limits.
The bigger $\Phi$ is, the more properties $\mathsf{Log}^{\Phi}$ has. 
In the next subsections we will consider when $\Phi$ consists of reflexive equalisers, pullbacks and finite product, taking inspiration from the case of topoi, where of course these limits are included in the class $\Phi$.
\end{rem}

\subsubsection{Reflexive equalisers: the (localization, conservative) factorization system}

The $2$-category of topoi has a well-known factorization system given by geometric surjections and geometric embeddings \cite[A4.2]{Sketches}. Its dual, in the opposite $2$-category is given by (localization, conservative). Let us briefly recall how (and why) it works. Consider a cocontinuous lex functor $f^*\colon \mathcal{E} \to \mathcal{F}$ between topoi. Then, one can factorize it as follows.

% https://q.uiver.app/#q=WzAsMyxbMCwwLCJcXG1hdGhjYWx7RX0iXSxbMiwwLCJcXG1hdGhjYWx7Rn0iXSxbMSwxLCJcXG1hdGhzZntDb2FsZ30oZl4qZl8qKSJdLFswLDEsImZeKiJdLFsyLDEsIlxcbWF0aHNme1V9XipfZiIsMix7InN0eWxlIjp7InRhaWwiOnsibmFtZSI6Im1vbm8ifSwiYm9keSI6eyJuYW1lIjoiZGFzaGVkIn19fV0sWzAsMiwiUV4qX2YiLDIseyJzdHlsZSI6eyJib2R5Ijp7Im5hbWUiOiJkYXNoZWQifSwiaGVhZCI6eyJuYW1lIjoiZXBpIn19fV1d
\[\begin{tikzcd}[ampersand replacement=\&]
	{\mathcal{E}} \&\& {\mathcal{F}} \\
	\& {\mathsf{Coalg}(f^*f_*)}
	\arrow["{f^*}", from=1-1, to=1-3]
	\arrow["{\mathsf{U}^*_f}"', dashed, tail, from=2-2, to=1-3]
	\arrow["{Q^*_f}"', dashed, two heads, from=1-1, to=2-2]
\end{tikzcd}\]

It is well known that $\mathsf{Coalg}(f^*f_*)$ is a Grothendieck topos, and it is clear that the forgetful functor $U^*_f$ is cocontinuous, lex and conservative. The functor $Q_f^*$ is induced by the universal property of the category of coalgebras, and is a localization, in the sense that its right adjoint is fully faithful. This factorization is then essentially unique \textit{because} every conservative cocontinuous lex functor is comonadic. This fact follows directly for the fact that it is lex, and thus the condition in Beck (co)monadicity theorem is trivially verified. Now, the two notions of morphism that participate to this factorization are perfectly available in the $2$-category of logoi. 

\begin{defn}[Conservative morphisms and localizations]
Let $F\colon \sL \to \sT$ be a morphism of Morita small logoi.
\begin{itemize}
    \item $F$ is a conservative morphism of logoi if its underlying functor is conservative and faithful. 
    \item $F$ is a localization if its right adjoint\footnote{We recall that a morphism between Morita small logoi is automatically a left adjoint, see \Cref{rmk:logoi-as-loc-pres}.} is fully faithful. 
\end{itemize}
\end{defn}

The same construction, i.e. using $\mathsf{Coalg}$, will lead to \textit{a} factorization of a morphism of logoi into a localization followed by a conservative functor. Yet, in full generality, because morphisms of logoi are not required to preserve reflexive equalizers, this factorization will not be essentially unique and thus this orthogonal factorization system will not be available in the $2$-category of logoi, unless some restriction on the notion of morphism is made.

Using the notation in \Cref{rmk:modular-log}, we expect this factorisation system to be available in $\Log^\Phi$ with $\Phi$ the reflexive equalisers. 

\subsubsection{Pullbacks: back to the lemme de comparison}

Let $\cS \to \sT$ be a dense functor into a logos. In \Cref{lemmelemme}, we have discussed that (in full generality) there is no way to equip $\cS$ with a sketch structure so that $\sT$ is the logos of fake sheaves over that structure. This is because in general, the functor $\text{lan}_{\yo} F$, constructed in the discussion at the end of \Cref{lemmelemme}, may not preserve pullbacks. Thus there is no way to transform the orthogonality class into a sketch structure. Of course, when $\sT$ is a topos and $F$ is flat, the Kan extension in question will preserve pullbacks, and thus one obtains a satisfying version of the lemme de comparison. 

Using the notation in \Cref{rmk:modular-log}, we expect a version of the lemme de comparison to be available in $\Log^\Phi$ with $\Phi$ the pullbacks. 

\subsubsection{Finite products: what is an open morphism of logoi?}

In the introduction we have mentioned that one of the motivations for this paper is to generalize the work of \cite{pitts1983application,pitts1983amalgamation} to other fragments of logic. 
In order to do so though, we will still need some notion of \textit{open} morphism of logoi. For the case of topoi, we know that a geometric morphism is open when its inverse image $f^*$ preserves exponentials. 
Unfortunately, because in a logos $\sL$ the reflector $\mathcal{P}(\cL)\to\widehat{\cL}$ may not preserve finite products, there is no guarantee that $\cL$ will be cartesian closed (see \cite{day1972reflection}). 
Thus we do not have a natural way to generalise the notion of open geometric morphism to the context of logoi, unless we restrict to those logoi such that the reflector preserve finite products. 

Using the notation in \Cref{rmk:modular-log}, we expect a satisfying notion of open morphism to be available in $\Log^\Phi$ with $\Phi$ the finite products. 

\section*{Acknowledgements}
We are grateful to Giacomo Tendas for suggesting to investigate flexible co/limits.

\bibliography{thebib}

\begin{thebibliography}{DLLNS21}

\bibitem[ADLL23]{ADLL:aft-kz-mon}
Nathanael Arkor, Ivan Di~Liberti, and Fosco Loregian.
\newblock Adjoint functor theorems for lax-idempotent pseudomonads.
\newblock {\em arXiv preprint \url{https://arxiv.org/abs/2306.10389}}, 2023.

\bibitem[AJ21]{anel2021topo}
Mathieu Anel and Andr{\'e} Joyal.
\newblock Topo-logie.
\newblock {\em New Spaces in Mathematics: Formal and Conceptual Reflections},
  1:155--257, 2021.

\bibitem[ALR03]{adamek2003duality}
Jir{\i} Ad{\'a}mek, F~William Lawvere, and Jir{\i} Rosick{\`y}.
\newblock On the duality between varieties and algebraic theories.
\newblock {\em Algebra universalis}, 49(1):35--49, 2003.

\bibitem[AR94]{adamek_rosicky_1994}
Ji\v{r}\'{i} Ad{\'a}mek and Ji\v{r}\'{i} Rosick{\'y}.
\newblock {\em Locally Presentable and Accessible Categories}.
\newblock London Mathematical Society Lecture Note Series. Cambridge University
  Press, 1994.

\bibitem[AR20]{adamek2020nice}
Ji{\v{r}}{\'\i} Ad{\'a}mek and Ji{\v{r}}{\'\i} Rosick{\`y}.
\newblock How nice are free completions of categories?
\newblock {\em Topology and its Applications}, 273:106972, 2020.

\bibitem[Bar74]{barr1974toposes}
Michael Barr.
\newblock Toposes without points.
\newblock {\em Journal of Pure and Applied Algebra}, 5(3):265--280, 1974.

\bibitem[Ben97]{benson1997multilinearity}
David~B. Benson.
\newblock Multilinearity of sketches.
\newblock {\em Theory and Applications of Categories}, 3(11):269--277, 1997.

\bibitem[BKPS89]{BKPS:flex-lim}
{G. J.} Bird, {G. M.} Kelly, {A. J.} Power, and {R. H.} Street.
\newblock Flexible limits for 2-categories.
\newblock {\em Journal of Pure and Applied Algebra}, 61(1):1--27, 1989.

\bibitem[BL98]{butz1998regular}
Carsten Butz and BRICS Lecture~Series LS.
\newblock Regular categories and regular logic.
\newblock {\em BRICS Lecture Series LS-98-2}, 1998.

\bibitem[Bor94]{borceux_1994}
Francis Borceux.
\newblock {\em Handbook of Categorical Algebra}, volume~1 of {\em Encyclopedia
  of Mathematics and its Applications}.
\newblock Cambridge University Press, 1994.

\bibitem[Bra21]{brandenburg2021large}
Martin Brandenburg.
\newblock Large limit sketches and topological space objects.
\newblock {\em arXiv preprint arXiv:2106.11115}, 2021.

\bibitem[Buc14]{Buck:2-fib}
Mitchell Buckley.
\newblock Fibred 2-categories and bicategories.
\newblock {\em Journal of Pure and Applied Algebra}, 218(6):1034--1074, 2014.

\bibitem[Car19]{caramello2019denseness}
Olivia Caramello.
\newblock Denseness conditions, morphisms and equivalences of toposes.
\newblock {\em arXiv preprint arXiv:1906.08737}, 2019.

\bibitem[CDL21]{coraglia2021context}
Greta Coraglia and Ivan Di~Liberti.
\newblock Context, judgement, deduction.
\newblock {\em arXiv preprint arXiv:2111.09438}, 2021.

\bibitem[CLW93]{carboni1993introduction}
Aurelio Carboni, Stephen Lack, and Robert~FC Walters.
\newblock Introduction to extensive and distributive categories.
\newblock {\em Journal of Pure and Applied Algebra}, 84(2):145--158, 1993.

\bibitem[Day72]{day1972reflection}
Brian Day.
\newblock A reflection theorem for closed categories.
\newblock {\em Journal of pure and applied algebra}, 2(1):1--11, 1972.

\bibitem[DLL23]{di2023accessibility}
Ivan Di~Liberti and Fosco Loregian.
\newblock Accessibility and presentability in 2-categories.
\newblock {\em Journal of Pure and Applied Algebra}, 227(1):107155, 2023.

\bibitem[DLLNS21]{di2021functorial}
Ivan Di~Liberti, Fosco Loregian, Chad Nester, and Pawe{\l} Soboci{\'n}ski.
\newblock Functorial semantics for partial theories.
\newblock {\em Proceedings of the ACM on Programming Languages}, 5(POPL):1--28,
  2021.

\bibitem[DLLS22]{dlsolo}
Ivan Di~Liberti, Gabriele Lobbia, and Lurdes Sousa.
\newblock Kz-monads and kan injectivity.
\newblock {\em to appear in Theory and Applications of Categories special
  volume in honour of Marta Bunge}, 2022.
\newblock preprint available at \url{arXiv:2211.00380}.

\bibitem[DLO22]{di2022bi}
Ivan Di~Liberti and Axel Osmond.
\newblock Bi-accessible and bipresentable 2-categories.
\newblock {\em arXiv preprint arXiv:2203.07046}, 2022.

\bibitem[DLRG20]{di2020gabriel}
Ivan Di~Liberti and Julia Ramos~Gonz{\'a}lez.
\newblock Gabriel--ulmer duality for topoi and its relation with site
  presentations.
\newblock {\em Applied Categorical Structures}, 28(6):935--962, 2020.

\bibitem[EK23]{espindola2023every}
Christian Esp{\'\i}ndola and Krist{\'o}f Kanalas.
\newblock Every theory is eventually of presheaf type.
\newblock {\em arXiv preprint arXiv:2312.12356}, 2023.

\bibitem[Esp19]{Espindola2019infinitary}
Christian Esp{\'\i}ndola.
\newblock Infinitary first-order categorical logic.
\newblock {\em Annals of Pure and Applied Logic}, 170(2):137--162, 2019.

\bibitem[Esp20]{espindola2020infinitary}
Christian Esp{\'\i}ndola.
\newblock Infinitary generalizations of deligne’s completeness theorem.
\newblock {\em The Journal of Symbolic Logic}, 85(3):1147--1162, 2020.

\bibitem[Fre72]{Aspects}
P.~Freyd.
\newblock Aspects of topoi.
\newblock {\em Bulletin of Australian Mathematical Society}, 7:1--76, 1972.

\bibitem[{Fre}02]{FreydCartesianLogic}
Peter {Freyd}.
\newblock Cartesian logic.
\newblock {\em {Theor. Comput. Sci.}}, 278(1-2):3--21, 2002.

\bibitem[FS90]{freyd1990categories}
Peter~J. Freyd and Andre Scedrov.
\newblock {\em Categories, allegories}.
\newblock Elsevier, 1990.

\bibitem[Gar14]{Garner_top-funct}
Richard Garner.
\newblock Topological functors as total categories.
\newblock {\em Theory and Applications of Categories}, 29(15):406--421, 2014.

\bibitem[GL12]{garner2012lex}
Richard Garner and Stephen Lack.
\newblock Lex colimits.
\newblock {\em Journal of Pure and Applied Algebra}, 216(6):1372--1396, 2012.

\bibitem[Gra]{Gray_formal_ct_adj}
{\em Formal Category Theory: Adjointness for 2-Categories}.
\newblock Lecture Notes in Mathematics.

\bibitem[Her74]{HERRLICH1974125}
Horst Herrlich.
\newblock Topological functors.
\newblock {\em General Topology and its Applications}, 4(2):125--142, 1974.

\bibitem[Isb60]{isbell1960adequate}
John~R. Isbell.
\newblock Adequate subcategories.
\newblock {\em Illinois Journal of Mathematics}, 4(4):541--552, 1960.

\bibitem[Joh02a]{Sketches}
Peter~T. Johnstone.
\newblock {\em {Sketches of an elephant: a Topos theory compendium}}.
\newblock Oxford logic guides. Oxford Univ. Press, New York, NY, 2002.

\bibitem[Joh02b]{elephant2}
Peter~T. Johnstone.
\newblock {\em Sketches of an Elephant: A Topos Theory Compendium: 2 Volume
  Set}.
\newblock Oxford University Press UK, 2002.

\bibitem[Joh06]{johnstone2006syntactic}
{Peter T.} Johnstone.
\newblock A syntactic approach to diers' localizable categories.
\newblock In {\em Applications of Sheaves: Proceedings of the Research
  Symposium on Applications of Sheaf Theory to Logic, Algebra, and Analysis,
  Durham, July 9--21, 1977}, pages 466--478. Springer, 2006.

\bibitem[Kel80]{kelly1980unified}
G~Max Kelly.
\newblock A unified treatment of transfinite constructions for free algebras,
  free monoids, colimits, associated sheaves, and so on.
\newblock {\em Bulletin of the Australian Mathematical Society}, 22(1):1--83,
  1980.

\bibitem[Lac10]{Lack2010}
Stephen Lack.
\newblock {\em A 2-Categories Companion}, pages 105--191.
\newblock Springer New York, New York, NY, 2010.

\bibitem[Lai75]{lair1975etude}
Christian Lair.
\newblock {\em Etude g{\'e}n{\'e}rale de la cat{\'e}gorie des esquisses},
  volume~23.
\newblock na, 1975.

\bibitem[LT24]{lack2024accessible}
Stephen Lack and Giacomo Tendas.
\newblock Accessible categories with a class of limits.
\newblock {\em Journal of Pure and Applied Algebra}, 228(2):107444, 2024.

\bibitem[Mak97a]{makkai1997generalized}
Michael Makkai.
\newblock Generalized sketches as a framework for completeness theorems. {P}art
  {I}.
\newblock {\em Journal of Pure and Applied Algebra}, 115(1):49--79, 1997.

\bibitem[Mak97b]{makkai1997generalizedII}
Michael Makkai.
\newblock Generalized sketches as a framework for completeness theorems. part
  ii.
\newblock {\em Journal of Pure and Applied Algebra}, 115(3):179--212, 1997.

\bibitem[Mak97c]{makkai1997generalizedIII}
Michael Makkai.
\newblock Generalized sketches as a framework for completeness theorems. {P}art
  {III}.
\newblock {\em Journal of Pure and Applied Algebra}, 115(3):241--274, 1997.

\bibitem[MP89]{Makkaipare}
Michael Makkai and Robert Par{\'{e}}.
\newblock {\em Accessible Categories: The Foundations of Categorical Model
  Theory}.
\newblock American Mathematical Society, 1989.

\bibitem[MV00]{moerdijk2000proper}
Ieke Moerdijk and Jacob Johan~Caspar Vermeulen.
\newblock {\em Proper maps of toposes}, volume 705.
\newblock American Mathematical Soc., 2000.

\bibitem[Pit83a]{pitts1983amalgamation}
Andrew~M. Pitts.
\newblock Amalgamation and interpolation in the category of heyting algebras.
\newblock {\em Journal of Pure and Applied Algebra}, 29(2):155--165, 1983.

\bibitem[Pit83b]{pitts1983application}
Andrew~M. Pitts.
\newblock An application of open maps to categorical logic.
\newblock {\em Journal of pure and applied algebra}, 29(3):313--326, 1983.

\bibitem[RG18]{ramos2018grothendieck}
Julia Ramos~Gonz{\'a}lez.
\newblock Grothendieck categories as a bilocalization of linear sites.
\newblock {\em Applied Categorical Structures}, 26(4):717--745, 2018.

\bibitem[Shu16]{shulman2016contravariance}
Michael Shulman.
\newblock Contravariance through enrichment.
\newblock {\em arXiv preprint arXiv:1606.05058}, 2016.

\bibitem[Str74]{10.1007/BFb0063102}
Ross Street.
\newblock Fibrations and yoneda's lemma in a 2-category.
\newblock In Gregory~M. Kelly, editor, {\em Category Seminar}, pages 104--133,
  Berlin, Heidelberg, 1974. Springer Berlin Heidelberg.

\bibitem[Str86]{street1986categories}
Ross Street.
\newblock Categories in which all strong generators are dense.
\newblock {\em Journal of Pure and Applied Algebra}, 43:235--242, 1986.

\bibitem[Wel94]{wells1994sketches}
Charles Wells.
\newblock Sketches: Outline with references.
\newblock {\em Dept. of Computer Science, Katholieke Universiteit Leuven
  Celestijnenlaan A}, 200, 1994.

\bibitem[Zaw95]{zawadowski1995descent}
Marek~W. Zawadowski.
\newblock Descent and duality.
\newblock {\em Annals of Pure and Applied Logic}, 71(2):131--188, 1995.

\end{thebibliography}
\bibliographystyle{alpha}

\appendix
\newpage
\section{Names and Symbols}\label{app:notation}

%\red GAB (regarding diagrams): I used $I$ for the cats and $D$ for the functor in \Cref{notat:weight-co/proj}\black

\begin{table}[!h]\renewcommand{\arraystretch}{1.3}
\begin{tabular}{|c|c|c|}
\hline
\textbf{Symbol }            & \textbf{Meaning} & \textbf{Reference  }                          \\ \hline
$\sS$ & a sketch       &           \Cref{skt}          \\ \hline
$\cS$ & the underlying category of a sketch & \Cref{skt} \\ \hline
$\mathsf{L}, \mathsf{C}$ & a class of (co/)cones      &      \Cref{skt}    \\ \hline
$D\colon\cI\to\skt$ & a diagram in $\skt$ & \Cref{notat:weight-co/proj} \\ \hline
$\sS^\sD$ & exponential sketch & \Cref{costr:exp-skt} \\ \hline
$\sS \boxtimes \sD$ & Benson's tensor product & \Cref{def:benson-tensor} \\ \hline
$(C,J)$ & a site &  \Cref{sites} \\ \hline
$\hat{\sS}$ & Left sketch classifier of $\sS$ & \Cref{costr:left-skt-class} \\ \hline
$\tilde{\sS}$ & Left normalization of $\sS$  & \Cref{costr:left-normalisation} \\ \hline
$U$ & some forgetful functor &  \\ \hline
\end{tabular}
\end{table}

%%%%%%%% Se le FOOTNOTE sono nella pagina prima mettere \newpage %%%%%%%%%%%%%%%%%%%
%\newpage
\begin{table}[!h]\renewcommand{\arraystretch}{1.25}
\begin{tabular}{|c|c|c|}
\hline
\textbf{Symbol}             & \textbf{Meaning} & \textbf{Reference}                            \\ \hline
$\skt$ & small sketches       &               \Cref{sktcats}      \\ \hline
$\Skt$ & large and locally small sketches\tablefootnote{Upper cases are always  used to jump from small to large and locally small}   & \Cref{sktcats} \\ \hline
$\skt_l/\skt_c$ & (respectively) limit and colimit sketches       &               \Cref{sktcats}      \\ \hline
$\Skt^{\mathsf{M}}$ & Morita small sketches\tablefootnote{The apex $(-)^{\mathsf{M}}$ is always used to carve out Morita small objects.}   & \Cref{Moritasmall} \\  \hline
$\mathsf{L}\Skt$ & Left sketches   &\Cref{def:left-sketch}  \\   \hline
$\leftnskt$ & Left normal sketches   & \Cref{skt}  \\ \hline
$r\skt$ & rounded sketches   & \Cref{def:rounded-skt} \\ \hline
$\lex$ & small categories with finite limits       &   \Cref{doctrines}(b)                \\ \hline
$\sites$ & small sites      &         \Cref{sites}             \\ \hline
$\sites^{\lex}$ & small sites with underlying lex category\tablefootnote{The apex $(-)^{\lex}$ is always used to carve out lex objects.}      & \Cref{sites}                      \\ \hline
\end{tabular}
\end{table}

% % https://q.uiver.app/#q=WzAsMTEsWzMsMiwibGVmdCJdLFsxLDIsIm5vcm1hbCJdLFsyLDMsImxlZnQgXFwgIG5vcm1hbCJdLFsyLDEsInJvdW5kZWQiXSxbMiwwLCJsb2dvcyJdLFswLDAsInNtYWxsIl0sWzAsMiwid2Vha2x5IFxcIE1vcml0YSBcXCAgc21hbGwiXSxbMCwzLCJsb2NhbGx5IFxcICBzbWFsbCJdLFsyLDIsInRlc3QiXSxbMSwwLCJjb2xpbWl0Il0sWzAsMSwibGVmdFxcIE1vcml0YSBcXCAgc21hbGwiXSxbMSwyLCIiLDAseyJsZXZlbCI6Mn1dLFs0LDAsIiIsMCx7ImxldmVsIjoyfV0sWzYsNywiIiwwLHsibGV2ZWwiOjJ9XSxbNCwzLCIiLDIseyJsZXZlbCI6Mn1dLFswLDIsIiIsMCx7ImxldmVsIjoyfV0sWzgsMSwiIiwwLHsibGV2ZWwiOjJ9XSxbOCwwLCIiLDAseyJsZXZlbCI6Mn1dLFs0LDEsIiIsMCx7ImxldmVsIjoyfV0sWzUsMTAsIiIsMix7ImxldmVsIjoyfV0sWzEwLDYsIiIsMix7ImxldmVsIjoyfV1d
% \[\begin{tikzcd}[ampersand replacement=\&]
% 	small \& colimit \& logos \\
% 	{left\ Morita \  small} \&\& rounded \\
% 	{weakly \ Morita \  small} \& normal \& test \& left \\
% 	{locally \  small} \&\& {left \  normal}
% 	\arrow[Rightarrow, from=3-2, to=4-3]
% 	\arrow[Rightarrow, from=1-3, to=3-4]
% 	\arrow[Rightarrow, from=3-1, to=4-1]
% 	\arrow[Rightarrow, from=1-3, to=2-3]
% 	\arrow[Rightarrow, from=3-4, to=4-3]
% 	\arrow[Rightarrow, from=3-3, to=3-2]
% 	\arrow[Rightarrow, from=3-3, to=3-4]
% 	\arrow[Rightarrow, from=1-3, to=3-2]
% 	\arrow[Rightarrow, from=1-1, to=2-1]
% 	\arrow[Rightarrow, from=2-1, to=3-1]
% \end{tikzcd}\]

% \section{Morita Small Sketch have pseudo-co/limits}

% \red The constrcutions in \Cref{sec:w-p-co/lim} apply to $\Skt^{\mathsf{M}}$ as well with a little bit more. \black

\vfill

\end{document}